\documentclass[12pt,preprint]{amsart}

\usepackage{amsmath, latexsym, amsfonts, amssymb, amsthm}
\usepackage{graphicx, color,hyperref,dsfont,epsfig,caption,wrapfig,subfig}
\usepackage{mathtools}
\usepackage{movie15}
\hypersetup{
	linktoc=page,
	linkcolor=red,          % color of internal links
	citecolor=blue,        % color of links to bibliography
	filecolor=blue,      % color of file links
	urlcolor=cyan,
	colorlinks=true           % color of external links
}

\numberwithin{equation}{section}

\newtheorem{theorem}{Theorem}[section]
\newtheorem{lemma}[theorem]{Lemma}
\newtheorem{proposition}[theorem]{Proposition}
\newtheorem{corollary}[theorem]{Corollary}

\newtheorem{remark}[theorem]{Remark}

\newcounter{thmc}

\theoremstyle{definition}

\renewcommand{\tilde}{\widetilde}          % wider `tilde'
\DeclareMathSymbol{\leqslant}{\mathalpha}{AMSa}{"36} % nicer `smaller or equal'
\DeclareMathSymbol{\geqslant}{\mathalpha}{AMSa}{"3E} % nicer `larger or equal'
\DeclareMathSymbol{\eset}{\mathalpha}{AMSb}{"3F}     % nicer `emptyset'
\renewcommand{\leq}{\;\leqslant\;}                   % redef. of < or =
\renewcommand{\geq}{\;\geqslant\;}                   % redef. of > or =
\newcommand{\C}{\mathbb{C}}
\newcommand{\R}{\mathbb{R}}

\newcommand{\N}{\mathbb{N}}
 
\newcommand{\D}{\mathbb{D}} 
\newcommand{\Heps}{\mathbb{H}_{\delta,\eps}} 
\newcommand{\Reps}{\mathbb{R}_{\eps}} 
\renewcommand{\H}{\mathbb{H}}

\newcommand{\E}{\mathds{E}}
\newcommand{\X}{\bm{\mathrm X}}

\newcommand{\V}{\bm{\mathrm V}}

\newcommand{\ps}[1]{\langle #1 \rangle}
\newcommand{\mc}[1]{\mathcal{#1}}

\def\V{\bm{\mathrm V}}

\def\SET{\bm{\mathrm T}}

\def\Lc{\bm{\mathcal L}}
\def\X{\bm{\mathrm X}}

\def\L{\bm{\mathrm L}}

\def\eps{\varepsilon}

\def\eps{\varepsilon}

\def\bi{\begin{itemize}}
	\def\ei{\end{itemize}}

\def\bnum{\begin{enumerate}}
	\def\enum{\end{enumerate}}
\newcommand{\hyper}[3]{ {}_2F_1\left(\begin{matrix}#1 \\ #2\end{matrix} \middle\vert #3 \right)}
\def\<#1{\langle #1 \rangle}

\newcommand{\norm}[1]{\left\lvert#1\right\rvert}
\newcommand{\expect}[1]{\mathbb{E}\left[#1\right]}
\usepackage{bm}
\usepackage{bbm}

\title[HEM for boundary Liouville CFT]{Higher equations of motion for boundary Liouville Conformal Field Theory from the Ward identities}

\author{Baptiste Cercl\'e}
\email{baptiste.cercle@epfl.ch}
\address{EPFL SB MATH RGM, MA B2 397, Station 8, CH-1015 Lausanne, Switzerland.}

\begin{document}

	\maketitle
	\begin{abstract}
		In this document we prove higher equations of motion at the level $2$ for boundary Liouville Conformal Field Theory. As a corollary we present a new derivation of the Belavin-Polyakov-Zamolodchikov differential equations. Our method of proof does not rely on the mating of trees machinery but rather exploits the symmetries of the model through the Ward identities it satisfies. 
		
		To do so we provide a definition of derivatives of the correlation functions with respect to a boundary insertion which was lacking in the existing literature, and introduce a new notion of descendant fields related to these Ward identities.   
	\end{abstract}
	
	%%%%%%%%%%%%%%%%%%%%%%%
	%%%%%%%%%%%%%%%%%%%%%%%
	%%%%%%%%%%%%%%%%%%%%%%%
	
	\tableofcontents
	
	%%%%%%%%%%%%%%%%%%%%%%%%%%%%%%%%%%%
	%%%%%%%%%%%%%%%%%%%%%%%%%%%%%%%%%%%
	%%%%%%%%%%%%%%%%%%%%%%%%%%%%%%%%%%%
	
	%%%%%%%%%%%%%%%%%%%%%%%%%%%%%%%%%%%
	%%%%%%%%%%%%%%%%%%%%%%%%%%%%%%%%%%%
	%%%%%%%%%%%%%%%%%%%%%%%%%%%%%%%%%%%	
	
	\section{Introduction}
	
	%%%%%%%%%%%%%%%%%%%%%%%%%%%%%%%%%%%%%%%%%%%%%%%%
	%%%%%%%%%%%%%%%%%%%%%%%%%%%%%%%%%%%%%%%%%%%%%%%%
	
	%%%%%%%%%%%%%%%%%%%%%%%%%%%%%%%%%%%%%%%%%%%%%%%%
	%%%%%%%%%%%%%%%%%%%%%%%%%%%%%%%%%%%%%%%%%%%%%%%%
	
	\subsection{The setting: boundary Liouville theory}
	The study of Liouville Conformal Field Theory (CFT hereafter) has now become a topic of key importance in both the mathematics and physics community. Initially introduced as a model for random two-dimensional geometry~\cite{Pol81}, it has since emerged as a fundamental theory in a wide scope of topics, ranging from string theory and quantum gravity to statistical physics at criticality. This is all the more true thanks to the conformal bootstrap method envisioned by Belavin-Polyakov-Zamolodchikov~\cite{BPZ} (BPZ in the sequel) and that makes its study somehow universal in the setting of two-dimensional CFT.
	
	\subsubsection{Liouville Conformal Field Theory}
	Liouville CFT is defined in the physics literature using a path integral approach, in the sense that one defines the law of a random function $\Phi$ via expressions of the form
	\begin{equation}\label{eq:path_integral}
		\langle F(\Phi) \rangle \coloneqq \frac1{\mc Z}\int_{\varphi\in\mc F} F( \varphi)e^{-S_L(\varphi)}D \varphi
	\end{equation}
	where $D \varphi$ plays the role of a \lq\lq Lebesgue measure" over some functional space $\mc F$ of maps $\varphi$ taking values on a Riemann surface $\Sigma$. The functional $S_L$ is given by the Liouville action, which in the case of a surface with boundary $\partial\Sigma$ (which may be empty) is equal to
	\begin{equation}\label{eq:Toda_action}
		S_{L}(\varphi)\coloneqq  \frac{1}{4\pi} \int_{\Sigma}  \Big (  \norm{\partial_g\varphi}^2   +Q R_g\varphi +4\pi \mu e^{\gamma    \varphi}   \Big)\,{\rm dv}_{g}+\frac{1}{2\pi} \int_{\partial\Sigma}  \Big (Q K_g\varphi +2\pi \mu_\partial e^{\frac\gamma2    \varphi}   \Big)\,{\rm dl}_{g}.
	\end{equation}
	Here $g$ is a Riemannian metric over $\Sigma$ with associated scalar (resp. geodesic) curvature $R_g$ (resp. $K_g$), gradient $\partial_g$ and volume form (resp. line element) ${\rm v}_g$ (resp. ${\rm l}_g$). The action also depends on a coupling constant $\gamma\in(0,2)$ and the background charge $Q\coloneqq \frac\gamma2+\frac2\gamma$, as well as cosmological constants $\mu\geq0$ and $\mu_\partial$, the latter being a piecewise constant, complex valued function on $\partial\Sigma$. 
	
	There are functionals of the random function $\Phi$ that are of special interest: the Vertex Operators. Computing their correlation functions is one of the main goals in the study of Liouville CFT. These observables depend on an insertion point in $\overline\Sigma\coloneqq\Sigma\cup\partial\Sigma$ and a weight in $\R$, and are of the form $V_{\alpha}(z)=e^{\alpha\Phi(z)}$ for $z\in\Sigma$ or $V_{\beta}(s)=e^{\frac\beta2\Phi(s)}$ if $s\in\partial\Sigma$. Their correlation functions are formally defined using the path integral by setting:
	\begin{equation}\label{eq:formal_correl}
		\langle \prod_{k=1}^NV_{\alpha_k}(z_k)\prod_{l=1}^MV_{\beta_l}(s_l) \rangle \coloneqq \int_{\varphi\in\mc F} \prod_{k=1}^Ne^{\alpha_k\varphi(z_k)}\prod_{l=1}^M e^{\frac{\beta_l}2\varphi(s_l)}e^{-S_L( \varphi)}D \varphi
	\end{equation}
	where $z_1,\cdots,z_N$ belong to $\Sigma$ while $s_1,\cdots,s_M$ are in $\partial\Sigma$.
	
	In the rest of this document we will consider the case where $\Sigma$ is given by the upper-half plane $\mathbb{H}$; its boundary is the real line $\R$. Indeed the structure constants for boundary Liouville CFT are defined in this setting.
	
	\subsubsection{Structure constants and BPZ differential equations}
	The basic correlation functions of the theory are the \textit{structure constants}, which for the boundary Liouville CFT are: the boundary three-point structure constant, which is a correlation function that contains three boundary Vertex Operators, and the bulk-boundary structure constant that contains one bulk and one boundary Vertex Operators. These structure constants have been computed in the physics literature and formulas for them have been proposed respectively in~\cite{PT02} and~\cite{Hos}.
	
	In order to compute these structure constants, physicists rely on so-called \textit{BPZ differential equations} that arise for correlation functions containing a degenerate Vertex Operator. When the correlation functions considered are derived from the structure constants solutions of these equations turn out to be expressible in terms of hypergeometric functions. Combining this fact with formal Taylor expansions of products of Vertex Operators (the \textit{Operator Product Expansions}) they are able to obtain \textit{shift equations} that characterize the structure constants. 
	
	Once the structure constants have been derived, it is predicted in the physics that all the correlation functions can be computed recursively based on the conformal bootstrap for boundary CFT and therefore that their knowledge can be reduced to that of the structure constants and universal objects called \textit{conformal blocks.}  
	More details on the techniques developed in the physics literature can be found \textit{e.g.} in~\cite{Nakayama} where a review of boundary Liouville CFT is proposed. 
	
	Besides the key role of these BPZ differential equations in the derivation of the structure constants of Liouville CFT they are of some interest in themselves since they are particular cases of so-called \textit{higher equations of motion}, which are predicted to arise in the presence of singular vectors associated to elements of the Kac table. Such equations of motion arise for instance in minimal gravity, and more specifically boundary minimal Liouville gravity in that they allow to define certain correlation functions in this context~\cite{BB10}. See for instance~\cite{Za04} where these higher equations of motion are considered in the case of Liouville CFT without boundary and~\cite{BB10} where the boundary ones, which are at the heart of the present document, are discussed.
	
	\subsection{Probabilistic approaches to boundary Liouville conformal field theory}
	In this context it is a natural question as to whether it is possible to study Liouville CFT from a mathematically rigorous perspective, and if so, how. Several approaches have been designed in order to address this issue, that differ in their means and motivations. 
	
	A first approach relies on the fundamental connection between Liouville theory and models arising in different branches of probability theory: the mathematical take on Liouville Quantum Gravity (LQG in the sequel), initiated in~\cite{DMS14}, emerges in this setting. This is exemplified for instance via the numerous connections between scaling limits of random planar maps~\cite{LeG13, Mie13} and LQG as described in~\cite{MS15a, MS16a, MS16b, GMS21, HoSu19}, connections that have become stronger thanks to the construction of the Liouville Quantum Gravity metric~\cite{DDDF,DFGPS, GM20,DDG}. Likewise there exist strong links between LQG, Schramm-Loewner Evolutions (SLEs) and Conformal Loop Ensembles (CLEs)~\cite{Sh_CLE, SW, MSW}, all the more thanks to the procedure of \textit{conformal welding}~\cite{She16, DMS14, AHS20}.
	
	Another approach to make sense of Liouville CFT in a mathematical language has been introduced by David-Guillarmou-Kupiainen-Rhodes-Vargas in their work~\cite{DKRV}, where the path integral definition of this theory was given a rigorous meaning. The tools then employed~\cite{KRV_loc} to understand the theory are closer to the method envisioned by physicists, which led to major breakthroughs in the mathematical comprehension of this model, should it be via the probabilistic derivation of the DOZZ formula~\cite{DO94, ZZ96} describing the structure constants of the theory on the sphere in~\cite{KRV_DOZZ}, the mathematical justification of the conformal bootstrap procedure for computing four-point correlation functions based on these structure constants~\cite{GKRV} and of Segal's axioms~\cite{Seg04, GKRV_Segal} for computing recursively all correlation functions. It is natural to wonder whether the same approach would remain valid for the boundary Liouville CFT. The theory has been defined in~\cite{HRV16, Wu} and integrability results provided in subsequent works~\cite{remy1,remy2} based on the techniques inspired by the physics literature. However the structure constants were derived in the special case where the bulk cosmological constant is zero and the conformal bootstrap method is yet to be implemented.
	
	These two approaches designed to make sense of Liouville theory are far from belonging to two distinct worlds. Indeed there are strong connections between these two perspectives and in some cases the objects considered are actually equivalent~\cite{equivalence1,equivalence2}. These links have been later extended~\cite{AHS20} and led to numerous integrability results for SLEs~\cite{AHS21} and CLEs~\cite{AS_DOZZ}. 
	
	\subsubsection{Mathematical derivation of higher equations of motion and BPZ differential equations}
	These connections are at the heart of a recent program aimed at computing all the structure constants of boundary Liouville CFT. Based on the interplays between SLEs and Liouville theory the FZZ formula~\cite{FZZ} was proved to hold in~\cite{ARS}, and all the structure constants of boundary Liouville CFT were finally computed in~\cite{ARSZ}. To do so the authors rely on tools coming from the mating-of-trees theory, and one of the key elements of the proof is a BPZ differential equation derived in~\cite{Ang_zipper} based on the quantum zipper. Before this achievement one of the main obstructions preventing the derivation of all structure constants of the theory was the lack of a proof of such BPZ differential equations, for which a proof based solely on the framework of~\cite{HRV16} remained unknown. 
	
	A second approach to this problem has been recently proposed in~\cite{BaWu} where the authors state higher equations of motion for Liouville theory. To do so the Vertex Operators are thought of as elements of the Hilbert space of the theory and can be studied via properties of the Hamiltonian associated to boundary Liouville theory, viewed as an operator acting on this Hilbert space.
	The singular vectors described in the Kac table are then shown to correspond to poles of the Poisson operator associated to the Hamiltonian, for which the residues are computed. The equations of motion are thus stated as an equality at the level of the Hilbert space of the theory, and one would need to rely on the conformal bootstrap for boundary Liouville CFT to be able to express this equality as an actual differential equation for the correlation functions.

	\subsubsection{Contribution of the present paper}
	This raises the following questions: is it possible to compute the structure constants of boundary Liouville CFT without the mating-of-trees machinery or the boundary conformal bootstrap, and by exclusively relying on the probabilistic framework of~\cite{HRV16}? Can one study boundary Liouville CFT in a completely intrinsic way by making sense of the techniques employed in the physics literature and to what extent are the tools coming from the mating-of-trees approach necessary in this perspective?
	
	We partially answer this question in this document by proving that the BPZ differential equations, at the heart of the derivation of these structure constants in~\cite{ARSZ}, can be proved without relying on tools coming from the mating-of-trees approach. But in addition to these BPZ differential equations we more generally prove that \textit{higher equations of motion} are valid. These were predicted to hold in the physics literature~\cite{BB10} based on the conformal bootstrap method and in particular the explicit expressions of the structure constants of the boundary Liouville theory. Our method of proof is completely different from this approach as it relies on the probabilistic definition of boundary Liouville theory and the study of its symmetries.  
	To be more specific our approach is based on a rigorous derivation of the Ward identities together with a definition of the descendant fields associated to Vertex Operators and does not rely on the conformal bootstrap procedure. This approach allows in particular to define derivatives of the correlation functions.
	
	Our main results in this perspective are the following:
	\begin{enumerate}
		\item we answer a question raised in~\cite{Ang_zipper} as to whether one can justify that correlation functions are differentiable. Indeed we explain here how one can define the derivatives of the correlation functions with respect to a boundary insertion; 
		\item we make sense of the heuristic considered in the physics by proving that Ward identities are related to descendant fields of the Vertex Operators;
		\item we show that equations of motion, proved at the Hilbert space level in~\cite{BaWu}, are actually valid for the correlation functions of the theory;
		\item we recover one of the main results of~\cite{Ang_zipper} as a particular case. This provides an intrinsic justification of the fact that the two boundary cosmological constants around the boundary degenerate insertion must be related in order to have BPZ differential equations. Note that this assumption is removed when $\alpha=-\frac\gamma2$ with $\gamma>\sqrt2$ due to a \textit{freezing phenomenon}.
	\end{enumerate}
	
	\subsection{Method of proof and main results}
	Let us now sketch the method developed in this document and more precisely describe how (weak) derivatives of the correlation functions as well as descendant fields can be defined. We will then state our main results based on these definitions, that is we will explain how this framework allows to rigorously derive Ward identities and higher equations of motion. 
	
	\subsubsection{Definition of the correlation functions}
	The starting point is the probabilistic interpretation of the correlation functions formally defined by the path integral~\eqref{eq:formal_correl} ---more details can be found in Section~\ref{sec:def}. For this purpose we first introduce the \textit{Liouville field} $\Phi$ together with the correlation functions by setting for suitable functionals $F$, weights $(\alpha_1,\cdots,\alpha_{N},\beta_1,\cdots,\beta_M)\in \R^{N+M}$ and insertions $(z_1,\cdots,z_N,s_1,\cdots,s_M)\in\H^N\times\R^M$
	\begin{equation}
		\begin{split}
			&\ps{F(\Phi)\prod_{k=1}^NV_{\alpha_k}(z_k)\prod_{l=1}^MV_{\beta_l}(s_l)}_{\delta,\eps,\rho}\coloneqq \\
			&P(\bm z,\bm \alpha)\int_{\R} e^{\bm s\bm c}\expect{F\left(\X_\rho-2Q\ln \norm{\cdot}_++ H+\bm c\right)\exp\left(-\mu e^{\gamma\bm c}\mc A_{\delta,\eps,\rho}-e^{\frac\gamma2 \bm c}\mc L_{\delta,\eps,\rho}\right)}d\bm c.
		\end{split}
	\end{equation}
	Here $\X_\rho$ is a regularization of a Gaussian Free Field $\X$ over $\overline\H$, while for positive $\eps$ and $\delta$ we have introduced the notations
	\begin{equation}
		\begin{split}
			&\mc A_{\delta,\eps,\rho}\coloneqq\int_{\Heps}\prod_{k=1}^{2N+M}\left(\frac{\norm{x}_+}{\norm{z_k-x}}\right)^{\gamma\alpha_k}\rho^{\frac{\gamma^2}2}e^{\gamma\X_\rho(x)}\norm{dx}^2,\\
			&\mc L_{\delta,\eps,\rho}\coloneqq\int_{\R_{\eps}}\prod_{k=1}^{2N+M}\left(\frac{\norm{x}_+}{\norm{z_k-x}}\right)^{\frac{\gamma\alpha_k }2}\rho^{\frac{\gamma^2}{4}}e^{\frac\gamma2 \X_\rho(x)}\mu_{\partial}(dx),\\
			&P(\bm z,\bm\alpha)\coloneqq \prod_{k\neq l}\norm{z_k-z_l}^{-\alpha_k\alpha_l}\prod_{k=1}^N\norm{z_k-\bar z_k}^{\frac{\alpha_k^2}{2}},\quad\text{and}\quad\bm s\coloneqq\sum_{k=1}^{N}\alpha_k+\sum_{l=1}^M\frac{\beta_l}2-Q.
		\end{split}
	\end{equation}
	We have also used the shorthands $(\alpha_1,\cdots,\alpha_{2N+M})\coloneqq(\alpha_1,\cdots,\alpha_N,\alpha_1,\cdots,\alpha_N,\beta_1,\cdots,\beta_M)$ and $(z_1,\cdots,z_{2N+M})\coloneqq(z_1,\cdots,z_N,\bar z_1,\cdots,\bar z_N,s_1,\cdots,s_M)$, and considered $\norm{\cdot}_+=\max(1,\norm{\cdot})$.
	Eventually the domains of integration avoid the singularities in that they are given by\\
	$\Heps\coloneqq \left(\H+i\delta\right)\setminus\left(\bigcup_{k=1}^{2N+M}B(z_k,\eps)\right)$ and $\R_{\eps}\coloneqq \R\setminus\cup_{l=1}^M(s_l-\eps,s_l+\eps)$. 
	
	We will see in Subsection~\ref{subsec:correl} that under the following assumptions:
	\begin{equation}\label{eq:seiberg_intro}
		\begin{split}
			&\alpha_k,\beta_l<Q\text{ for all }k,l;\quad\sum_{k=1}^N\alpha_k+\sum_{l=1}^M\frac{\beta_l}2-Q>0\quad\text{(Seiberg bounds)};\\
			&\mu>0\quad\text{and}\quad\Re(\mu_l)\geq0\quad\text{for all }1\leq l\leq M
		\end{split}
	\end{equation}
	the limit as $\rho$, $\eps$ and then $\delta$ go to $0$ of the correlation functions are well-defined and non-trivial (\textit{i.e.} neither zero or infinite). We will then extend the range of validity for which we can define the correlation functions (Proposition~\ref{prop:analycity}), the corresponding set for admissible weights being denoted by $\mc A_{N,M}$. 
	
	Hereafter we will assume that the assumptions made on the cosmological constants in Equation~\eqref{eq:seiberg_intro} are satisfied. 
	We will also use the shorthand $\V=\prod_{k=1}^NV_{\alpha_k}(z_k)\prod_{l=1}^MV_{\beta_l}(s_l)$.
	%%%%%%%%%%%%%%%%%%%%%%%%%%%%%%%%%%%%%%%%%%%%%%%%
	
	%%%%%%%%%%%%%%%%%%%%%%%%%%%%%%%%%%%%%%%%%%%%%%%%
	%%%%%%%%%%%%%%%%%%%%%%%%%%%%%%%%%%%%%%%%%%%%%%%%
	
	\subsubsection{Derivatives of the correlation functions}
	Our first task is to give a justification that the correlation functions are differentiable, at least in the weak sense. This property seems to be absent in the existing literature when the insertion belongs to the boundary of the domain. In this perspective we first provide in Subsection~\ref{subsec:der1} a definition of these derivatives via a limiting procedure:
	\begin{lemma}\label{lemma:desc_Lintro}
		Assume that the weights satisfy $\bm\alpha\in \mc A_{N,M}$. Then as $\rho$, $\eps$ and then $\delta\to0$, the following limit exists and is finite:
		\begin{equation}
			\begin{split}
				&\ps{\L_{-1}V_{\beta_1}(s_1)\prod_{k=1}^NV_{\alpha_k}(z_k)\prod_{l=2}^MV_{\beta_l}(s_l)}\coloneqq\lim\limits_{\delta,\eps,\rho\to0}\partial_{s_1}\ps{\prod_{k=1}^NV_{\alpha_k}(z_k)\prod_{l=1}^MV_{\beta_l}(s_l)}_{\delta,\eps,\rho}\\
				&-\Big(\ps{V_{\gamma}(s_1-\eps)\V}_{\delta,\eps,\rho}\mu_{\partial}(s_1-\eps)-\ps{V_{\gamma}(s_1+\eps)\V}_{\delta,\eps,\rho}\mu_\partial(s_1+\eps)\Big).
			\end{split}
		\end{equation}
		Moreover we have the equality in the sense of weak derivatives :
		\begin{equation}\label{eq:der1intro}
			\ps{\L_{-1}V_\beta(s_1)\prod_{k=1}^NV_{\alpha_k}(z_k)\prod_{l=2}^MV_{\beta_l}(s_l)}=\partial_{s_1}\ps{\prod_{k=1}^NV_{\alpha_k}(z_k)\prod_{l=1}^MV_{\beta_l}(s_l)}.
		\end{equation}
	\end{lemma}
	
	Second order derivatives are defined in Subsection~\ref{subsec:der2} by a limiting procedure of the form
	\begin{equation}
		\begin{split}
			&\ps{\L_{-(1,1)}V_{\beta_1}(s_1)\prod_{k=1}^NV_{\alpha_k}(z_k)\prod_{l=2}^MV_{\beta_l}(s_l)}\\
			&\coloneqq\lim\limits_{\delta,\eps,\rho\to0}\partial_{s_1}^2\ps{\prod_{k=1}^NV_{\alpha_k}(z_k)\prod_{l=1}^MV_{\beta_l}(s_l)}_{\delta,\eps,\rho}-\mathfrak{R}_{-(1,1)}(\delta,\eps,\rho,\bm\alpha)
		\end{split}
	\end{equation}
	where $\mathfrak{R}_{-(1,1)}(\delta,\eps,\bm\alpha)$ is a (rather complicated) remainder term. The object thus defined satisfies, again in the weak sense,
	\begin{equation}\label{eq:L11_intro}
		\ps{\L_{-(1,1)}V_{\beta_1}(s_1)\prod_{k=1}^NV_{\alpha_k}(z_k)\prod_{l=2}^MV_{\beta_l}(s_l)}=\partial_{s_1}^2\ps{\prod_{k=1}^NV_{\alpha_k}(z_k)\prod_{l=1}^MV_{\beta_l}(s_l)}.
	\end{equation}

	\subsubsection{Ward identities and descendant fields}	
	Ward identities naturally arise when the stress-energy tensor $\SET$ is inserted within correlation functions. This tensor is the observable that is formally defined from the field $\Phi$ for $z\in\overline\H$ by
	\begin{equation}
		\SET(z)[\Phi]\coloneqq Q\partial^2\Phi(z)-\left(\partial\Phi(z)\right)^2,
	\end{equation}	
	and somehow encodes the conformal covariance of the model. 
	Now for $\beta<Q$ and $t$ in $\partial\H$, let us consider the observable inspired from the stress-energy tensor
	\begin{equation}
		\begin{split}
			F\left[\Phi\right]&\coloneqq \left(Q+\beta\right)\partial^2\Phi(t)-\left(\partial\Phi(t)\right)^2.
		\end{split}
	\end{equation}
	Using the regularization of the Liouville field $\Phi_\rho$ we can make sense of these derivatives at the regularized level: we prove in Subsection~\ref{subsec:ward} that we can make sense of the limit
	\[
	\ps{\L_{-2}V_{\beta}(t)\prod_{k=1}^NV_{\alpha_k}(z_k)\prod_{l=1}^MV_{\beta_l}(s_l)}\coloneqq\lim\limits_{\delta,\eps,\rho\to0}\ps{F[\Phi]V_{\beta}(t)\prod_{k=1}^NV_{\alpha_k}(z_k)\prod_{l=1}^MV_{\beta_l}(s_l)}_{\delta,\eps,\rho}-\mathfrak{R}_{-2}(\delta,\eps,\rho,\bm\alpha)
	\]
	where $\mathfrak{R}_{-2}(\delta,\eps,\rho,\bm\alpha)$ is an explicit remainder term. We will then prove a local Ward identity (Theorem~\ref{thm:desc_two_set}) in the form of the following statement:
	\begin{theorem}\label{thm:desc_two_setintro}
		Assume that $\bm\alpha\in\mc A_{N,M}$. Then in the weak sense:
		\begin{equation}\label{eq:L2_intro}
			\begin{split}
				\ps{\L_{-2}V_{\beta}(t)\prod_{k=1}^NV_{\alpha_k}(z_k)\prod_{l=1}^MV_{\beta_l}(s_l)}&\\
				=\sum_{k=1}^{2N+M}\left(\frac{\partial_{z_k}}{t-z_k}
				+\frac{\Delta_{\alpha_k}}{(t-z_k)^2}\right)&\ps{V_\beta(t)\prod_{k=1}^NV_{\alpha_k}(z_k)\prod_{l=1}^MV_{\beta_l}(s_l)}
			\end{split}
		\end{equation}
		where the $\Delta_\alpha\coloneqq\frac\alpha2(Q-\frac\alpha2)$ are called the \textit{conformal weights}.
	\end{theorem}
	
	%%%%%%%%%%%%%%%%%%%%%%%%%%%%%%%%%%%%%%%%%%%%%%%%
	%%%%%%%%%%%%%%%%%%%%%%%%%%%%%%%%%%%%%%%%%%%%%%%%
	
	%%%%%%%%%%%%%%%%%%%%%%%%%%%%%%%%%%%%%%%%%%%%%%%%
	%%%%%%%%%%%%%%%%%%%%%%%%%%%%%%%%%%%%%%%%%%%%%%%%
	
	\subsubsection{Degenerate fields and higher equations of motion}
	The Ward identities associated to the descendant at the second order from Theorem~\ref{thm:bpz_set_intro} hold true for any value of the weight $\beta$, provided that the quantities are well-defined. If however this weight takes a specific value then can combine the two equations~\eqref{eq:L2_intro} and~\eqref{eq:L11_intro} and obtain higher equations of motion for the correlation functions. 
	Indeed if we assume that $\beta$ is of the form $\beta=-\chi$ with $\chi\in\{\frac\gamma2,\frac2\gamma\}$ then
	\[
	(Q+\beta)\partial^2\Phi(t)-\left(\partial\Phi(t)\right)^2=-\frac1{\beta^2}\left[\beta\partial^2\Phi(t)+\left(\beta\partial\Phi(t)\right)^2\right].
	\]
	The left-hand side corresponds to the observables used to define the $\L_{-2}V_\beta$ descendant, while the right-hand side arises in the definition of $\L_{-(1,1)}V_\beta$ (see Subsection~\ref{subsec:bpzhem}).
	Based on this simple observation we will show in Subsection~\ref{subsec:bpzhem} that 
	\[
	\lim\limits_{\delta,\eps,\rho\to0}\ps{\frac1{\beta^2}\L_{-(1,1)}V_{\beta}(t)\prod_{k=1}^NV_{\alpha_k}(z_k)\prod_{l=1}^MV_{\beta_l}(s_l)}_{\delta,\eps,\rho}+\ps{\L_{-2}V_{\beta}(t)\prod_{k=1}^NV_{\alpha_k}(z_k)\prod_{l=1}^MV_{\beta_l}(s_l)}_{\delta,\eps,\rho}=0.
	\]
	Based on the definitions of the $\L_{-(1,1)}$ and $\L_{-2}$ descendants we will obtain that in the weak sense
	\begin{align*}
		&\left(\frac1{\beta^2}\partial_t^2+\sum_{k=1}^{2N+M}\frac{\partial_{z_k}}{t-z_k}+\frac{\Delta_{\alpha_k}}{(t-z_k)^2}\right)\ps{V_{\beta}(t)\prod_{k=1}^NV_{\alpha_k}(z_k)\prod_{l=1}^MV_{\beta_l}(s_l)}\\
		&=\lim\limits_{\delta,\eps,\rho\to0}\frac1{\beta^2}\mathfrak{R}_{-(1,1)}(\delta,\eps,\rho,\bm\alpha)+\mathfrak{R}_{-2}(\delta,\eps,\rho,\bm\alpha).
	\end{align*}
	
	On the other hand we will be able to compute in Subsection~\ref{subsec:remainder} the explicit expressions for the above limits of the remainder terms. Thanks to this we will prove the following:
	\begin{theorem}\label{thm:bpz_set_intro}
		For $\gamma\in(0,2)$, assume that $\beta\in\{-\frac\gamma2,-\frac2\gamma\}$ and set
		\[
		\mc D_{\beta}(t)\coloneqq \left(\frac1{\beta^2}\partial_t^2+\sum_{k=1}^{2N+M}\frac{\partial_{z_k}}{t-z_k}+\frac{\Delta_{\alpha_k}}{(t-z_k)^2}\right)\ps{V_{\beta}(t)\prod_{k=1}^NV_{\alpha_k}(z_k)\prod_{l=1}^MV_{\beta_l}(s_l)}.
		\]	
		Then as soon as $(\bm\alpha,\beta)\in\mc A_{N,M+1}$ we have in the weak sense of derivatives:
		\begin{equation}
			\mc D_{\alpha}(t)=\left\{\begin{matrix}
				&\left(1-\frac{\gamma^2}{4}\right)\left(\mu_L+\mu_R\right)\ps{V_{\gamma-\frac2\gamma}(t)\V}&\text{ if }\beta=-\frac2\gamma\\
				&{\scriptstyle\left(\mu_L^2+\mu_R^2-2\mu_L\mu_R\cos\left(\frac{\pi\gamma^2}{4}\right)-\mu\sin\left(\frac{\pi\gamma^2}{4}\right)\right)}\frac{\Gamma\left(\frac{\gamma^2}{4}\right)\Gamma\left(1-\frac{\gamma^2}{2}\right)}{\Gamma\left(1-\frac{\gamma^2}{4}\right)}\displaystyle\ps{V_{\frac{3\gamma}2}(t)\V}&\text{ if }\beta=-\frac\gamma2\text{ and }\gamma<\sqrt2\\
				&0&\text{ if }\beta=-\frac\gamma2\text{ and }\gamma>\sqrt2
			\end{matrix}\right.
		\end{equation}
		where we have set $\mu_L\coloneqq\lim\limits_{\eps\to 0}\mu_\partial(t-\eps)$, $\mu_R\coloneqq\lim\limits_{\eps\to 0}\mu_\partial(t+\eps)$ and $\V\coloneqq\prod_{k=1}^NV_{\alpha_k}(z_k)\prod_{l=1}^MV_{\beta_l}(s_l)$.
	\end{theorem}
	In the case where $\beta=-\frac\gamma2$ with $\gamma=\sqrt2$ one would need additional probabilistic estimates to conclude. Though we do not discuss this in the present document we believe it is an interesting question that would perhaps give new insights on Gaussian Multiplicative Chaos. Likewise with a little bit of extra work it might be possible to prove that this equation actually holds in the strong sense of derivatives: this is indeed the case, as shown in~\cite{Ang_zipper} based on hypoellipticity arguments, for the BPZ differential equations that arise when additional assumptions are made on the boundary cosmological constants (see Corollary~\ref{cor:bpz} below). 
	
	This statement allows to make sense of the higher equations of motion in terms of correlation functions. In the language of physics, it is a manifestation of the fact that null vectors (corresponding to degenerate fields) at the level two do not necessarily decouple, as proved in~\cite{BaWu}. We also stress that a consequence of Lemma~\ref{lemma:desc_Lintro} is the following result, which translates this fact for the null vector at the level one:
	\begin{proposition}
		For $\gamma\in(0,2)$, assume that $\beta=0$. Then for $\bm\alpha\in\mc A_{N,M}$ we have in the weak sense:
		\begin{equation}
			\partial_t\ps{\prod_{k=1}^NV_{\alpha_k}(z_k)\prod_{l=1}^MV_{\beta_l}(s_l)}=\left(\mu_L-\mu_R\right)\ps{V_{\gamma}(t)\prod_{k=1}^NV_{\alpha_k}(z_k)\prod_{l=1}^MV_{\beta_l}(s_l)}.
		\end{equation}
		Here the dependence in $t$ of the left-hand side is hidden in the change of boundary cosmological constant around $t$.
	\end{proposition}
	
	%%%%%%%%%%%%%%%%%%%%%%%%%%%%%%%%%%%%%%%%%%%%%%%%
	%%%%%%%%%%%%%%%%%%%%%%%%%%%%%%%%%%%%%%%%%%%%%%%%
	
	\subsubsection{BPZ differential equations}
	Finally, in the statement of Theorem~\ref{thm:bpz_set_intro} we see that if one assumes that $\mu_L+\mu_R=0$ in the case where $\beta=-\frac2\gamma$ or that\\
	$\mu_L^2+\mu_R^2-2\mu_L\mu_R\cos\left(\frac{\pi\gamma^2}{4}\right)-\mu\sin\left(\frac{\pi\gamma^2}{4}\right)$ if $\beta=-\frac\gamma2$ then the right-hand side vanishes. We thus recover one of the main results from~\cite{Ang_zipper}:
	\begin{corollary}\label{cor:bpz}
		Under the assumptions of Theorem~\ref{thm:bpz_set_intro}, further assume that $\mu_L$ and $\mu_R$ are of the form
		\begin{equation}
			\mu_L=g(\sigma_l)\text{ and }\mu_R=g(\sigma_R),\quad g(\sigma)\coloneqq \frac{\cos\left(\pi\gamma(\sigma-\frac Q2)\right)}{\sqrt{\sin\left(\pi\frac{\gamma^2}4\right)}}
		\end{equation}
		where $\sigma_L-\sigma_R=\pm\frac\beta2$. Then in the weak sense of derivatives
		\begin{equation}
			\left(\frac1{\beta^2}\partial_t^2+\sum_{k=1}^{2N+M}\frac{\partial_{z_k}}{t-z_k}+\frac{\Delta_{\alpha_k}}{(t-z_k)^2}\right)\ps{V_{\alpha}(t)\prod_{k=1}^NV_{\alpha_k}(z_k)\prod_{l=1}^MV_{\beta_l}(s_l)}=0.
		\end{equation}
	\end{corollary}
	Combining this statement with conformal covariance of the correlation functions we can deduce that correlation functions of the form $\ps{V_{\beta_1}(0)V_{\beta}(t)V_{\beta_2}(1)V_{\beta_3}(+\infty)}$ and \\
	$\ps{V_{\alpha_1}(i)V_{\beta}(t)V_{\beta}(+\infty)}$ are solutions of Fuchsian differential equations in the $t$ variable. .
	
	%%%%%%%%%%%%%%%%%%%%%%%%%%%%%%%%%%%%%%%%%%%%%%%%
	%%%%%%%%%%%%%%%%%%%%%%%%%%%%%%%%%%%%%%%%%%%%%%%%
	
	%%%%%%%%%%%%%%%%%%%%%%%%%%%%%%%%%%%%%%%%%%%%%%%%
	%%%%%%%%%%%%%%%%%%%%%%%%%%%%%%%%%%%
	
	\subsection{Some perspectives}
	
	%%%%%%%%%%%%%%%%%%%%%%%%%%%%%%%%%%%%%%%%%%%%%%%%
	%%%%%%%%%%%%%%%%%%%%%%%%%%%%%%%%%%%%%%%%%%%%%%%%
	\iffalse
	\subsubsection{Organization of the paper}
	
	%%%%%%%%%%%%%%%%%%%%%%%%%%%%%%%%%%%%%%%%%%%%%%%%
	%%%%%%%%%%%%%%%%%%%%%%%%%%%%%%%%%%%%%%%%%%%%%%%%
	In order to prove our main results we first need to provide a rigorous definition of the objects we consider. This is done in Section~\ref{sec:def} where we introduce the correlation functions studied in this document and recall some of their analytic properties. We then introduce in Section~\ref{sec:ward} the descendant fields $\L_{-1}$ and $\L_{-2}$ and show that Equations~\eqref{eq:der1intro} and~\eqref{eq:L2_intro} hold true. Finally in Section~\ref{sec:bpz} we prove Theorem~\ref{thm:bpz_set_intro} by providing a rigorous meaning to the above reasoning.
	\fi
	
	\subsubsection{Higher-order BPZ equations}
	In the present we provide a proof of the BPZ differential equation that corresponds to a degenerate field at the second order, that is a Vertex Operator of the form $V_{-\chi}$ where $\chi\in\{-\frac\gamma2,-\frac2\gamma\}$. However it is predicted in the physics that more generally, BPZ differential equations should hold true when one considers Vertex Operators $V_{\alpha_{r,s}}$ for any pair $(r,s)$ of positive integers, where we have set $\alpha_{r,s}=-\frac{(r-1)\gamma}{2}-\frac{(s-1)2}{\gamma}$. We believe that the methodology developed in this document should allow to provide a rigorous method to prove these identities for boundary Liouville theory, though it would involve rather heavy computations.
	
	\subsubsection{Boundary Toda CFT}
	Liouville theory actually corresponds to the simplest instance of a Toda conformal field theory. These two-dimensional quantum field theories possess an enhanced level of symmetry which make their study more involved than Liouville theory. Still we should be able to prove BPZ-type differential equations for these models; this has been proved for the $\mathfrak{sl}_3$ Toda theory on the sphere in the series of work~\cite{Toda_construction,Toda_OPEWV,Toda_correl1,Toda_correl2}. 
	In recent works with Huguenin~\cite{CH_sym1,CH_sym2} and using some of the ideas developed in this document we have proved Ward identities and higher equations of motion for the boundary $\mathfrak{sl}_3$ Toda theory. To the best of our knowledge there are no such identities in the physics apart from the case studied in~\cite{FaRi} corresponding to the one-point bulk correlator. We believe that proving these identities for the boundary Toda theory will allow us to compute some of the structure constants of the theory which are still unknown in the physics.
	
	\vspace{0.3cm}
	\textit{\textbf{Acknowledgments:}}
	I would like to warmly thank Rémi Rhodes and Vincent Vargas for having encouraged me to write the present document and for fruitful discussions. I would also like to express my gratitude to Nikos Zygouras and the University of Warwick, where part of this work has been undertaken, for their hospitality. I am also indebted to Guillaume Baverez for having pointed out a mistake in a prior version of the document and for the improvement of the paper it led to.
	
	The author has been supported by Eccellenza grant 194648 of the Swiss National Science Foundation and is a member of NCCR SwissMAP.
	
	%%%%%%%%%%%%%%%%%%%%%%%%%%%%%%%%%%%%%%%%%%%%%%%%
	%%%%%%%%%%%%%%%%%%%%%%%%%%%%%%%%%%%%%%%%%%%%%%%%
	%%%%%%%%%%%%%%%%%%%%%%%%%%%%%%%%%%%%%%%%%%%%%%%%
	
	%%%%%%%%%%%%%%%%%%%%%%%%%%%%%%%%%%%
	%%%%%%%%%%%%%%%%%%%%%%%%%%%%%%%%%%%
	%%%%%%%%%%%%%%%%%%%%%%%%%%%%%%%%%%%
	
	\section{Probabilistic definition of the correlation functions}\label{sec:def}
	The goal of this section is to provide the mathematical definition of the correlation functions presented in the introduction and describe some of their analytic properties. For this purpose we first briefly recall the probabilistic setting introduced in~\cite{HRV16} where they are defined and proceed to their rigorous definition. We then provide some of their properties that we will use throughout the document.
	%%%%%%%%%%%%%%%%%%%%%%%%%%%%%%%%%%%%%%%%%%%%%%%%
	%%%%%%%%%%%%%%%%%%%%%%%%%%%%%%%%%%%%%%%%%%%%%%%%
	
	%%%%%%%%%%%%%%%%%%%%%%%%%%%%%%%%%%%%%%%%%%%%%%%%
	%%%%%%%%%%%%%%%%%%%%%%%%%%%%%%%%%%%
	
	\subsection{The probabilistic setting}
	%%%%%%%%%%%%%%%%%%%%%%%%%%%%%%%%%%%%%%%%%%%%%%%%
	%%%%%%%%%%%%%%%%%%%%%%%%%%%%%%%%%%%%%%%%%%%%%%%%
	
	%%%%%%%%%%%%%%%%%%%%%%%%%%%%%%%%%%%%%%%%%%%%%%%%
	%%%%%%%%%%%%%%%%%%%%%%%%%%%%%%%%%%%
	To start with we would like to make sense of the path integral~\eqref{eq:path_integral} in the case where we assume the underlying surface with boundary to be given by the upper-half plane $\H$ equipped with the metric $g=\norm{z}_+\norm{dz}^2$ where recall that $\norm{\cdot}_+=\max(1,\norm{\cdot})$ (so that $(\H,g)$ is conformally equivalent to the unit disk equipped with a flat metric). To this end recall that the action functional~\eqref{eq:Toda_action} is made of several terms. 
	
	In order to provide a mathematical interpretation of the path integral~\eqref{eq:path_integral} the first step is to consider the quadratic term in the action and make sense of the formal measure $e^{\frac{1}{4\pi} \int_{\Sigma}\norm{\partial_g\varphi}^2\,{\rm dv}_{g}}D\varphi$.
	This achieved in~\cite{HRV16} based on the consideration of a Gaussian Free Field $\X$ with Neumann boundary conditions. In the present document this random distribution will be defined as a Gaussian field that is centered and with covariance kernel given by
	\begin{equation}
		\expect{\X(x)\X(y)}=G(x,y),\quad\text{with }G(x,y)\coloneqq\ln\frac{1}{\norm{x-y}\norm{x-\bar y}}+2\ln\norm{x}_++2\ln\norm{y}_+
	\end{equation} 
	for $x,y$ in $\overline\H$. Note that if we set
	\begin{equation}
		G_0(x,y)\coloneqq \ln\frac{1}{\norm{x-y}}+\ln\norm{x}_++\ln\norm{y}_+
	\end{equation} then we can write that $G(x,y)=G_0(x,y)+G_0(x,\bar y)$.
	This defines~\cite{dubedat,She07}, an element of the Sobolev space with negative index $H^{-1}_{loc}(\H)$ but that only makes sense as a generalized function. To obtain a smooth function out of it it is standard to regularize this field \textit{e.g.} using a smooth mollifier $\eta$ and setting for positive $\rho$
	\begin{equation}\label{eq:regular_X}
		\X_\rho(x)\coloneqq\int_\H\X_\rho(y)\eta_\rho(x-y)\norm{dy}^2
	\end{equation}
	where $\eta_\rho(\cdot)\coloneqq\frac1{\rho^2}\eta(\frac{\cdot}{\rho})$. 
	
	We also need to take into account the more geometric terms associated to the scalar and geodesic curvatures in the action~\eqref{eq:Toda_action}. 
	Using their explicit expression this yields the following probabilistic interpretation for such terms (see~\cite{HRV16} for more details)
	\[
	\frac1{\mc Z}\int_{\mc F}F(\varphi)e^{\frac{1}{4\pi} \int_{\Sigma}  \Big (  \norm{\partial_g\varphi}^2   +Q R_g\varphi \Big)\,{\rm dv}_{g}+\frac{1}{2\pi} \int_{\partial\Sigma}  Q K_g\varphi \,{\rm dl}_{g}}D\varphi\rightarrow \lim\limits_{\rho\to0}\int_\R e^{-Q\bm c}\expect{F(\X_\rho-2Q\ln\norm{\cdot}_++\bm c)}d\bm c
	\]
	where the constant mode $\bm c$ is sampled according to the Lebesgue measure.
	
	The probabilistic interpretation of the full path integral~\eqref{eq:path_integral} is made by inserting in the above the rest of the action~\eqref{eq:Toda_action}, which is made of exponentials of the field $\Phi$. To this end we use the theory of Gaussian Multiplicative Chaos, namely that based on this GFF $\X$ one can define bulk and boundary Gaussian Multiplicative Chaos measures via the following limits:
	\begin{equation}
		\mc A(dx)\coloneqq\lim\limits_{\rho\to0}\rho^{\frac{\gamma^2}{2}}e^{\gamma\X_\rho(x)}\norm{dx}^2;\quad \mc L(dx)\coloneqq\lim\limits_{\rho\to0}\rho^{\frac{\gamma^2}{4}}e^{\frac\gamma2\X_\rho(x)}dx.
	\end{equation}
	These limits hold in probability and in the sense of weak convergence of measures~\cite{Ber,RV_GMC}. Before moving on let us stress that the bulk Gaussian Multiplicative Chaos measure is such that (see~\cite[Remark 2.7]{ARSZ})
	\begin{equation*}
		\mc A(dx)=\frac{\norm{x}_+^{2\gamma^2}}{\norm{x-\bar x}^{\frac{\gamma^2}2}}\lim\limits_{\rho\to0}e^{\gamma\X_\rho(x)-\frac{\gamma^2}2\expect{\X_\rho(x)^2}}\norm{dx}^2.
	\end{equation*}
	
	To summarize the probabilistic interpretation of the path integral~\eqref{eq:path_integral} and the definition of the \textit{Liouville field} $\Phi$ is made by considering for $F$ a bounded continuous functional over the Sobolev space $\mathrm H^{-1}_{loc}(\H)$ such that the following limit is well-defined:
	\begin{equation}\label{eq:liouville_field}
		\begin{split}
			\ps{F(\Phi)}\coloneqq \lim\limits_{\rho\to0}\int_{\R} &e^{-Q\bm c}\E\Big[F\left(\X_\rho-2Q\ln \norm{\cdot}_++\bm c\right)\\
			&\exp\left(-\mu e^{\gamma\bm c}\int_\H \rho^{\frac{\gamma^2}{2}}e^{\gamma\X_\rho(x)}\norm{dx}^2-e^{\frac\gamma2 \bm c}\int_\R \rho^{\frac{\gamma^2}{4}}e^{\frac\gamma2\X_\rho(x)}\mu_\partial(dx)\right)\Big]d\bm c
		\end{split}
	\end{equation}
	where $\mu>0$ is the bulk cosmological constant while $\frac{\mu_\partial(dx)}{dx}$ is a piecewise constant function that is complex-valued with non-negative real part.
	
	\subsection{Correlation functions within the Seiberg bounds}\label{subsec:correl}
	
	%%%%%%%%%%%%%%%%%%%%%%%%%%%%%%%%%%%%%%%%%%%%%%%%
	%%%%%%%%%%%%%%%%%%%%%%%%%%%%%%%%%%%%%%%%%%%%%%%%
	
	%%%%%%%%%%%%%%%%%%%%%%%%%%%%%%%%%%%%%%%%%%%%%%%%
	%%%%%%%%%%%%%%%%%%%%%%%%%%%%%%%%%%%%%%%%%%%%%%%%
	
	The correlation functions of Vertex Operators depend on a family of $N+M$ weights $(\alpha_1,\cdots,\alpha_N,\beta_1,\cdots,\beta_M)\in\left(-\infty,Q\right)^{N+M}$ together with distinct insertions $(z_1,\cdots,z_N,s_1,\cdots,s_M)\in\H^{N}\times\R^{M}$. They formally correspond to considering in the definition of the Liouville field~\eqref{eq:liouville_field} the functional
	\[
	F(\Phi)=\prod_{k=1}^N e^{\alpha_k\Phi(z_k)}\prod_{l=1}^M e^{\frac{\beta_l}2\Phi(s_l)}.
	\]
	To the boundary insertions is also associated a choice of boundary measure $\mu_{\partial}(dx)$ that is piecewise constant and complex-valued and given by
	\begin{equation}\label{eq:mu_partial}
		\frac{\mu_{\partial}(dx)}{dx}\coloneqq\sum_{i=0}^{M}\mu_{i+1}\mathds1_{x\in(s_i,s_i+1)}
	\end{equation}
	with the convention that $s_0=-\infty$, $s_{M+1}=+\infty$ and $\mu_{M+1}=\mu_1$.
	
	In the sequel we denote
	\begin{equation}
		\bm s\coloneqq\sum_{k=1}^{N}\alpha_k+\sum_{l=1}^M\frac{\beta_l}2-Q,
	\end{equation} 
	and we will extensively use the shorthands $(\alpha_1,\cdots,\alpha_{2N+M})\coloneqq(\alpha_1,\cdots,\alpha_N,\alpha_1,\cdots,\alpha_N,\beta_1,\cdots,\beta_M)$ and $(z_1,\cdots,z_{2N+M})\coloneqq(z_1,\cdots,z_N,\bar z_1,\cdots,\bar z_N,s_1,\cdots,s_M)$.
	%%%%%%%%%%%%%%%%%%%%%%%%%%%%%%%%%%%%%%%%%%%%%%%%
	
	%%%%%%%%%%%%%%%%%%%%%%%%%%%%%%%%%%%%%%%%%%%%%%%%
	\subsubsection{The regularized correlation functions}
	We can make sense of the correlation function by using a regularization procedure and rely on the probabilistic definition of the Liouville field~\eqref{eq:liouville_field}. Basic manipulations based on Girsanov's theorem then allow to arrive at the following expression for the regularized correlation function, given by
	\begin{equation}\label{eq:def_reg_correl}
		\begin{split}
			&\ps{\prod_{k=1}^NV_{\alpha_k}(z_k)\prod_{l=1}^MV_{\beta_l}(s_l)}_{\delta,\eps,\rho}\coloneqq
			P(\bm z,\bm \alpha)\int_{\R} e^{\bm s\bm c}\expect{\exp\left(-\mu e^{\gamma \bm c}\mc A_{\delta,\eps,\rho}-e^{\frac\gamma2 \bm c}\mc L_{\delta,\eps,\rho}\right)}d\bm c
		\end{split}
	\end{equation}
	for $\delta,\eps,\rho>0$ small enough and where we have introduced the notations
	\begin{equation}\label{eq:AL}
		\begin{split}
			&\mc A_{\delta,\eps,\rho}\coloneqq\int_{\H_{\delta,\eps}}\prod_{k=1}^{2N+M}\left(\frac{\norm{x}_+}{\norm{z_k-x}}\right)^{\gamma\alpha_k}\rho^{\frac{\gamma^2}2}e^{\gamma\X_\rho(x)}\norm{dx}^2\quad\text{and}\\
			&\mc L_{\delta,\eps,\rho}\coloneqq\int_{\R_{\eps}}\prod_{k=1}^{2N+M}\left(\frac{\norm{x}_+}{\norm{z_k-x}}\right)^{\frac{\gamma\alpha_k }2}\rho^{\frac{\gamma^2}4}e^{\frac\gamma2 \X_\rho(x)}\mu_{\partial}(dx).
		\end{split}
	\end{equation}
	The domains of integration are given by
	$\H_{\delta,\eps}\coloneqq \left(\H+i\delta\right)\setminus\left(\cup_{k=1}^{2N+M}B(z_k,\eps)\right)$ and $\R_{\eps}\coloneqq \R\setminus\left(\cup_{l=1}^M(s_l-\eps,s_l+\eps)\right)$. We assume that $\delta$ and $\eps$ are small enough so that the balls $B(z_k,\eps)$ are disjoint and all contained within $\H+i\delta$ (see Figure~\ref{fig:domains} below). Finally, the prefactor is
	\begin{equation}\label{eq:pref}
		P(\bm z,\bm\alpha)\coloneqq \prod_{k< l}\norm{z_k-z_l}^{-\alpha_k\alpha_l}\prod_{k=1}^N\norm{z_k-\bar z_k}^{\frac{\alpha_k^2}{2}}.
	\end{equation}
	One can check that $P$ behaves like $\norm{z_k-\bar z_k}^{-\frac{\alpha_k^2}2}$ as $z_k-\bar z_k\to0$.
	
	\begin{center}
		\includegraphics[width=0.98\linewidth]{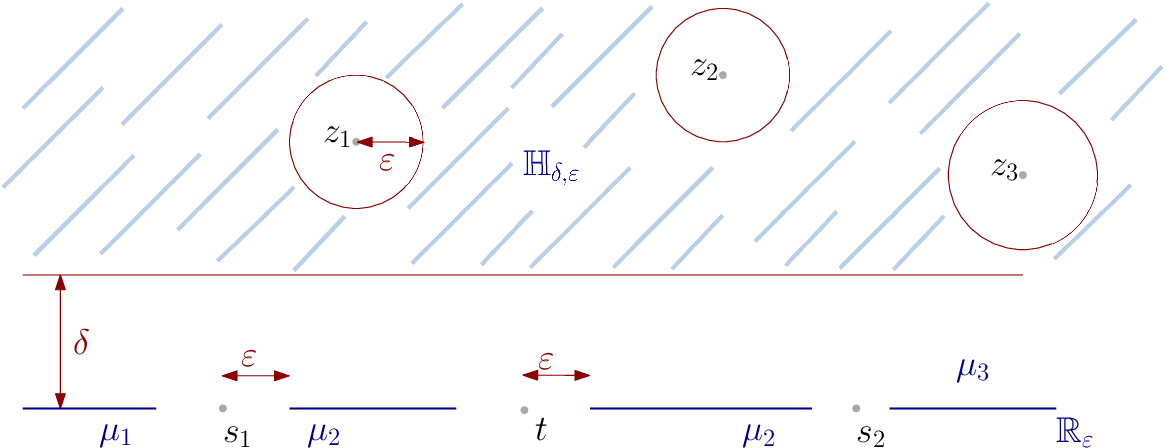}
		\captionof{figure}{The domains of integration}
		\label{fig:domains}
	\end{center}

	More generally for $F$ a bounded continuous functional over the Sobolev space $\mathrm H^{-1}_{loc}(\H)$ we define the \textit{Liouville field} $\Phi$ within correlation functions by setting
	\begin{equation}
		\begin{split}
			&\ps{F(\Phi)\prod_{k=1}^NV_{\alpha_k}(z_k)\prod_{l=1}^MV_{\beta_l}(s_l)}_{\delta,\eps,\rho}\coloneqq \\
			&P(\bm z,\bm \alpha)\int_{\R} e^{\bm s\bm c}\expect{F\left(\X_\rho-2Q\ln \norm{\cdot}_++ H+\bm c\right)\exp\left(-\mu e^{\gamma\bm c}\mc A_{\delta,\eps,\rho}-e^{\frac\gamma2 \bm c}\mc L_{\delta,\eps,\rho}\right)}d\bm c
		\end{split}
	\end{equation}
	where at last we have introduced
	\begin{equation}
		H(x)\coloneqq\sum_{k=1}^N\alpha_kG(x,z_k)+\sum_{l=1}^M\frac{\beta_l}2G(x,s_l)=\sum_{k=1}^{2N+M}\alpha_kG_0(x,z_k)\cdot
	\end{equation}

	\subsubsection{Existence of the correlation functions}
	We first need to ensure that these objects are well-defined. This is the case provided that the following set of assumptions is satisfied:
	\begin{equation}\label{eq:seiberg}
		\begin{split}
			&\alpha_k,\beta_l<Q\text{ for all }k,l;\quad\sum_{k=1}^N\alpha_k+\sum_{l=1}^M\frac{\beta_l}2-Q>0\quad\text{(Seiberg bounds)};\\
			&\mu>0\quad\text{and}\quad\Re(\mu_l)\geq0\quad\text{for all }1\leq l\leq M.
		\end{split}
	\end{equation}
	Indeed this statement is proved in~\cite[Theorem 3.1]{HRV16} under the stronger assumption that the boundary cosmological constant $\mu_\partial$ is constant and real-valued with $\mu_\partial\geq 0$. However if we make the weaker assumption that it is as in Equation~\eqref{eq:mu_partial} with $Re(\mu_l)\geq0$ for all $1\leq l\leq M$ then we can bound
	\[
	\Re\left(\mu e^{\gamma \bm c}\mc A_{\delta,\eps,\rho}+e^{\frac\gamma2 \bm c}\mc L_{\delta,\eps,\rho}\right)\geq \mu e^{\gamma \bm c}\mc A_{\delta,\eps,\rho}.
	\]  
	Thanks to this assumption and going along the proof of~\cite[Theorem 3.1]{HRV16} we see that the integral in the $\bm c$ variable is indeed absolutely convergent as soon as $\bm s>0$, showing that the regularized correlation functions are well-defined. 
	
	To take the limit of these regularized correlation functions when the parameters go to $0$ we again rely on the proof of~\cite[Theorem 3.1]{HRV16}. Under the above assumptions the following limit is then seen to be well-defined:
	\begin{equation}
		\ps{\prod_{k=1}^NV_{\alpha_k}(z_k)\prod_{l=1}^MV_{\beta_l}(s_l)}\coloneqq\lim\limits_{\delta\to0}\lim\limits_{\eps\to0}\lim\limits_{\rho\to0}\ps{\prod_{k=1}^NV_{\alpha_k}(z_k)\prod_{l=1}^MV_{\beta_l}(s_l)}_{\delta,\eps,\rho}.
	\end{equation}
	More generally, under the assumptions of Equation~\eqref{eq:seiberg} we can define for observables $F:\mathrm H^{-1}_{loc}(\H)\to\R$:
	\begin{equation}
		\ps{F[\Phi]\prod_{k=1}^NV_{\alpha_k}(z_k)\prod_{l=1}^MV_{\beta_l}(s_l)}\coloneqq\lim\limits_{\delta\to0}\lim\limits_{\eps\to0}\lim\limits_{\rho\to0}\ps{F[\Phi]\prod_{k=1}^NV_{\alpha_k}(z_k)\prod_{l=1}^MV_{\beta_l}(s_l)}_{\delta,\eps,\rho}
	\end{equation}
	as soon as the above limit exists. For the sake of simplicity we will omit the regularization scale $\rho$ in the notations when the limit $\rho\to0$ has been taken.

	%%%%%%%%%%%%%%%%%%%%%%%%%%%%%%%%%%%%%%%%%%%%%%%%
	%%%%%%%%%%%%%%%%%%%%%%%%%%%%%%%%%%%%%%%%%%%%%%%%
	
	%%%%%%%%%%%%%%%%%%%%%%%%%%%%%%%%%%%%%%%%%%%%%%%%
	%%%%%%%%%%%%%%%%%%%%%%%%%%%%%%%%%%%%%%%%%%%%%%%%
	
	\subsection{Analytic extension of the correlation functions}
	The correlation functions as defined before are well defined provided that the Seiberg bounds~\eqref{eq:seiberg} are satisfied. In order to extend the range of validity for which such correlation functions make sense we provide here an analytic continuation of such quantities. 
	\begin{proposition}\label{prop:analycity}
		Let us denote by $\mc A_{N,M}$ the subset of $\R^{N+M}$ defined by the conditions:
		\begin{equation}
			\alpha_k,\beta_l<Q\text{ for all }k,l;\quad\sum_{k=1}^{N}\alpha_k+\sum_{l=1}^M\frac{\beta_l}2-Q>-\gamma\vee-\frac2\gamma\vee\max\limits_{1\leq k\leq 2N+M}(\alpha_k-Q).
		\end{equation}
		Under the additional assumptions that $\mu>0\text{ and }\Re(\mu_l)\geq0\text{ for all }1\leq l\leq M$ the map 
		\[
		\bm\alpha\mapsto \int_{\R} e^{\bm s\bm c}\expect{\exp\left(-\mu e^{\gamma \bm c}\mc A_{\delta,\eps}-e^{\frac\gamma2 \bm c}\mc L_{\delta,\eps}\right)-\mathfrak{R}_{\bm\alpha(\bm c)}}d\bm c
		\]
		is meromorphic in a complex neighborhood of $\mc A_{N,M}$, where the remainder term is given by
		\[
		\mathfrak{R}_{\bm\alpha}(\bm c)\coloneqq \mathds{1}_{\bm s<0}-e^{\frac\gamma2\bm c}\mc L_{\delta,\eps}\mathds{1}_{\bm s<-\frac\gamma2}.
		\]
		Its poles are given by the $\bm\alpha$ for which $\bm s\in\{0,-\frac\gamma2\}$.
	\end{proposition}
	\begin{proof}
		Note that here we have already taken the limit $\rho\to0$. This statement has been proved in~\cite{ARSZ} for the case of the boundary three-point constant. The general proof relies on the very same arguments as the ones developed in~\cite[Theorem 2.2]{Toda_correl2}, apart from the fact that the present one is much simpler in that no reflection coefficients are involved.
	\end{proof}
	With a little bit of extra work and following the arguments developed in~\cite{Toda_correl2} one could extend the range of values for which the analytic continuation makes sense. 
	
	Proposition~\ref{prop:analycity} allows to provide an analytic extension of the correlation functions beyond the Seiberg bounds:
	\begin{proposition}\label{prop:ana_correl}
		For $\bm\alpha\in\mc A_{N,M}$, let us extend the definition of the regularized correlation functions~\eqref{eq:def_reg_correl} by setting
		\begin{equation}\label{eq:ana_correl}
			\begin{split}
				&\ps{\prod_{k=1}^NV_{\alpha_k}(z_k)\prod_{l=1}^MV_{\beta_l}(s_l)}_{\delta,\eps}\\
				&\coloneqq P(\bm z,\bm \alpha)\int_{\R} e^{\bm s\bm c}\E\Big[\left(\exp\left(-\mu e^{\gamma\bm c}\mc A_{\delta,\eps}-e^{\frac\gamma2 \bm c}\mc L_{\delta,\eps}\right)-\mathfrak{R}_{\bm\alpha}(\bm c)\right)\Big]d\bm c.
			\end{split}
		\end{equation}
		Then as soon as $\mu>0\text{ and }\Re(\mu_l)\geq0\text{ for all }1\leq l\leq M$ the map 
		\[
		\bm\alpha\mapsto \int_{\R} e^{\bm s\bm c}\expect{\exp\left(-\mu e^{\gamma \bm c}\mc A_{\delta,\eps}-e^{\frac\gamma2 \bm c}\mc L_{\delta,\eps}\right)-\mathfrak{R}_{\bm\alpha(\bm c)}}d\bm c
		\]
		is meromorphic in a complex neighborhood of $\mc A_{N,M}$.
	\end{proposition}
	Thanks to this statement we are naturally lead to extending the definition of the Liouville field when these more generic assumptions are made. Namely we set for functionals $F$ for which the following is well-defined:
	\begin{equation}\label{eq:liouville_field_ana}
		\begin{split}
			&\ps{F(\Phi)\prod_{k=1}^NV_{\alpha_k}(z_k)\prod_{l=1}^MV_{\beta_l}(s_l)}_{\delta,\eps}\coloneqq P(\bm z,\bm \alpha)\int_{\R} e^{\bm s\bm c}\\
			&\E\Big[F\left(\X_\rho-2Q\ln \norm{\cdot}_++ H+\bm c\right)\left(\exp\left(-\mu e^{\gamma\bm c}\mc A_{\delta,\eps}-e^{\frac\gamma2 \bm c}\mc L_{\delta,\eps}\right)-\mathfrak{R}_{\bm\alpha}(\bm c)\right)\Big]d\bm c.
		\end{split}
	\end{equation}
	%%%%%%%%%%%%%%%%%%%%%%%%%%%%%%%%%%%%%%%%%%%%%%%%
	%%%%%%%%%%%%%%%%%%%%%%%%%%%%%%%%%%%%%%%%%%%%%%%%
	
	%%%%%%%%%%%%%%%%%%%%%%%%%%%%%%%%%%%%%%%%%%%%%%%%
	%%%%%%%%%%%%%%%%%%%%%%%%%%%%%%%%%%%%%%%%%%%%%%%%
	
	\subsection{Some technical estimates}
	In order to ensure finiteness of the correlation functions and of their derivatives we first need some estimates on their dependence in the insertion points $z_k$, $1\leq k\leq 2N+M$. We describe such properties below.
	To simplify the notations we set $\V\coloneqq\prod_{k=1}^NV_{\alpha_k}(z_k)\prod_{l=1}^MV_{\beta_l}(s_l)$.
	
	\subsubsection{Bounds at infinity}
	To ensure integrability of the correlation functions first we are concerned at what happens when insertions diverge.
	\begin{lemma}\label{lemma:inf_integrability}
		Assume that $\bm{\alpha}\in\mc A_{N,M}$ and consider two families $\bm{x} \coloneqq\left(x^{(1)},\cdots,x^{(n)}\right)\in\H^n$ and $\bm{y} \coloneqq\left(y^{(1)},\cdots,y^{(m)}\right)\in\R^m$. Then for any  $h>0$, if the $\bm{x}$, $\bm y$ and $\bm{z}$ stay in the domain $U_h\coloneqq \left\lbrace{\bm w : h<\min\limits_{i\neq j}\norm{w_i-w_j},\quad h<\min\limits_{i, w_i\in\H}\norm{w_i-\overline{w_i}}}\right\rbrace$, then there exists $C=C_h$ such that, uniformly in $\delta,\eps,\rho$,
		\begin{equation*}
			\ps{\prod\limits_{i=1}^{n}V_{\gamma}\left(x^{(i)}\right)\prod\limits_{j=1}^{m}V_{\gamma}\left(y^{(j)}\right)\V}_{\delta,\eps,\rho}\leq C_h \prod\limits_{i=1}^{n}\left(1+\norm{x^{(i)}}\right)^{-4}\prod\limits_{j=1}^{m}\left(1+\norm{j^{(l)}}\right)^{-2}.
		\end{equation*}
	\end{lemma}
	\begin{proof}
		The reasoning is very similar to that conducted in the proof of~[Item (2) of Lemma 3.2]\cite{Toda_OPEWV} so we will provide details for the points that differ. To start with we recall the definition of the correlation functions in Equation~\eqref{eq:ana_correl}, with the bulk and boundary measures defined in Equation~\eqref{eq:AL} and the prefactor in Equation~\eqref{eq:pref}. To start with we can rewrite the terms containing the $x^{(i)}$ and the $y^{(j)}$ in both the prefactor and the measures by writing
		\begin{align*}
			&\norm{x^{(i)}-z}=\left(1+\norm{x^{(i)}}\right)\frac{\norm{x^{(i)}-z}}{\left(1+\norm{x^{(i)}}\right)}\quad\text{and}\\
			&\norm{x^{(i)}-x^{(j)}}=\left(1+\norm{x^{(i)}}\right)\left(1+\norm{x^{(j)}}\right)\frac{\norm{x^{(i)}-x^{(j)}}}{\left(1+\norm{x^{(i)}}\right)\left(1+\norm{x^{(j)}}\right)}
		\end{align*}
		where $z$ is either an insertion point (such terms appear in the prefactor) or a variable of integration (in the bulk and boundary measures). We then factor out the terms $\left(1+\norm{x^{(i)}}\right)$: in the prefactor we get a multiplicative factor
		\begin{align*}
			\prod_{i=1}^n\left(1+\norm{x^{(i)}}\right)^{-2\gamma\left(\sum_{k=1}^{2N+M}\alpha_k+2(n-1)\gamma+m\gamma\right)}\prod_{j=1}^m\left(1+\norm{y^{(j)}}\right)^{-\gamma\left(\sum_{k=1}^{2N+M}\alpha_k+2n\gamma+(m-1)\gamma\right)}
		\end{align*}
		while in front of respectively the bulk and boundary measures we have
		\begin{align*}
			K\coloneqq\prod_{i=1}^n\left(1+\norm{x^{(i)}}\right)^{-2\gamma^2}\prod_{j=1}^m\left(1+\norm{y^{(j)}}\right)^{-\gamma^2},\quad \prod_{i=1}^n\left(1+\norm{x^{(i)}}\right)^{-\gamma^2}\prod_{j=1}^m\left(1+\norm{y^{(j)}}\right)^{-\frac{\gamma^2}2}.
		\end{align*}
		By shifting the constant mode $\bm c$ by $\frac1\gamma\ln(K)$ we end up with a prefactor
		\begin{align*}
			&\prod_{i=1}^n\left(1+\norm{x^{(i)}}\right)^{-2\gamma\left(\sum_{k=1}^{2N+M}\alpha_k+2(n-1)\gamma+m\gamma\right)}\prod_{j=1}^m\left(1+\norm{y^{(j)}}\right)^{-\gamma\left(\sum_{k=1}^{2N+M}\alpha_k+2n\gamma+(m-1)\gamma\right)}\times K^{\frac{\bm s}\gamma}\\
			&=\prod_{i=1}^n\left(1+\norm{x^{(i)}}\right)^{-2\gamma\left(Q-2\gamma\right)}\prod_{j=1}^m\left(1+\norm{y^{(j)}}\right)^{-\gamma\left(Q-2\gamma\right)}=\prod_{i=1}^n\left(1+\norm{x^{(i)}}\right)^{-4}\prod_{j=1}^m\left(1+\norm{y^{(j)}}\right)^{-2}.
		\end{align*}
		We conclude in the same fashion as in the proof of~[Item (2) of Lemma 3.2]\cite{Toda_OPEWV} by relying on the fact that the other terms remain uniformly bounded over $U_h$ since the ratio $\frac{\norm{x^{(i)}-z}}{\left(1+\norm{x^{(i)}}\right)}$ is bounded for $x^{(i)}$ ranging over $\C\setminus B(z,h)$.
	\end{proof}
	
	\subsubsection{Fusion estimates}	
	We then consider what happens when two insertion points collide: for this purpose we describe so-called \textit{fusion estimates}:
	\begin{lemma}\label{lemma:fusion}
		Assume that $\bm{\alpha}\in\mc A_{N,M}$ and that all pairs of points in $\bm{z}$ are separated by some distance $h>0$ except for one pair $(z_1,z_2)$. 
		In the limit where $z_1\to z_2$ with all insertions except for $z_1$ fixed, for any positive $\eta$ there exists a positive constant $K=K(h,\bm z,\bm \alpha)$ such that:
		\begin{enumerate}
			\item if $z_1,z_2\in\H$ then
			\begin{equation}\label{eq:fusion_hh}
				\ps{\V}_{\delta,\eps,\rho}\leq K \norm{z_1-z_2}^{-\alpha_1\alpha_2+\left(\frac{1}{2}\left(\alpha_1+\alpha_2-Q\right)^2-\eta\right)\mathds{1}_{\alpha_1+\alpha_2-Q>0}};
			\end{equation}
			\item if $z_1,z_2\in\R$ then
			\begin{equation}\label{eq:fusion_rr}
				\ps{\V}_{\delta,\eps,\rho}\leq K \norm{z_1-z_2}^{-\frac{\alpha_1\alpha_2}2+\left(\frac{1}{4}\left(\alpha_1+\alpha_2-Q\right)^2-\eta\right)\mathds{1}_{\alpha_1+\alpha_2-Q>0}};
			\end{equation}
			\item if $z_1\in\H$ while $z_2\in\R$ then
			\begin{equation}\label{eq:fusion_hr}
				\begin{split}
					&\ps{\V}_{\delta,\eps,\rho}\leq \\
					&K \norm{z_1-\bar{z_1}}^{-\frac{\alpha_1^2}2+\left(\left(\alpha_1-\frac Q2\right)^2-\eta\right)\mathds{1}_{\alpha_1-\frac Q2>0}}\norm{z_1-z_2}^{-\alpha_1\alpha_2+\left(\left(\alpha_1+\frac{\alpha_2}2-\frac Q2\right)^2-\eta\right)\mathds{1}_{\alpha_1+\frac{\alpha_2}2-\frac Q2>0}}.
				\end{split}
			\end{equation}
		\end{enumerate}
	\end{lemma}
	\begin{proof}
		This type of estimates is standard in the study of probabilistic Liouville theory; we refer for instance to~\cite[Section 5]{KRV_loc} where are discussed the integrability properties as well as the fusion asymptotics of the correlation functions on the sphere. The method is easily adapted to our setting, see for instance~\cite{fusion} where the case of fusion asymptotics for boundary Liouville theory is treated: the first case here can be deduced from the proof of~\cite[Theorem 5.3]{fusion}, the second one from the proof of~\cite[Theorem 5.5]{fusion}. As for the third one it follows from~\cite[Equation (5.21)]{fusion} and the paragraph above it: the first factor corresponds to the \lq\lq absorption" of the bulk insertion while the second one to the fusion of the two boundary insertions thus obtained.
	\end{proof}

	\subsubsection{Singular integrals}
	A first consequence of the fusion asymptotics is the existence of \textit{a priori} ill-defined integrals that contain correlation functions. To be more specific we have the following statement:
	\begin{lemma}~\label{lemma:fusion_integrability}
		Assume that $\bm{\alpha}\in\mc A_{N,M}$ and that for any $z\in\bm{z}$ we have $\norm{z}>2$. Then the following integrals are absolutely convergent:
		\begin{equation}\label{eq:fusion_int}
			\begin{split}
				&\int_{\frac12\D\times(\D\setminus\frac12\D)}\frac1{y-x}\ps{V_{\gamma}(x+i)V_{\gamma}(y+i)\V}\norm{d^2x}\norm{d^2y},\\
				&\int_{-1}^0\int_0^1\frac1{y-x}\ps{V_{\gamma}(x)V_{\gamma}(y)\V}dxdy\quad\text{and}\\
				&\int_{\D\cap\H}\int_1^2\frac1{y-x}\ps{V_{\gamma}(x)V_{\gamma}(y)\V}\norm{d^2x}dy.
			\end{split}	
		\end{equation}
	\end{lemma}
	\begin{proof}
		We rely on the fact that the following integrals are absolutely convergent
		\[
		\int_{A\times B}\norm{x-y}^p\lambda(dx)\nu(dy)
		\]
		as soon as $p>-3$ and $A=\frac12\D$, $B=\D\setminus\frac12\D$; $p>-2$ and $A=(-1,0)$, $B=(0,1)$; $p>-3$ and $A=\D\cap\H$, $B=(1,2)$. Here $\lambda(dx)$ and $\nu(dx)$ are either $\norm{dx}^2$ or $dx$. 
		The statement then boils down to proving that if we take $\alpha_1=\alpha_2=\gamma$ then the exponents that appear in items $(1)$, $(2)$ and $(3)$ in Lemma~\ref{lemma:fusion} are respectively greater than $-2$, $-1$ and $-2$. The first two assumptions are equivalent and it is readily seen that when $\gamma<\sqrt {\frac43}$ the exponent in Equation~\eqref{eq:fusion_hh} is given by $-\gamma^2>-\frac43>-2$ while for $\gamma\geq\sqrt{\frac43}$ it is given by $\frac12(\frac{\gamma^2}4-6+\frac4{\gamma^2})\geq\frac12(\frac13-3)>-2$. The third one is also proved in the same way by treating separately the cases $\gamma<\sqrt {\frac45}$ (with exponent $-\frac32\gamma^2>-2$), $\sqrt {\frac45}\leq\gamma<\sqrt {\frac43}$ (with exponent $\frac14(\frac{\gamma^2}{4}-10+\frac4{\gamma^2})\geq \frac14(\frac15-5)>-2$) and $\gamma\geq\sqrt{\frac43}$ (with exponent $\frac{5\gamma^2}{8}-4+\frac2{\gamma^2}\geq\frac56-4+\frac32>-2$).
	\end{proof}
	
	\subsubsection{KPZ identity}
	Thanks to these properties of the correlation functions we are now in position to prove the following so-called \textit{KPZ identity}~\cite{KPZ}:
	\begin{lemma}
		For $\bm\alpha\in\mc A_{N,M}$ we have
		\begin{equation}\label{eq:KPZ}
			\begin{split}
				&\left(\sum_{k=1}^{2N+M}\frac{\alpha_k}2-Q\right)\ps{\prod_{k=1}^NV_{\alpha_k}(z_k)\prod_{l=1}^MV_{\beta_l}(s_l)}_{\delta,\eps}=\\
				&-\mu\gamma \int_{\Heps} \ps{V_{\gamma}(x)\prod_{k=1}^NV_{\alpha_k}(z_k)\prod_{l=1}^MV_{\beta_l}(s_l)}_{\delta,\eps}\norm{dx}^2\\
				&-\frac\gamma2\int_{\Reps}\ps{V_\gamma(x)\prod_{k=1}^NV_{\alpha_k}(z_k)\prod_{l=1}^MV_{\beta_l}(s_l)}_{\delta,\eps,}\mu_{\partial}(dx)
			\end{split}
		\end{equation}
		where recall that $\ps{\prod_{k=1}^NV_{\alpha_k}(z_k)\prod_{l=1}^MV_{\beta_l}(s_l)}_{\delta,\eps}\coloneqq\lim\limits_{\rho\to0}\ps{\prod_{k=1}^NV_{\alpha_k}(z_k)\prod_{l=1}^MV_{\beta_l}(s_l)}_{\delta,\eps,\rho}$.
	\end{lemma}
	\begin{proof}
		First of all we see that all the terms that appear above are well-defined in virtue of Lemma~\ref{lemma:inf_integrability}. In the case where we assume that $\bm\alpha$ satisfies the bounds of Equation~\eqref{eq:seiberg} we get by a change of variable in the variable $\bm c$ that the map over $\R$ defined by
		\[
		h\mapsto\int_{\R} e^{\bm s(\bm c+h)}\expect{\exp\left(-\mu e^{\gamma(\bm c+h)}\mc A_{\delta,\eps}+e^{\frac\gamma2(\bm c+h)}\mc L_{\delta,\eps}\right)}d\bm c
		\]
		is constant. By differentiating with respect to $h$ and evaluating at $h=0$ we obtain
		\[
		\int_{\R} \left(\bm s-\mu \gamma e^{\gamma\bm c}\mc A_{\delta,\eps}-\frac\gamma2 e^{\frac\gamma2\bm c}\mc L_{\delta,\eps}\right) e^{\bm s\bm c}\expect{\exp\left(-\mu e^{\gamma\bm c}\mc A_{\delta,\eps}+e^{\frac\gamma2\bm c}\mc L_{\delta,\eps}\right)}d\bm c=0.
		\]
		The differentiation is justified by absolute convergence in the $\bm c$ variable of the above as explained in the derivation of the bounds~\eqref{eq:seiberg}.
		One then checks that 
		\begin{align*}
			&P(\bm z,\bm \alpha)\int_{\R} \mc A_{\delta,\eps} e^{(\bm s+\gamma)\bm c}\expect{\exp\left(-\mu e^{\gamma\bm c}\mc A_{\delta,\eps}+e^{\frac\gamma2\bm c}\mc L_{\delta,\eps}\right)}d\bm c\\
			&=\int_{\Heps} \ps{V_{\gamma}(x)\prod_{k=1}^NV_{\alpha_k}(z_k)\prod_{l=1}^MV_{\beta_l}(s_l)}_{\delta,\eps}\norm{dx}^2
		\end{align*}
		and likewise for the boundary measure (to see why this last equality is true one can go back to the regularized level with $\rho>0$ and then take a limit, which is uniform for $x$ ranging over $\bar\H$ in virtue of Lemma~\ref{lemma:inf_integrability}).
		
		When we no longer assume that $\bm\alpha$ satisfies the bounds of Equation~\eqref{eq:seiberg} we can still proceed in the exact same way by considering the constant map
		\[
		h\mapsto\int_{\R} e^{\bm s(\bm c+h)}\expect{\exp\left(-\mu e^{\gamma(\bm c+h)}\mc A_{\delta,\eps}+e^{\frac\gamma2(\bm c+h)}\mc L_{\delta,\eps}\right)-(\mathds{1}_{\bm s<0}-e^{\frac\gamma2(\bm c+h)}\mc L_{\delta,\eps}\mathds{1}_{\bm s<-\frac\gamma2})}d\bm c.
		\]
		If $0<\bm s<-\frac\gamma2$ then the proof remains the same, the correlation functions featuring an additional Vertex Operator $V_\gamma$ being such that $\bm s+\gamma>0$ and $\bm s+\frac\gamma2>0$ and therefore being defined without remainder term. If we assume that $-\frac\gamma2>\bm s>-\gamma$ then differentiating with respect to $h$ yields an additional term, and we will instead write that
		\begin{align*}
			&\int_{\R} \left(\bm s-\mu \gamma e^{\gamma\bm c}\mc A_{\delta,\eps}\right) e^{\bm s\bm c}\expect{\exp\left(-\mu e^{\gamma\bm c}\mc A_{\delta,\eps}+e^{\frac\gamma2\bm c}\mc L_{\delta,\eps}\right)}d\bm c\\
			&-\frac\gamma2\int_{\R}  \mc L_{\delta,\eps} e^{(\bm s+\frac\gamma2)\bm c}\expect{\exp\left(-\mu e^{\gamma\bm c}\mc A_{\delta,\eps}+e^{\frac\gamma2\bm c}\mc L_{\delta,\eps}\right)-1}d\bm c=0.
		\end{align*}
		We conclude in the same fashion as before since this last quantity is found to be equal to
		\[
		\int_{\Reps} \ps{V_{\gamma}(x)\prod_{k=1}^NV_{\alpha_k}(z_k)\prod_{l=1}^MV_{\beta_l}(s_l)}_{\delta,\eps}\mu_\partial(dx).
		\]
	\end{proof}
	As is now standard in the probabilistic approach to Liouville theory~\cite{KRV_loc,Oi19, Toda_OPEWV}, this KPZ identity is key in order to remove, when taking derivatives of the correlation functions, the metric-dependent terms of the form $\norm{\cdot}_+$ in the covariance kernel of the GFF.
	
	\subsection{Gaussian integration by parts}
	In order to make sense of the descendant fields we will need to use Malliavin calculus for the Gaussian Free Field, which translates as Gaussian integration by parts. It is based on the consideration of the derivatives of the Liouville field $\partial_z\Phi(z)$, and that correspond to a holomorphic derivative when $z\in\H$ and to a real derivative when $z\in\R$. The field $\Phi$ being only a distribution to make sense of these derivatives we first need to consider the regularization $\X_\rho$~\eqref{eq:regular_X} of the GFF $\X$, in which case such derivatives are well-defined. 
	
	\subsubsection{Singularity away from the insertions}
	To start with we consider the case where the point in $\bar\H$ on which we evaluate the derivatives of $\Phi$ are away from the singularities. In that case we have the following:
	\begin{lemma}\label{lemma:GaussianIPP}
		Let $p$ be a positive integer and take $\bm\alpha\in\mc A_{N,M}$. Then for $t\in(s_1-\eps,s_1+\eps)\setminus\{s_1\}$ we have
		\begin{equation}
			\begin{split}
				&\lim\limits_{\rho\to0}\Big\langle\frac{\partial_{t}^p\Phi(t)}{(p-1)!}\prod_{k=1}^NV_{\alpha_k}(z_k)\prod_{l=1}^MV_{\beta_l}(s_l)\Big\rangle_{\delta,\eps,\rho}=\sum_{k=1}^{ 2N+M}\frac{\alpha_k}{2(z_k-t)^p}\ps{\V}_{\eps,\delta}\\
				& -\mu
				\int_{\Heps'}\left(\frac{\gamma}{2(x-t)^p}+\frac{\gamma}{2(\bar x-t)^p}\right)\ps{V_\gamma(x)\V}_{\eps,\delta}\norm{dx}^2-
				\int_{\Reps'}\frac{\gamma}{2(x-t)^p}\ps{V_{\gamma}(x)\V}_{\eps,\delta}\mu_{\partial}(dx).
			\end{split}
		\end{equation}
		In particular the map $t\mapsto \Big\langle\Phi(t)\prod_{k=1}^NV_{\alpha_k}(z_k)\prod_{l=1}^MV_{\beta_l}(s_l)\Big\rangle_{\delta,\eps,\rho}$ is $p$-times differentiable in a neighborhood of $s_1$. Moreover the same statement remains valid if one rather considers $t$ in a neighborhood of a bulk insertion.
	\end{lemma}
	\begin{remark}
		The assumption that $t$ lies in the neighborhood of a singularity stems from the fact that we want to prevent the integrals involved to be singular. By adding a \lq\lq fake" singularity with weight zero (but which has the effect of regularizing the domains around this singularity)  we see that this statement readily applies if the domains $\Heps$ and $\Reps$ are changed to $\Heps\setminus B(t,\eps)$ and $\Reps\setminus B(t,\eps)$.
	\end{remark}
	\begin{proof}
		By definition of the Liouville field~\eqref{eq:liouville_field_ana} we can write that
		\begin{align*}
			&\ps{\partial_t^p\Phi(t)\prod_{k=1}^NV_{\alpha_k}(z_k)\prod_{l=1}^MV_{\beta_l}(s_l)}_{\delta,\eps,\rho}=\partial^p_t\left(H(t)-2Q\ln \norm{t}_+\right)\ps{\prod_{k=1}^NV_{\alpha_k}(z_k)\prod_{l=1}^MV_{\beta_l}(s_l)}_{\delta,\eps,\rho}\\
			&+P(\bm z,\bm \alpha)\int_{\R} e^{\bm s\bm c}\expect{\partial^p_t\X_\rho(t)\left(\exp\left(-\mu e^{\gamma\bm c}\mc A_{\delta,\eps,\rho}-e^{\frac\gamma2 \bm c}\mc L_{\delta,\eps,\rho}\right)-\mathfrak{R}_{\bm\alpha}(\bm c)\right)}d\bm c.
		\end{align*}
		The term that appears in the second line can be treated using Gaussian integration by parts, namely the fact that for a centered Gaussian vector $(X,Y_1,\dots,Y_N)$ and a smooth function on $\R^N$ with bounded derivatives $f$ we have
		\begin{equation*}\label{eq:Gauss_IPP}
			\E\left[Xf(Y_1,\dots,Y_N)\right]=\sum_{k=1}^N\E\left[XY_k\right]\E\left[\partial_{Y_k}f(Y_1,\dots,Y_N)\right].
		\end{equation*}
		From this general property we can deduce that
		\begin{align*}
			\expect{\partial^p_t\X_\rho(t)\left(\exp\left(-\mu e^{\gamma\bm c}\mc A_{\delta,\eps,\rho}-e^{\frac\gamma2 \bm c}\mc L_{\delta,\eps,\rho}\right)-\mathfrak{R}_{\bm\alpha}(\bm c)\right)}&\\
			=-\mu\gamma e^{\gamma\bm c}\int_{\Heps}\partial^p_t\expect{\X_\rho(t)\X_\rho(x)}\prod_{k=1}^{2N+M}\left(\frac{\norm{x}_+}{\norm{x-z_k}}\right)^{\gamma\alpha_k}&\\
			&\hspace{-5cm}\expect{\rho^{\frac{\gamma^2}{2}}e^{\gamma\X_\rho(x)}\exp\left(-\mu e^{\gamma\bm c}\mc A_{\delta,\eps,\rho}-e^{\frac\gamma2 \bm c}\mc L_{\delta,\eps,\rho}\right)} \norm{dx}^2\\
			-\frac\gamma2e^{\frac\gamma2\bm c}\int_{\Reps}\partial^p_t\expect{\X_\rho(t)\X_\rho(x)}\prod_{k=1}^{2N+M}\left(\frac{\norm{x}_+}{\norm{x-z_k}}\right)^{\frac{\gamma\alpha_k}2}&\\
			&\hspace{-5cm}\expect{\rho^{\frac{\gamma^2}{4}}e^{\frac\gamma2\X_\rho(x)}\left(\exp\left(-\mu e^{\gamma\bm c}\mc A_{\delta,\eps,\rho}-e^{\frac\gamma2 \bm c}\mc L_{\delta,\eps,\rho}\right)-\mathds{1}_{\bm s<-\frac\gamma2}\right)}\mu_{\partial}(dx).
		\end{align*}
		
		Like in the proof of the KPZ identity (Lemma~\eqref{eq:KPZ} we first use absolute convergence of the integrals involved (via Lemma~\ref{lemma:inf_integrability}) to switch the integrals and rewrite the latter as
		\begin{align*}
			&\ps{\partial_t^p\Phi(t)\prod_{k=1}^NV_{\alpha_k}(z_k)\prod_{l=1}^MV_{\beta_l}(s_l)}_{\delta,\eps,\rho}=\partial_t^p\left(H(t)-2Q\ln \norm{t}_+\right)\ps{\prod_{k=1}^NV_{\alpha_k}(z_k)\prod_{l=1}^MV_{\beta_l}(s_l)}_{\delta,\eps,\rho}\\
			&-\mu\gamma
			\int_{\Heps}\partial^p_tG_\rho(x,t)\ps{V_\gamma(x)\V}_{\eps,\delta}\norm{dx}^2-\frac\gamma2
			\int_{\Reps}\partial^p_tG_\rho(x,t)\ps{V_{\gamma}(x)\V}_{\eps,\delta}\mu_{\partial}(dx)
		\end{align*}
		where $G_\rho(x,y)\coloneqq \int_{\H^2}G(w,z)\eta_\rho(x-w)\eta_\rho(y-z)\norm{dw}^2\norm{dz}^2$.
		We can then take the $\rho\to0$ limit and get thanks to Lemma~\ref{lemma:inf_integrability} the same equality with $\rho=0$.
		
		Now thanks to the explicit expression of $G$ we further have
		\begin{equation*}
			\begin{split}
				&\lim\limits_{\rho\to0}\Big\langle\frac{\partial_t^p\Phi(t)}{(p-1)!}\prod_{k=1}^NV_{\alpha_k}(z_k)\prod_{l=1}^MV_{\beta_l}(s_l)\Big\rangle_{\delta,\eps,\rho}=\sum_{k=1}^{2N+M}\frac{\alpha_k}{2(z_k-t)^p}\ps{\V}_{\eps,\delta}\\
				& -\mu
				\int_{\Heps}\left(\frac{\gamma}{2(x-t)^p}+\frac{\gamma}{2(\bar x-t)^p}\right)\ps{V_\gamma(x)\V}_{\eps,\delta}\norm{dx}^2-
				\int_{\Reps}\frac{\gamma}{2(x-t)^p}\ps{V_{\gamma}(x)\V}_{\eps,\delta}\mu_{\partial}(dx)\\
				&+\left(\sum_{k=1}^{2N+M}(\alpha_k-2Q)\ps{\V}_{\delta,\eps}-2\mu\gamma
				\int_{\Heps}\ps{V_\gamma(x)\V}_{\eps,\delta}\norm{dx}^2-\gamma				\int_{\Reps}\ps{V_{\gamma}(x)\V}_{\eps,\delta}\mu_{\partial}(dx)\right)\partial_t^p\ln\norm{t}.
			\end{split}
		\end{equation*}
		The last line vanishes thanks to the KPZ identity~\eqref{eq:KPZ}.
	\end{proof}
	
	A similar statement holds true for products of derivatives of the field, and for this to make sense we need to interpret the products as Wick products. In the case we will focus on here we set
	\[
	:XY:=XY-\expect {X Y},
	\] 
	but more general expressions can be defined recursively.
	The general form of Gaussian integration by parts is then derived from Proposition~\ref{lemma:GaussianIPP} and takes the form:
	\begin{equation}\label{eq:IPP_product}
		\begin{split}
			&\lim\limits_{\rho\to0}\frac{1}{(p_1-1)!}\Big\langle:\prod_{i=1}^m\partial^{p_i}_t\Phi(t):\prod_{k=1}^NV_{\alpha_k}(z_k)\prod_{l=1}^MV_{\beta_l}(s_l)\Big\rangle_{\delta,\eps,\rho}\\
			&=\sum_{k=1}^{2N+M}\frac{\alpha_k}{2(z_k-t)^{p_1}}\ps{:\prod_{l=2}^m\partial^{p_l}_t\Phi(t):\V}_{\eps,\delta}\\
			& -\mu
			\int_{\Heps}\left(\frac{\gamma}{2(x-t)^{p_1}}+\frac{\gamma}{2(\bar x-t)^{p_1}}\right)\ps{:\prod_{l=2}^m\partial^{p_l}_t\Phi(t):V_\gamma(x)\V}_{\eps,\delta}\norm{dx}^2\\
			&-
			\int_{\Reps}\frac{\gamma}{2(x-t)^{p_1}}\ps{:\prod_{l=2}^m\partial^{p_l}_t\Phi(t):V_{\gamma}(x)\V}_{\eps,\delta}\mu_{\partial}(dx).
		\end{split}
	\end{equation}
	valid for $p_1,\cdots,p_m$ positive integers (though we will use here only the $m=2$ case). 
	
	\subsubsection{Gaussian integration by parts at an insertion}
	We will also consider the case where $t$ actually coincides with an insertion. In that case we have to slightly adapt the definition of the corresponding quantity by setting
	\begin{equation}
		\begin{split}
			&\Big\langle :\prod_{i=1}^m\partial^{p_i}_{s_1}\Phi(s_1)V_{\beta_1}(s_1):\prod_{k=1}^NV_{\alpha_k}(z_k)\prod_{l=2}^MV_{\beta_l}(s_l)\Big\rangle_{\delta,\eps,\rho}\coloneqq P(\bm z,\bm \alpha)\int_{\R} e^{\bm s\bm c}\\
			&\E\Big[F\left(\X_\rho-2Q\ln \norm{\cdot}_++ H^{(2N+1)}+\bm c\right)\left(\exp\left(-\mu e^{\gamma\bm c}\mc A_{\delta,\eps}-e^{\frac\gamma2 \bm c}\mc L_{\delta,\eps}\right)-\mathfrak{R}_{\bm\alpha}(\bm c)\right)\Big]d\bm c
		\end{split}
	\end{equation}
	where $F(\Phi)=:\prod_{i=1}^m\partial^{p_i}_{s_1}\Phi(s_1):$ with the Wick products defined as above, and 
	\[
	H^{(k)}(x)=H(x)+\alpha_k\ln\norm{x-z_k}.
	\]
	In that case it is readily checked that Lemma~\ref{lemma:GaussianIPP} and Equation~\eqref{eq:IPP_product} remain valid for this defining expression.
	
	%%%%%%%%%%%%%%%%%%%%%%%%%%%%%%%%%%%%%%%%%%%%%%%%
	%%%%%%%%%%%%%%%%%%%%%%%%%%%%%%%%%%%%%%%%%%%%%%%%
	%%%%%%%%%%%%%%%%%%%%%%%%%%%%%%%%%%%%%%%%%%%%%%%%
	
	%%%%%%%%%%%%%%%%%%%%%%%%%%%%%%%%%%%
	%%%%%%%%%%%%%%%%%%%%%%%%%%%%%%%%%%%
	%%%%%%%%%%%%%%%%%%%%%%%%%%%%%%%%%%%
	
	\section{Rigorous derivation of the Ward identities}\label{sec:ward}
	To make sense of these Ward identities we first define the descendant fields at the first order that allow to make sense of (weak) derivatives of the correlation functions. We then introduce a descendant at the order two and show that it is related to the local Ward identities, which we also prove to hold.
	
	In order to define such descendants, we will rely on a regularization procedure by first defining regularized correlation functions containing descendant fields. As we will see, these regularizations may feature some divergent terms: these terms are taken care of by substracting \lq\lq remainder terms" from the regularized descendants. The descendants are then defined to be the limit obtained by this procedure.
	
	These remainder terms will be obtained through application of Stokes' formula: to be more specific the definition of the regularized correlation functions with descendant fields features singular integrals similar to that that arise in Equation~\eqref{eq:IPP_product} above. In order to understand precisely at which rate these singular terms scale when we take the $\rho,\eps,\delta\to0$ limits we will use Stokes' formula to transform this term into more tractable expressions for which we can compute explicitly the rate of divergence. And as we will see the higher order the descendant, the stronger the divergence.
	
	\subsection{Descendant fields at the first order and derivatives}\label{subsec:der1}
	The statement of the local Ward identity associated to $\SET$ is made in terms of descendant fields of the Vertex Operators, that is quantities of the form $\ps{\L_{-1}V_{\alpha_k}(z_k)\prod_{i\neq k}V_{\alpha_i}(z_i)\prod_{l=1}^MV_{\beta_l}(s_l)}$. 
	Our goal here is to define such objects. We recall the shorthand $\V$ for $\prod_{k=1}^NV_{\alpha_k}(z_k)\prod_{l=1}^MV_{\beta_l}(s_l)$.
	
	\subsubsection{Bulk Vertex Operator}
	We first consider a bulk Vertex Operator, that is one whose corresponding insertion $z_k$ belongs to $\H$. 
	\begin{lemma}\label{lemma:desc1_bulk}
		As $\rho,\eps,\delta\to0$ (meaning like before that first $\rho$, then $\eps$ and finally $\delta$ are taken to $0$), the following limit exists and is finite:
		\begin{equation}
			\ps{\L_{-1}V_{\alpha_1}(z_1)\prod_{k=2}^NV_{\alpha_k}(z_k)\prod_{l=1}^MV_{\beta_l}(s_l)}\coloneqq\lim\limits_{\delta\to0}\lim\limits_{\eps\to0}\lim\limits_{\rho\to0} \partial_{z_1}\ps{\prod_{k=1}^NV_{\alpha_k}(z_k)\prod_{l=1}^MV_{\beta_l}(s_l)}_{\delta,\eps,\rho}.
		\end{equation}
		Moreover the map $\bm\alpha\mapsto\ps{\L_{-1}V_{\alpha_1}(z_1)\prod_{k=2}^NV_{\alpha_k}(z_k)\prod_{l=1}^MV_{\beta_l}(s_l)}$ is analytic in a complex neighborhood of $\mc A_{N,M}$.
	\end{lemma}
	\begin{proof}
		The computations parallel the ones conducted in the case of the sphere in~\cite{KRV_loc,Oi19} so we will be brief. To start with we compute
		\begin{equation}\label{eq:der1}
			\begin{split}
				&\lim\limits_{\rho\to0}\partial_{z_1}\ps{\prod_{k=1}^NV_{\alpha_k}(z_k)\prod_{l=1}^MV_{\beta_l}(s_l)}_{\delta,\eps,\rho}=\sum_{k=2}^{2N+M}\frac{\alpha_k\alpha_1}{2(z_k-z_1)}\ps{\V}_{\delta,\eps}\\
				&-\mu\int_{ \Heps}\left(\frac{\gamma \alpha_1}{2(x-z_1)}+\frac{\gamma\alpha_1}{2(\bar x-z_1)}\right)\ps{V_{\gamma}(x)\V}_{\delta,\eps}\norm{dx}^2\\
				&-\int_{\R_{\eps}}\frac{\gamma  \alpha_1}{2(x-z_1)}\ps{V_{\gamma}(x)\V}_{\delta,\eps}\mu_{\partial}(dx).
			\end{split}
		\end{equation} 
		This can be seen by noting that for $\rho>0$, $$\partial_{z_1}\ps{\prod_{k=1}^NV_{\alpha_k}(z_k)\prod_{l=1}^MV_{\beta_l}(s_l)}_{\delta,\eps,\rho}=\ps{:\alpha_1\partial\Phi(z_1)V_{\alpha_1}(z_1):\prod_{k=2}^NV_{\alpha_k}(z_k)\prod_{l=1}^MV_{\beta_l}(s_l)}_{\delta,\eps,\rho}$$ so that we are in the setting of Lemma~\ref{lemma:GaussianIPP}. Alternatively this can be obtained by differentiating the defining expression~\eqref{eq:ana_correl} for the correlation functions: since the integration domains avoid the singularities at the $x=z_k$ we can rely on Lemma~\ref{lemma:inf_integrability} to infer absolute convergence of the above integrals at the $\rho>0$ level and subsequently take the $\rho\to0$ limit.
		
		Then we have to take care of what happens when the integrand has a singularity at $x=z_1$. For this we split the integrals in different pieces: a part away from the singular point where the singularity $\frac{1}{x-z_1}$ remains bounded and one around the singular point. 
		Namely let us set for a fixed $u>0$ small enough and independent of $\delta,\eps$,  $\H_{u}\coloneqq \H\setminus B(z_1,u)$, so that $\Heps$ is the disjoint union of  $\Heps\cap \H_u$ and $A(z_1;\eps,u)\coloneqq\left\{z:\eps<\norm{z-z_k}<u\right\}$. This allows to write
		\begin{equation}\label{eq:to_prove1}
			\begin{split}
				&\int_{\Heps}\frac{\gamma \alpha_1}{2(x-z_1)}\ps{V_{\gamma}(x)\V}_{\delta,\eps}\norm{dx}^2=\int_{\Heps\cap \H_u}\frac{\gamma \alpha_1}{2(x-z_1)}\ps{V_{\gamma}(x)\V}_{\delta,\eps}\norm{dx}^2\\
				&+\int_{A(z_1;\eps,u)}\left(-\partial_x+\sum_{k=2}^{2N+M}\frac{\gamma\alpha_k}{2(z_k-x)}\right)\ps{V_{\gamma}(x)\V}_{\delta,\eps}\norm{dx}^2\\
				&-\mu\int_{\Heps\times A(z_1;\eps,u)}\left(\frac{\gamma^2}{4(y-x)}+\frac{\gamma^2}{4(\bar y-x)}\right)\ps{V_{\gamma}(x)V_{\gamma}(y)\V}_{\delta,\eps}\norm{dx}^2\norm{dy}^2\\
				&-\int_{\Reps\times A(z_1;\eps,u)}\frac{\gamma^2}{4(y-x)}\ps{V_{\gamma}(x)V_{\gamma}(y)\V}_{\delta,\eps}\norm{dx}^2\mu_{\partial}(dy).
			\end{split}
		\end{equation}
		where we have used the KPZ identity~\eqref{eq:KPZ} to remove the metric-dependent terms.
		The second term in the right-hand side is treated by using Stokes' formula:
		\begin{align*}
			&\int_{A(z_1;\eps,u)}\left(-\partial_x+\sum_{k=2}^{2N+M}\frac{\gamma\alpha_k}{2(z_k-x)}\right)\ps{V_{\gamma}(x)\V}_{\delta,\eps}\norm{dx}^2=\int_{\partial B(z_1,\eps)}\ps{V_{\gamma}(x)\V}\frac{i d\bar x}2\\
			&-\int_{\partial B(z_1,u)}\ps{V_{\gamma}(x)\V}\frac{i d\bar x}2+\int_{B(z_1,u)}\sum_{k=2}^{2N+M}\frac{\gamma\alpha_k}{2(z_k-x)}\ps{V_{\gamma}(x)\V}\norm{dx}^2.
		\end{align*}
		In the limit where $\eps\to0$ the term $\int_{\partial B(z_1,\eps)}\ps{V_{\gamma}(x)\V}_{\delta,\eps}d\bar x$ vanishes via the fusion asymptotics from Lemma~\ref{lemma:fusion_integrability}. As for the third integral in Equation~\eqref{eq:to_prove1} we split the domain $\Heps\times A(z_1;\eps,u)$ as $\left((\Heps\cap\H_u)\times A(z_1;\eps,u)\right)\cup A(z_1;\eps,u)^2$ and rely on symmetry in the $x,y$ variables to infer that the integral over $A(z_1;\eps,u)^2$ is equal to $0$. Hence in the end we are left with
		\begin{align*}
			&\lim\limits_{\delta,\eps\to0}\int_{\Heps}\frac{\gamma \alpha_1}{2(x-z_1)}\ps{V_{\gamma}(x)\V}_{\delta,\eps}\norm{dx}^2=\int_{ \H_u}\frac{\gamma \alpha_1}{2(x-z_1)}\ps{V_{\gamma}(x)\V}\norm{dx}^2\\
			&-\int_{\partial B(z_1,u)}\ps{V_{\gamma}(x)\V}\frac{i d\bar x}2+\int_{B(z_1,u)}\sum_{k=2}^{2N+M}\frac{\gamma\alpha_k}{2(z_k-x)}\ps{V_{\gamma}(x)\V}\norm{dx}^2\\
			&-\mu\int_{\H\setminus B(z_1,u)\times B(z_1,u)}\left(\frac{\gamma^2}{4(y-x)}+\frac{\gamma^2}{4(\bar y-x)}\right)\ps{V_{\gamma}(x)V_{\gamma}(y)\V}\norm{dx}^2\norm{dy}^2\\
			&-\int_{\R\times B(z_1,u)}\frac{\gamma^2}{4(y-x)}\ps{V_{\gamma}(x)V_{\gamma}(y)\V}\norm{dx}^2\mu_{\partial}(dy).
		\end{align*}
		These quantities are finite thanks to Lemma~\ref{lemma:fusion_integrability}. Moreover all these integrals are seen to be analytic in $\bm\alpha\in\mc A_{N,M}$ thanks to Proposition~\ref{prop:ana_correl} together with Lemma~\ref{lemma:fusion_integrability} for the fourth one and Lemma~\ref{lemma:inf_integrability} for all the other ones. 
		
		The other terms that appear in the expression~\eqref{eq:der1} of the derivative are dealt with in a similar way. Namely the singularity $\frac{\gamma\alpha_1}{2(\bar x-z_1)}$ in the integral over $\Heps$ is treated in the exact same way while the integral over $\Reps$ remains away from the singularity at $x=z_1$ so is readily taken care of.
	\end{proof}
	
	\subsubsection{Boundary Vertex Operator}
	The case of a boundary Vertex Operator is more involved, since in some cases the above limit may actually diverge. This is due to the fact that in the same spirit as the above proof, we use Stokes' formula to remove singularities within the integrals. However unlike in the bulk case the associated remainder term will \textit{not} always converge to $0$ and are actually divergent. Still we can control the rate at which this quantity diverges as the following statement explains:
	\begin{lemma}\label{lemma:desc_L}
		As $\rho,\eps,\delta\to0$, the following limit exists and is finite:
		\begin{equation}
			\begin{split}
				&\ps{\L_{-1}V_{\beta_1}(s_1)\prod_{k=1}^NV_{\alpha_k}(z_k)\prod_{l=2}^MV_{\beta_l}(s_l)}\coloneqq\lim\limits_{\delta\to0}\lim\limits_{\eps\to0}\lim\limits_{\rho\to0}\\
				&\partial_{s_1}\ps{\prod_{k=1}^NV_{\alpha_k}(z_k)\prod_{l=1}^MV_{\beta_l}(s_l)}_{\delta,\eps,\rho}-\Big(\ps{V_{\gamma}(s_1-\eps)\V}_{\delta,\eps,\rho}\mu_{1}-\ps{V_{\gamma}(s_1+\eps)\V}_{\delta,\eps,\rho}\mu_{2}\Big).
			\end{split}
		\end{equation}
		Moreover the map $\bm\alpha\mapsto\ps{\L_{-1}V_{\beta_1}(s_1)\prod_{k=1}^NV_{\alpha_k}(z_k)\prod_{l=2}^MV_{\beta_l}(s_l)}$ is analytic in a complex neighborhood of $\mc A_{N,M}$.
	\end{lemma}
	\begin{proof}
		We proceed like before but this time we have to use Stokes' formula both in the bulk and on the boundary. Namely in the same fashion as in the proof of Lemma~\ref{lemma:desc1_bulk} above:
		\begin{align*}
			&\lim\limits_{\rho\to0}\partial_{s_1}\ps{\prod_{k=1}^NV_{\alpha_k}(z_k)\prod_{l=1}^MV_{\beta_l}(s_l)}_{\delta,\eps,\rho}=\sum_{k\neq 2N+1}\frac{\alpha_k\beta_1}{2(z_k-s_1)}\ps{\V}_{\delta,\eps}\\
			&-\mu\int_{ \Heps}\left(\frac{\gamma \beta_1}{2(x-s_1)}+\frac{\gamma\beta_1}{2(\bar x-s_1)}\right)\ps{V_{\gamma}(x)\V}_{\delta,\eps}\norm{dx}^2\\
			&-\int_{\R_{\eps}}\frac{\gamma  \beta_1}{2(x-s_1)}\ps{V_{\gamma}(x)\V}_{\delta,\eps}\mu_{\partial}(dx).
		\end{align*}
		There may be singularities as $\H$\reflectbox{$\in$}$x\to s_1$ in the bulk measure and as $\R$\reflectbox{$\in$}$x\to s_1$ in the boundary one. We start by treating the bulk measure, and like before we set $\H_{u}\coloneqq \H\setminus B(z_1,u)$ for $u$ small enough, so that $\Heps$ is the disjoint union of  $\Heps\cap \H_u$ and $\Heps\cap B(z_1,u)$. We get
		\begin{align*}
			&\int_{ \Heps}\left(\frac{\gamma \beta_1}{2(x-s_1)}+\frac{\gamma\beta_1}{2(\bar x-s_1)}\right)\ps{V_{\gamma}(x)\V}_{\delta,\eps}\norm{dx}^2=\int_{ \H_u}\left(\frac{\gamma \beta_1}{2(x-s_1)}+\frac{\gamma\beta_1}{2(\bar x-s_1)}\right)\ps{V_{\gamma}(x)\V}_{\delta,\eps}\norm{dx}^2\\
			&-\int_{ \Heps\cap B(s_1,u)}\left(\partial_x+\partial_{\bar x}\right)\ps{V_{\gamma}(x)\V}_{\delta,\eps}+\sum_{k\neq 2N+1}\left(\frac{\gamma \alpha_k}{2(x-z_k)}+\frac{\gamma\alpha_k}{2(\bar x-z_k)}\right)\ps{V_{\gamma}(x)\V}_{\delta,\eps}\norm{dx}^2\\
			&+\mu\int_{ \Heps\cap B(s_1,u)\times\Heps}\left(\frac{\gamma^2}{2(x-y)}+\frac{\gamma^2}{2(x-\bar y)}+\frac{\gamma^2}{2(\bar x-y)}+\frac{\gamma^2}{2(\bar x-\bar y)}\right)\ps{V_{\gamma}(x)V_\gamma(y)\V}_{\delta,\eps}\norm{dx}^2\norm{dy}^2\\
			&+\int_{ \Heps\cap B(s_1,u)\times\Reps}\left(\frac{\gamma^2}{2(x-y)}+\frac{\gamma^2}{2(x-\bar y)}\right)\ps{V_{\gamma}(x)V_{\gamma}(y)\V}_{\delta,\eps}\norm{dx}^2\mu_{\partial}(dy).
		\end{align*} 
		Like before by symmetry in the $x$ and $y$ variable the integral over $\Heps\cap B(s_1,u)\times\Heps$ vanishes over $\left(\Heps\cap B(s_1,u)\right)^2$, while the term containing derivatives is seen to be given by
		\begin{align*}
			-\int_{ \Heps\cap \partial B(s_1,u)}\ps{V_{\gamma}(\xi)\V}_{\delta,\eps}\frac{id\bar\xi-id\xi}{2}+\int_{(s_1-u,s_1+u)}\ps{V_{\gamma}(x+i\delta)\V}_{\delta,\eps}-\ps{V_{\gamma}(x-i\delta)\V}_{\delta,\eps}dx
		\end{align*}
		with the last integral being equal to $0$. To sum up this allows to write that 
		\begin{align*}
			&\int_{ \Heps}\left(\frac{\gamma \beta_1}{2(x-s_1)}+\frac{\gamma\beta_1}{2(\bar x-s_1)}\right)\ps{V_{\gamma}(x)\V}_{\delta,\eps}\norm{dx}^2\\
			&=\int_{ \H_u}\left(\frac{\gamma \beta_1}{2(x-s_1)}+\frac{\gamma\beta_1}{2(\bar x-s_1)}\right)\ps{V_{\gamma}(x)\V}\norm{dx}^2\\
			&-\int_{ \H\cap B(s_1,u)}\sum_{k\neq 2N+1}\left(\frac{\gamma \alpha_k}{2(x-z_k)}+\frac{\gamma\alpha_k}{2(\bar x-z_k)}\right)\ps{V_{\gamma}(x)\V}\norm{dx}^2\\
			&+\mu\int_{ \H\cap B(s_1,u)\times\H_u}\left(\frac{\gamma^2}{2(x-y)}+\frac{\gamma^2}{2(x-\bar y)}+\frac{\gamma^2}{2(\bar x-y)}+\frac{\gamma^2}{2(\bar x-\bar y)}\right)\ps{V_{\gamma}(x)V_\gamma(y)\V}\norm{dx}^2\norm{dy}^2\\
			&+\int_{ \H\cap B(s_1,u)\times\R_u}\left(\frac{\gamma^2}{2(x-y)}+\frac{\gamma^2}{2(x-\bar y)}\right)\ps{V_{\gamma}(x)V_{\gamma}(y)\V}\norm{dx}^2\mu_{\partial}(dy)\\
			&-\int_{ \H\cap \partial B(s_1,u)}\ps{V_{\gamma}(\xi)\V}\frac{id\bar\xi-id\xi}{2}\\
			&+\int_{ \Heps\cap B(s_1,u)\times\Reps\cap (s_1-u,s_1+u)}\left(\frac{\gamma^2}{2(x-y)}+\frac{\gamma^2}{2(x-\bar y)}\right)\ps{V_{\gamma}(x)V_{\gamma}(y)\V}\norm{dx}^2\mu_{\partial}(dy)
		\end{align*}
		up to a term that vanishes in the $\rho,\eps,\delta\to0$ limit, and where $\R_u\coloneqq\Reps\setminus(s_1-u,s_1+u)$. Thanks to the fusion asymptotics from Lemma~\ref{lemma:fusion_integrability} together with analycity of the correlation functions from Proposition~\ref{prop:analycity} we see that in this expression all terms except for the very last integral are well-defined and analytic in $\bm\alpha\in\mc A_{N,M}$. 
		
		Let us now look at the boundary measure. Around $s_1$ we have
		\begin{align*}
			&\int_{(s_1-u,s_1+u)\cap\Reps}\frac{\gamma  \beta_1}{2(x-s_1)}\ps{V_{\gamma}(x)\V}_{\delta,\eps}\mu_{\partial}(dx)\\
			&=\int_{(s_1-u,s_1+u)\cap\Reps}\left(-\partial_x-\sum_{k=1}^{2N+M}\frac{\gamma\alpha_k}{2(x-z_k)}\right)\ps{V_{\gamma}(x)\V}_{\delta,\eps}\mu_{\partial}(dx)\\
			&+\int_{(s_1-u,s_1+u)\cap\Reps\times\Reps}\frac{\gamma  ^2}{2(x-y)}\ps{V_{\gamma}(x)V_{\gamma}(y)\V}_{\delta,\eps}\mu_{\partial}(dx)\mu_{\partial}(dx)\\
			&+\mu\int_{(s_1-u,s_1+u)\cap\Reps\times\H_{\delta,\eps}}\frac{\gamma  ^2}{2}\left(\frac1{x-y}+\frac1{x-\bar y}\right)\ps{V_{\gamma}(x)V_{\gamma}(y)\V}_{\delta,\eps}\mu_{\partial}(dx)\norm{dy}^2.
		\end{align*}
		In the same fashion as for the bulk measure by symmetry in $x,y$ we can simplify
		\begin{align*}
			&\int_{(s_1-u,s_1+u)\cap\Reps\times\Reps}\frac{\gamma  ^2}{2(x-y)}\ps{V_{\gamma}(x)V_{\gamma}(y)\V}_{\delta,\eps}\mu_{\partial}(dx)\mu_{\partial}(dy)=\\
			&\int_{(s_1-u,s_1+u)\cap\Reps\times\Reps\setminus(s_1-u,s_1+u)}\frac{\gamma  ^2}{2(x-y)}\ps{V_{\gamma}(x)V_{\gamma}(y)\V}_{\delta,\eps}\mu_{\partial}(dx)\mu_{\partial}(dy)
		\end{align*}
		which is uniformly bounded in $\delta,\eps$ via Lemma~\ref{lemma:fusion_integrability}.	The integral over $(s_1-u,s_1+u)\cap\Reps\times\H_{\delta,\eps}$ can be split as	the disjoint union of $(s_1-u,s_1+u)\cap\Reps\times\H_u$ and $(s_1-u,s_1+u)\cap\Reps\times\H_{\delta,\eps}\cap B(s_1,u)$. This last term will cancel out the integral over $\H_{\delta,\eps}\cap B(s_1,u)\times(s_1-u,s_1+u)\cap\Reps$ that arose from the treatment of the bulk measure. As for the integral over $(s_1-u,s_1+u)\cap\Reps$ by integration by parts we get
		\begin{align*}
			&\int_{(s_1-u,s_1+u)\cap\Reps}\partial_x\ps{V_{\gamma}(x)\V}_{\delta,\eps}\mu_{\partial}(dx)=\Big(\ps{V_{\gamma}(s_1-\eps)\V}_{\delta,\eps}\mu_{l}-\ps{V_{\gamma}(s_1-u)\V}_{\delta,\eps}\mu_{l}\Big)\\
			&+\Big(\ps{V_{\gamma}(s_1+u)\V}_{\delta,\eps}\mu_{l+1}-\ps{V_{\gamma}(s_1+\eps)\V}_{\delta,\eps}\mu_{l+1}\Big).
		\end{align*}
		We recognize above the remainder term that shows up in the definition of $\L_{-1}V_{\beta_1}$. As a consequence we end up with
		\begin{align*}
			&-\int_{\Reps}\frac{\gamma  \beta_1}{2(x-s_1)}\ps{V_{\gamma}(x)\V}_{\delta,\eps}\mu_{\partial}(dx)-\Big(\ps{V_{\gamma}(s_1-\eps)\V}_{\delta,\eps}\mu_{l}-\ps{V_{\gamma}(s_1+\eps)\V}_{\delta,\eps}\mu_{l+1}\Big)\\
			&=-\int_{\R\setminus(s_1-u,s_1+u)}\frac{\gamma  \beta_1}{2(x-s_1)}\ps{V_{\gamma}(x)\V}\mu_{\partial}(dx)\\
			&+\int_{(s_1-u,s_1+u)}\sum_{k=1}^{2N+M}\frac{\gamma\alpha_k}{2(x-z_k)}\ps{V_{\gamma}(x)\V}\mu_{\partial}(dx)\\
			&-\int_{(s_1-u,s_1+u)\times\R\setminus(s_1-u,s_1+u)}\frac{\gamma  ^2}{2(x-y)}\ps{V_{\gamma}(x)V_{\gamma}(y)\V}\mu_{\partial}(dx)\mu_{\partial}(dx)\\
			&-\mu\int_{(s_1-u,s_1+u)\times\H\setminus B(s_1,u)}\frac{\gamma  ^2}{2}\left(\frac1{x-y}+\frac1{x-\bar y}\right)\ps{V_{\gamma}(x)V_{\gamma}(y)\V}\mu_{\partial}(dx)\norm{dy}^2\\
			&+\Big(\ps{V_{\gamma}(s_1-u)\V}\mu_{l}-\ps{V_{\gamma}(s_1+u)\V}\mu_{l+1}\Big)\\
			&-\mu\int_{(s_1-u,s_1+u)\times\H\cap B(s_1,u)}\frac{\gamma  ^2}{2}\left(\frac1{x-y}+\frac1{x-\bar y}\right)\ps{V_{\gamma}(x)V_{\gamma}(y)\V}\mu_{\partial}(dx)\norm{dy}^2.
		\end{align*}
		Like before using the fusion asymptotics from Lemma~\ref{lemma:fusion_integrability} together with Proposition~\ref{prop:ana_correl} we see that in the above all the terms expect for the last integral are well-defined and depend analytically on $\bm\alpha\in\mc A_{N,M}$. The proof is complete since this last integral cancels out with the one coming from the bulk measure when putting together all the terms that appear in the expression of $\partial_{s_1}\ps{\prod_{k=1}^NV_{\alpha_k}(z_k)\prod_{l=1}^MV_{\beta_l}(s_l)}_{\delta,\eps,\rho}$.
	\end{proof}
	
	As a first application of the definition of the descendant at the first order $\L_{-1}$ we can define the (weak) derivative of the correlation function with respect to a boundary insertion: 
	\begin{corollary}\label{cor:der}
		Assume that $\bm\alpha\in\mc A_{N,M}$. Then for any smooth and compactly supported test function $f$ on $\R$,
		\begin{equation}\label{eq:weak_der1}
			\int_\R f(s_1)\ps{\L_{-1}V_{\beta_1}(s_1)\prod_{k=1}^NV_{\alpha_k}(z_k)\prod_{l=2}^MV_{\beta_l}(s_l)}ds_1=-\int_\R\partial_{s_1}f(s_1)\ps{\prod_{k=1}^NV_{\alpha_k}(z_k)\prod_{l=1}^MV_{\beta_l}(s_l)}ds_1.
		\end{equation}
		Put differently $\ps{\L_{-1}V_{\beta_1}(s_1)\prod_{k=1}^NV_{\alpha_k}(z_k)\prod_{l=2}^MV_{\beta_l}(s_l)}$ is the weak derivative of the correlation function. The same applies for a bulk insertion.
	\end{corollary}
	\begin{proof}
		To start with we assume that $\beta_1<0$. In that case the remainder term that shows up in the definition of the $\L_{-1}$ descendant vanishes in the limit thanks to the second item of the fusion asymptotics from Lemma~\ref{lemma:fusion}. As a consequence we have the true equality
		\begin{align*}
			&\ps{\L_{-1}V_{\beta_1}(s_1)\prod_{k=1}^NV_{\alpha_k}(z_k)\prod_{l=2}^MV_{\beta_l}(s_l)}=\lim\limits_{\delta\to0}\lim\limits_{\eps\to0}\lim\limits_{\rho\to0}\partial_{s_1}\ps{\prod_{k=1}^NV_{\alpha_k}(z_k)\prod_{l=1}^MV_{\beta_l}(s_l)}_{\delta,\eps,\rho}
		\end{align*}
		and in particular Equation~\eqref{eq:weak_der1} holds true as soon as $\beta_1<0$. 
		
		Moreover we know from Lemma~\ref{lemma:desc_L} and Proposition~\ref{prop:analycity} that both the left and right-hand sides in Equation~\ref{eq:weak_der1} depend analytically in $\beta_1$. Therefore by uniqueness of the analytic continuation we infer the result as soon as the set of $\beta_1$ such that $\beta_1<0$ and $\bm\alpha\in\mc A_{N,M}$ contains an open subset of $\R$, which is easily seen. The same argument remains valid for a bulk insertion since in that case the associated remainder term is zero regardless of the value of $\alpha$.
	\end{proof}

	\subsection{A descendant field at the order two and Ward identities}\label{subsec:ward}
	We have defined above first order descendants via a limiting procedure based on derivatives of the regularized correlation function. We now turn to descendants at the order two, for which the procedure is more involved but remains doable. There are two such descendants and as we will see their definition relies on the same type of arguments. We will define here only a first descendant and relate it to the Ward identities; the other descendant will be defined in the next section.
	
	\subsubsection{A second order descendant}
	To start with we describe the descendant field at the order two $\L_{-2}$, whose definition is based on the stress-energy tensor $\SET$.
	For this purpose we take $\beta<Q$ as well as $t\in\R$ and introduce the functional 
	\begin{equation}
		\L_{-2}V_\beta(t)[\Phi]\coloneqq :\left((Q+\beta)\partial^2\Phi(t)-:\left(\partial\Phi(t)\right)^2:\right) V_\beta(t):.
	\end{equation}
	The associated (regularized) correlation function is then defined by 
	\begin{equation}
		\ps{\L_{-2}V_{\beta}(t)\prod_{k=1}^NV_{\alpha_k}(z_k)\prod_{l=1}^MV_{\beta_l}(s_l)}_{\delta,\eps,\rho},
	\end{equation}
	provided that the assumptions of Proposition~\ref{prop:analycity} are satisfied.
	Then the descendant is defined via a limiting procedure like before:
	\begin{lemma}\label{lemma:desc_two_L}
		For $t\in\R\setminus\{s_1,\cdots,s_M\}$, assume that  $(\bm\alpha,\beta)\in\mc A_{N,M+1}$.
		Then as $\rho,\eps,\delta\to0$, the following limit exists and is finite:
		\begin{equation}
			\begin{split}
				&\ps{\L_{-2}V_{\beta}(t)\prod_{k=1}^NV_{\alpha_k}(z_k)\prod_{l=1}^MV_{\beta_l}(s_l)}\coloneqq\\
				&\lim\limits_{\delta,\eps,\rho\to0}\ps{\L_{-2}V_{\beta}(t)\prod_{k=1}^NV_{\alpha_k}(z_k)\prod_{l=1}^MV_{\beta_l}(s_l)}_{\delta,\eps,\rho}-\mathfrak R_{-2}(\delta,\eps,\rho,\bm\alpha)
			\end{split}
		\end{equation}
		where the remainder term admits the following explicit expression:
		\begin{equation}\label{eq:rem2}
			\begin{split}
				\mathfrak R_{-2}&(\delta,\eps,\rho,\bm\alpha)\coloneqq-\mu\delta\int_{\R}\frac{\ps{V_{\gamma}(x+i\delta)V_\beta(t)\V}_{\delta,\eps,\rho}}{(t-x)^2+\delta^2}dx\\
				&+\frac1{\eps}\Big(\ps{V_{\gamma}(t-\eps)V_\beta(t)\V}_{\delta,\eps,\rho}\mu_{\partial}(t-\eps)+\ps{V_{\gamma}(t+\eps)V_\beta(t)\V}_{\delta,\eps,\rho}\mu_{\partial}(t+\eps)\Big).
			\end{split}
		\end{equation}
		Moreover the map $\bm\alpha\mapsto\ps{\L_{-2}V_{\beta}(t)\prod_{k=1}^NV_{\alpha_k}(z_k)\prod_{l=1}^MV_{\beta_l}(s_l)}$ is analytic in a complex neighborhood of $\mc A_{N,M+1}$. Finally for $\beta$ negative enough this remainder term vanishes in the $\rho,\eps,\delta\to0$ limit.
	\end{lemma}
	\begin{proof}
		The fact that the above limit exists follows from the proof of Theorem~\ref{thm:desc_two_set} below, while its analycity is a consequence of Theorem~\ref{thm:desc_two_set} since the descendants and correlation functions involved are already known to depend analytically in $\bm\alpha\in\mc A_{N,M}$.
		
		Finally we see that if we take $\beta<-\frac1\gamma$ then the second line in the expression of the remainder term vanishes in virtue of item $(2)$ of Lemma~\ref{lemma:fusion}. As for the first line it is seen to vanish for $\beta<-\frac1\gamma$ as well thanks to Lemma~\ref{lemma:remainderL2} below.
	\end{proof}
	Having properly introduced the $\L_{-2}V_\beta$ descendant we are now in position to derive the Ward identities for the boundary Liouville theory. To this end we introduce the shorthand
	\begin{align*}
		&\bm {\mathcal L_{-1}^{(k)}}\ps{V_\beta(t)\prod_{k=1}^NV_{\alpha_k}(z_k)\prod_{l=1}^MV_{\beta_l}(s_l)}\\
		&\coloneqq \lim\limits_{\delta,\eps,\rho\to0} \partial_{z_k}\ps{V_\beta(t)\prod_{k=1}^NV_{\alpha_k}(z_k)\prod_{l=1}^MV_{\beta_l}(s_l)}_{\delta,\eps,\rho}-\mathfrak{R}_{-1}^{(k)}(\delta,\eps,\rho,\bm\alpha)
	\end{align*}
	for $k\in\{1,2N+M\}$ and with $\mathfrak{R}_{-1}^{(k)}(\delta,\eps,\rho,\bm\alpha)\coloneqq 0$ for $1\leq k\leq 2N$ while \\
	$\mathfrak{R}_{-1}^{(2N+l)}(\delta,\eps,\rho,\bm\alpha)\coloneqq \Big(\ps{V_{\gamma}(s_l-\eps)\V}_{\delta,\eps,\rho}\mu_{l}-\ps{V_{\gamma}(s_l+\eps)\V}_{\delta,\eps,\rho}\mu_{l+1}\Big)$ for $2N+1\leq k\leq 2N+M$.
	\begin{theorem}\label{thm:desc_two_set}
		Assume that the assumptions of Proposition~\ref{prop:analycity} are satisfied. Then:
		\begin{equation}
			\begin{split}
				&\ps{\L_{-2}V_{\beta}(t)\prod_{k=1}^NV_{\alpha_k}(z_k)\prod_{l=1}^MV_{\beta_l}(s_l)}\\
				&=\sum_{k=1}^{2N+M}\left(\frac{\bm {\mathcal L_{-1}^{(k)}}}{t-z_k}
				+\frac{\Delta_{\alpha_k}}{(t-z_k)^2}\right)\ps{V_\beta(t)\prod_{k=1}^NV_{\alpha_k}(z_k)\prod_{l=1}^MV_{\beta_l}(s_l)}.
			\end{split}
		\end{equation}
		In particular we have the following equality in the weak sense:
		\begin{equation}
			\begin{split}
				&\ps{\L_{-2}V_{\beta}(t)\prod_{k=1}^NV_{\alpha_k}(z_k)\prod_{l=1}^MV_{\beta_l}(s_l)}\\
				&=\sum_{k=1}^{2N+M}\left(\frac{\partial_{z_k}}{t-z_k}
				+\frac{\Delta_{\alpha_k}}{(t-z_k)^2}\right)\ps{V_\beta(t)\prod_{k=1}^NV_{\alpha_k}(z_k)\prod_{l=1}^MV_{\beta_l}(s_l)}.
			\end{split}
		\end{equation}
	\end{theorem}
	\begin{proof}
		Using Gaussian integration by parts in the form of Lemma~\ref{lemma:GaussianIPP}, on the one hand
		\begin{align*}
			&\lim\limits_{\rho\to0}\ps{:(Q+\beta)\partial^2_t\Phi(t)V_{\beta}(t):\V}_{\delta,\eps,\rho}=\sum_{k=1}^{2N+M}\frac{(Q+\beta)\alpha_k}{2(t-z_k)^2}\\
			&-\mu\int_{ \H_{\delta,\eps}}\left(\frac{\gamma(Q+\beta)}{2(t-x)^2}+\frac{\gamma(Q+\beta)}{2(t-\bar x)^2}\right)\ps{V_{\gamma}(x)V_{\beta}(t)\V}_{\delta,\eps}\norm{dx}^2\\
			&-\int_{ \Reps}\frac{(Q+\beta)\gamma}{2(t-x)^2}\ps{V_{\gamma}(x)V_{\beta}(t)\V}_{\delta,\eps}\mu_{\partial}(dx).
		\end{align*}
		On the other hand by recursive application of Equation~\eqref{eq:IPP_product} we can write
		\begin{align*}
			&\lim\limits_{\rho\to0}\ps{:\left(:\partial_t\Phi(t):^2V_{\beta}(t)\right):\V}_{\delta,\eps,\rho}=\sum_{k,l}\frac{\alpha_k\alpha_l}{4(t-z_k)(t-z_l)}\\
			&-\mu\int_{ \H_{\delta,\eps}}\left(\frac{\gamma^2}{4(t-x)^2}+\sum_{k=1}^{2N+M}\frac{\gamma\alpha_k}{4(t-x)(t-z_k)}+\frac{\gamma^2}{4(t-x)(t-\bar x)}\right)\ps{V_{\gamma}(x)V_{\beta}(t)\V}_{\delta,\eps}\norm{dx}^2\\
			&-\mu\int_{ \H_{\delta,\eps}}\left(\frac{\gamma^2}{4(t-\bar x)^2}+\sum_{k=1}^{2N+M}\frac{\gamma\alpha_k}{4(t-\bar x)(t-z_k)}+\frac{\gamma^2}{4(t-x)(t-\bar x)}\right)\ps{V_{\gamma}(x)V_{\beta}(t)\V}_{\delta,\eps}\norm{dx}^2\\
			&-\int_{ \Reps}\left(\frac{\gamma^2}{4(t-x)^2}+\sum_{k=1}^{2N+M}\frac{\gamma\alpha_k}{4(t-x)(t-z_k)}\right)\ps{V_{\gamma}(x)V_{\beta}(t)\V}_{\delta,\eps}\mu_{\partial}(dx)\\
			&+\int_{\Heps^2}\frac{\gamma^2}{4}\left(\frac{1}{t-x}+\frac{1}{t-\bar x}\right)\left(\frac{1}{t-y}+\frac{1}{t-\bar y}\right)\ps{V_{\gamma}(x)V_\gamma(y)V_{\beta}(t)\V}_{\delta,\eps}\norm{dx}^2\norm{dy}^2\\
			&+2\int_{\Heps\times\Reps}\frac{\gamma^2}{4}\left(\frac{1}{t-x}+\frac{1}{t-\bar x}\right)\frac{1}{t-y}\ps{V_{\gamma}(x)V_{\gamma}(y)V_{\beta}(t)\V}_{\delta,\eps}\norm{dx}^2\mu_{\partial}(dy)\\
			&+\int_{\Reps^2}\frac{\gamma^2}{4(t-x)(t-y)}\ps{V_{\gamma}(x)V_{\gamma}(y)V_{\beta}(t)\V}_{\delta,\eps}\mu_{\partial}(dx)\mu_{\partial}(dy).
		\end{align*}
		Recollecting terms and using the explicit expression of the conformal weights $\Delta_\alpha=\frac\alpha2(Q-\frac\alpha2)$ as well as $\Delta_\gamma=1$ we arrive at the following expression for the regularized $\L_{-2}$ descendant:
		\begin{align*}
			&\lim\limits_{\rho\to0}\ps{\L_{-2}V_{\beta}(t)\V}_{\delta,\eps,\rho}=\left(\sum_{k=1}^{2N+M}			\frac{\Delta_{\alpha_k}}{(t-z_k)^2}-\sum_{k\neq l}\frac{\alpha_k\alpha_l}{4(t-z_k)(t-z_l)}\right)\ps{V_\beta(t)\V}_{\delta,\eps,\rho}\\
			&=-\mu\int_{ \H_{\delta,\eps}}\frac{1+\frac{\gamma\beta}2}{(t-x)^2}+\sum_{k=1}^{2N+M}\frac{\gamma\alpha_k}{2(t-x)(z_k-t)}+\frac{\gamma^2}{4(t-x)(\bar x-t)}\ps{V_{\gamma}(x)V_{\beta}(t)\V}_{\delta,\eps}\norm{dx}^2\\
			&-\mu\int_{ \H_{\delta,\eps}}\frac{1+\frac{\gamma\beta}2}{(t-\bar x)^2}+\sum_{k=1}^{2N+M}\frac{\gamma\alpha_k}{2(t-\bar x)(z_k-t)}+\frac{\gamma^2}{4(t-\bar x)(x-t)}\ps{V_{\gamma}(x)V_{\beta}(t)\V}_{\delta,\eps}\norm{dx}^2\\
			&-\int_{ \Reps}\frac{1+\frac{\gamma\beta}2}{(t-x)^2}+\sum_{k=1}^{2N+M}\frac{\gamma\alpha_k}{2(t-x)(z_k-t)}\ps{V_{\gamma}(x)V_{\beta}(t)\V}_{\delta,\eps}\mu_{\partial}(dx)\\
			&+\int_{\Heps^2}\frac{\gamma^2}{2(t-x)}\left(\frac{1}{y-x}+\frac{1}{\bar y-x}\right)+\frac{\gamma^2}{2(t-\bar x)}\left(\frac{1}{y-\bar x}+\frac{1}{\bar y-\bar x}\right)\ps{V_{\gamma}(x)V_\gamma(y)V_{\beta}(t)\V}_{\delta,\eps}\norm{dx}^2\norm{dy}^2\\
			&+\int_{\Heps\times\Reps}\frac{\gamma^2}{2(t-x)(y-x)}+\frac{1}{t-y}\left(\frac{1}{x-y}+\frac{1}{\bar x-y}\right)\ps{V_{\gamma}(x)V_{\gamma}(y)V_{\beta}(t)\V}_{\delta,\eps}\norm{dx}^2\mu_{\partial}(dy)\\
			&+\int_{\Reps^2}\frac{\gamma^2}{2(t-x)(y-x)}\ps{V_{\gamma}(x)V_{\gamma}(y)V_{\beta}(t)\V}_{\delta,\eps}\mu_{\partial}(dx)\mu_{\partial}(dy).
		\end{align*}
		Now recalling the computations conducted in the proofs of Lemmas~\ref{lemma:desc1_bulk} and~\ref{lemma:desc_L} and using the \lq\lq symmetrization identity" $\frac{1}{(x-y)(x-w)}=\frac{1}{y-w}\left(\frac{1}{x-y}-\frac{1}{x-w}\right)$ we get
		\begin{align*}
			&\lim\limits_{\rho\to0}\ps{\L_{-2}V_{\beta}(t)\V}_{\delta,\eps,\rho}-\sum_{k=1}^{2N+M}\left(\frac{\partial_{z_k}}{t-z_k}
			+\frac{\Delta_{\alpha_k}}{(t-z_k)^2}\right)\ps{V_\beta(t)\V}_{\delta,\eps,\rho}\\
			&=-\mu\int_{ \H_{\delta,\eps}}\frac{1+\frac{\gamma\beta}2}{(t-x)^2}+\sum_{k=1}^{2N+M}\frac{\gamma\alpha_k}{2(t-x)(z_k-x)}+\frac{\gamma^2}{2(t-x)(\bar x-x)}\ps{V_{\gamma}(x)V_{\beta}(t)\V}_{\delta,\eps}\norm{dx}^2\\
			&-\mu\int_{ \H_{\delta,\eps}}\frac{1+\frac{\gamma\beta}2}{(t-\bar x)^2}+\sum_{k=1}^{2N+M}\frac{\gamma\alpha_k}{2(t-\bar x)(z_k-\bar x)}+\frac{\gamma^2}{2(t-\bar x)(x-\bar x)}\ps{V_{\gamma}(x)V_{\beta}(t)\V}_{\delta,\eps}\norm{dx}^2\\
			&-\int_{ \Reps}\frac{1+\frac{\gamma\beta}2}{(t-x)^2}+\sum_{k=1}^{2N+M}\frac{\gamma\alpha_k}{2(t-x)(z_k-x)}\ps{V_{\gamma}(x)V_{\beta}(t)\V}_{\delta,\eps}\mu_{\partial}(dx)\\
			&+\int_{\Heps^2}\frac{\gamma^2}{2(t-x)}\left(\frac{1}{y-x}+\frac{1}{\bar y-x}\right)+\frac{\gamma^2}{2(t-\bar x)}\left(\frac{1}{y-\bar x}+\frac{1}{\bar y-\bar x}\right)\ps{V_{\gamma}(x)V_\gamma(y)V_{\beta}(t)\V}_{\delta,\eps}\norm{dx}^2\norm{dy}^2\\
			&+\int_{\Heps\times\Reps}\frac{\gamma^2}{2(t-x)(y-x)}+\frac{1}{t-y}\left(\frac{1}{x-y}+\frac{1}{\bar x-y}\right)\ps{V_{\gamma}(x)V_{\gamma}(y)V_{\beta}(t)\V}_{\delta,\eps}\norm{dx}^2\mu_{\partial}(dy)\\
			&+\int_{\Reps^2}\frac{\gamma^2}{2(t-x)(y-x)}\ps{V_{\gamma}(x)V_{\gamma}(y)V_{\beta}(t)\V}_{\delta,\eps}\mu_{\partial}(dx)\mu_{\partial}(dy).
		\end{align*}
		We recognize derivatives of the correlation functions in the above; to be more specific
		\begin{align*}
			&\lim\limits_{\rho\to0}\ps{\L_{-2}V_{\beta}(t)\V}_{\delta,\eps,\rho}-\sum_{k=1}^{2N+M}\left(\frac{\partial_{z_k}}{t-z_k}
			+\frac{\Delta_{\alpha_k}}{(t-z_k)^2}\right)\ps{V_\beta(t)\V}_{\delta,\eps,\rho}\\
			&=-\mu\int_{ \H_{\delta,\eps}}\left(\partial_xJ(x)+\partial_{\bar x}J(\bar x)\right)\norm{dx}^2-\int_{ \Reps}\partial_x J(x)\mu_{\partial}(dx)
		\end{align*}
		where we have set $J(x)\coloneqq \frac{1}{(t-x)}\ps{V_{\gamma}(x)V_{\beta}(t)\V}_{\delta,\eps}$ for $x\in\overline\H$. Hence using Stokes' formula in the same fashion as in the proofs of Lemmas~\ref{lemma:desc1_bulk} and~\ref{lemma:desc_L} yields
		\begin{align*}
			&\ps{\L_{-2}V_{\beta}(t)\V}_{\delta,\eps,\rho}-\sum_{k=1}^{2N+M}\left(\frac{\partial_{z_k}}{t-z_k}
			+\frac{\Delta_{\alpha_k}}{(t-z_k)^2}\right)\ps{\V}_{\delta,\eps,\rho}=\mu\frac i2\int_{ \R}\left(J(x+i\delta)-J(x-i\delta)\right)dx\\
			&-\sum_{l=1}^M\left(J(s_l-\eps)\mu_{l}-J(s_l+\eps)\mu_{l+1}\right)-\left(J(t-\eps)\mu_{\partial}(t-\eps)-J(t+\eps)\mu_{\partial}(t+\eps)\right)+o(1)
		\end{align*} 
		where the $o(1)$ vanishes in the $\rho,\delta,\eps\to0$ limit.
		This equality can be rewritten under the form
		\begin{align*}
			&\ps{\L_{-2}V_{\beta}(t)\V}_{\delta,\eps,\rho}-\mu\frac i2\int_{ \R}\left(J(x+i\delta)-J(x-i\delta)\right)dx+\Big(J(t-\eps)\mu_{\partial}(t-\eps)-J(t+\eps)\mu_{\partial}(t+\eps)\Big)\\
			&=\sum_{k=1}^{2N+M}\left(\frac{\partial_{z_k}}{t-z_k}
			+\frac{\Delta_{\alpha_k}}{(t-z_k)^2}\right)\ps{\V}_{\delta,\eps,\rho}-\sum_{l=1}^M\left(J(s_l-\eps)\mu_{l}-J(s_l+\eps)\mu_{l+1}\right)+o(1).
		\end{align*} 
		The expression that appears on the first line corresponds to the defining expression for the descendant $\L_{-2}$ (the remainder terms that appear are those from the definition of $\L_{-2}$). Likewise the terms that appear in the second line can be rewritten as 
		\begin{align*}
			&\sum_{k=1}^{2N+M}\left(\frac{\Lc_{-1}^{(k)}}{t-z_k}
			+\frac{\Delta_{\alpha_k}}{(t-z_k)^2}\right)\ps{\V}_{\delta,\eps,\rho}\\
			&+\sum_{l=1}^M\left(\frac{\eps}{(t-s_l)(t-s_l+\eps)}\ps{V_{\gamma}(s_l-\eps)V_{\beta}(t)\V}_{\delta,\eps}\mu_l+\frac{\eps}{(t-s_l)(t-s_l-\eps)}\ps{V_{\gamma}(s_l+\eps)V_{\beta}(t)\V}_{\delta,\eps}\mu_{l+1}\right)
		\end{align*} 
		where the last expression converges to $0$ as $\eps\to0$ based on the fusion asymptotics  from Lemma~\ref{lemma:fusion_integrability}. Therefore the second line converges towards $\sum_{k=1}^{2N+M}\left(\frac{\Lc_{-1}^{(k)}}{t-z_k}
		+\frac{\Delta_{\alpha_k}}{(t-z_k)^2}\right)\ps{\V}$. In the end we see that, as desired,
		\begin{align*}
			&\ps{\L_{-2}V_{\beta}(t)\V}=\sum_{k=1}^{2N+M}\left(\frac{\Lc_{-1}^{(k)}}{t-z_k}
			+\frac{\Delta_{\alpha_k}}{(t-z_k)^2}\right)\ps{\V}.
		\end{align*}
		
		As for the second part of our claim, that is expressing the latter in terms of weak derivatives of the correlation functions, it is a consequence of the expression of the $\L_{-1}$ descendants in terms of derivatives via Corollary~\ref{cor:der}.
	\end{proof}

	\subsection{Ward identities for the stress-energy tensor}
	We have described above a second order descendant whose expression closely resembles that of the stress-energy tensor, that is 
	\begin{equation}
		\SET(z)[\Phi]\coloneqq Q\partial^2\Phi(z)-:\left(\partial\Phi(z)\right)^2:
	\end{equation}	
	for $z\in\H$. We will explain here to what extent these two expressions are actually related by explaining that one can derive Ward identities based on the very same reasoning.
	\subsubsection{Local Ward identities}
	The Ward identity when the stress-energy tensor $\SET$ is evaluated on the boundary corresponds to Theorem~\ref{thm:desc_two_set} in the special case where we take $\beta=0$. 
	The same reasoning applies when we take this functional to be evaluated in the bulk:
	\begin{theorem}\label{thm:ward_local_set}
		Assume that $\bm\alpha\in\mc A_{N,M}$ and take $z\in\H$. Then:
		\begin{equation}
			\begin{split}
				&\lim\limits_{\delta,\eps,\rho\to0}\ps{\SET(z)\prod_{k=1}^NV_{\alpha_k}(z_k)\prod_{l=1}^MV_{\beta_l}(s_l)}_{\delta,\eps,\rho}=\sum_{k=1}^{2N+M}\left(\frac{\bm {\mathcal L_{-1}^{(k)}}}{z-z_k}
				+\frac{\Delta_{\alpha_k}}{(z-z_k)^2}\right)\ps{\prod_{k=1}^NV_{\alpha_k}(z_k)\prod_{l=1}^MV_{\beta_l}(s_l)}.
			\end{split}
		\end{equation}
	\end{theorem}
	\begin{proof}
		The exact same computations as in the proof of Theorem~\ref{thm:desc_two_set} yield
		\begin{align*}
			&\lim\limits_{\rho\to0}\ps{\SET(z)\V}_{\delta,\eps,\rho}-\sum_{k=1}^{2N+M}\left(\frac{\partial_{z_k}}{z-z_k}
			+\frac{\Delta_{\alpha_k}}{(z-z_k)^2}\right)\ps{\V}_{\delta,\eps,\rho}\\
			&=-\mu\int_{ \H_{\delta,\eps}}\left(\partial_xJ(\bar x)+\partial_{\bar x}J(\bar x)\right)\norm{dx}^2+\int_{ \Reps}\partial_x J(x)\mu_\partial(dx)
		\end{align*}
		where this time $J(x)\coloneqq \frac{1}{(z-x)}\ps{V_{\gamma}(x)\V}_{\delta,\eps}$. Hence like before we can write that
		\begin{align*}
			&\ps{\SET(z)\V}_{\eps}-\sum_{k=1}^{2N+M}\left(\frac{\partial_{z_k}}{t-z_k}
			+\frac{\Delta_{\alpha_k}}{(t-z_k)^2}\right)\ps{\V}_{\delta,\eps,\rho}\\
			&=-\mu\int_{ \R}\left(J(x+i\delta)-J(x-i\delta)\right)dx+\sum_{l=1}^MJ(s_l-\eps)\mu_{l}-J(s_l+\eps)\mu_{l+1}+o(1)
		\end{align*}
		where the last sum corresponds to the remainder term that appears in the definition of the descendant for boundary Vertex Operators. As for the first integral we have 
		\begin{align*}
			\int_{\R}\left(J(x+i\delta)-J(x-i\delta)\right)dx=\int_\R \frac{2i\delta}{(z-x)^2+\delta^2}\ps{V_{\gamma}(x+i\delta)\V}_\delta dx\to0
		\end{align*}
		since $z\in\H$ and therefore the whole integral $\int_\R \frac{1}{(z-x)^2}\ps{V_{\gamma}(x+i\delta)\V}_\delta dx$ remains uniformly bounded in $\delta$ via Lemma~\ref{lemma:inf_integrability}.
	\end{proof}
	
	\subsubsection{Global Ward identities}
	The global Ward identities are consequences of the local Ward identities and that can be thought of as concrete constraints posed by the global conformal invariance of the correlation functions. Put differently they correspond to the infinitesimal version of the covariance under M\"obius transforms of the plane of the correlation functions in the sense that for any $F$ for which it makes sense,
	\begin{equation}
		\begin{split}
			&\ps{F[\Phi\circ\psi+Q\ln\norm{\psi'}]\prod_{k=1}^NV_{\alpha_k}(\psi(z_k))\prod_{l=1}^MV_{\beta_l}(\psi(s_l))}=\\
			&\prod_{k=1}^N\norm{\psi'(z_k)}^{-2\Delta_{\alpha_k}}\prod_{l=1}^M\norm{\psi'(z_k)}^{-\Delta_{\beta_l}}\ps{F[\Phi]\prod_{k=1}^NV_{\alpha_k}(z_k)\prod_{l=1}^MV_{\beta_l}(s_l)}
		\end{split}
	\end{equation}
	where $\psi$ is a M\"obius transform of the half-plane (see~\cite[Theorem 3.5]{HRV16}).
	\begin{theorem}\label{thm:ward_global_set}
		Assume that the Seiberg bounds hold true. Then for $n=0,1,2$:
		\begin{equation}
			\begin{split}
				&\sum_{k=1}^{2N+M}\left(z_k^n\bm {\mathcal L_{-1}^{(k)}}
				+nz_k^{n-1}\Delta_{\alpha_k}\right)\ps{\prod_{k=1}^NV_{\alpha_k}(z_k)\prod_{l=1}^MV_{\beta_l}(s_l)}=0.
			\end{split}
		\end{equation}
	\end{theorem}
	\begin{proof}
		A direct computation shows that if $\psi$ is a M\"obius transform of the half-plane then for any map $\Phi$ we have $\SET(t)[\Phi\circ\psi+Q\ln\norm{\psi'}]=\psi'(t)^2\SET(\psi(t))[\Phi]$. This implies with $\psi(t)\coloneqq\frac{-1}t$ together with conformal covariance of the correlation functions that $\ps{\SET(t)\V}$ scales like $\frac{1}{t^4}$ as $t\to\infty$. Now note that if $\norm{t}>\norm{z_k}$ for $1\leq k\leq 2N+M$, then the local Ward identities from Theorem~\ref{thm:ward_local_set} can be written as
		\[
		\ps{\SET(t)\V}=\sum_{n\geq0}\frac1{t^{n+1}}\left(\sum_{k=1}^{2N+M}\left(z_k^n\bm {\mathcal L_{-1}^{(k)}}+
		nz_k^{n-1}\Delta_{\alpha_k}\right)\ps{\prod_{k=1}^NV_{\alpha_k}(z_k)\prod_{l=1}^MV_{\beta_l}(s_l)}\right).
		\]
		Therefore for $0\leq n\leq 2$ the coefficient in front of $\frac{1}{t^{n+1}}$ must vanish to be consistent with the fact that $\ps{\SET(t)\V}\simeq\frac c{t^4}$.
	\end{proof}
	
	\section{Implications on the correlation functions}\label{sec:bpz}

	\subsection{Second order derivatives}\label{subsec:der2}
	We first focus on derivatives of the correlation functions before describing another descendant field derived from the stress-energy tensor.
	To start with like before if we consider a bulk Vertex Operator then the limit does make sense: 
	\begin{lemma}
		A $\rho,\eps,\delta\to0$  the following limit exists and is finite:
		\begin{equation}
			\ps{\L_{-(1,1)}V_{\alpha_1}(z_1)\prod_{k=2}^NV_{\alpha_k}(z_k)\prod_{l=1}^MV_{\beta_l}(s_l)}\coloneqq\lim\limits_{\delta\to0}\lim\limits_{\eps\to0}\lim\limits_{\rho\to0} \partial_{z_1}^2\ps{\prod_{k=1}^NV_{\alpha_k}(z_k)\prod_{l=1}^MV_{\beta_l}(s_l)}_{\delta,\eps,\rho}.
		\end{equation}
	\end{lemma}
	\begin{proof}
		The method is exactly the same as in the proof of Lemma~\ref{lemma:desc1_bulk} and is less involved than the one defining its boundary counterpart, so that we refer to the proof of Lemma~\ref{lemma:desc2_L} below.
	\end{proof}
	
	Like in the case of the descendant at the first order we will need to be more careful to treat the case of a boundary Vertex Operator:
	\begin{lemma}\label{lemma:desc2_L}
		Assume that $\bm\alpha\in\mc A_{N,M}$. Then one can define a quantity $\mathfrak{R}_{-(1,1)}(\delta,\eps,\rho,\bm\alpha)$ satisfying the following properties:
		\begin{enumerate}
			\item as $\rho$, $\eps$ and then $\delta$ go to $0$, the limit exists and is finite:
			\begin{equation}
				\begin{split}
					&\ps{\L_{-(1,1)}V_{\beta_1}(s_1)\prod_{k=1}^NV_{\alpha_k}(z_k)\prod_{l=2}^MV_{\beta_l}(s_l)}\\
					&\coloneqq\lim\limits_{\delta,\eps,\rho\to0}\partial_{s_1}^2\ps{\prod_{k=1}^NV_{\alpha_k}(z_k)\prod_{l=1}^MV_{\beta_l}(s_l)}_{\delta,\eps,\rho}-\mathfrak{R}_{-(1,1)}(\delta,\eps,\bm\alpha);
				\end{split}
			\end{equation}
			\item for $\beta_1$ negative enough,
			\[
			\lim\limits_{\delta,\eps,\rho\to0}\mathfrak{R}_{-(1,1)}(\delta,\eps,\bm\alpha)=0;
			\]
			\item the map $\bm\alpha\mapsto\ps{\L_{-(1,1)}V_{\beta_1}(s_1)\prod_{k=1}^NV_{\alpha_k}(z_k)\prod_{l=2}^MV_{\beta_l}(s_l)}$ is analytic in a complex neighborhood of $\mc A_{N,M}$.
		\end{enumerate}
	\end{lemma}
	We provide an explicit expression for this remainder term in Equation~\eqref{eq:rem11} below.
	\begin{remark}
		There is \lq\lq uniqueness" of this remainder term in the sense that if $\tilde{\mathfrak{R}}_{-(1,1)}(\delta,\eps,\bm\alpha)$ is such that the same set of assumptions (that is items $(1)$, $(2)$ and $(3)$) is satisfied then 
		\[
		\lim\limits_{\delta,\eps,\rho\to0}\mathfrak{R}_{-(1,1)}(\delta,\eps,\bm\alpha)-\tilde{\mathfrak{R}}_{-(1,1)}(\delta,\eps,\bm\alpha)=0.
		\]
		This is a direct consequence of the uniqueness of the analytic continuation. More will be said about this remainder term in Subsection~\ref{subsec:remainder} below.
	\end{remark}
	As a corollary by combining items $(1)$, $(2)$ and $(3)$ we have the following equality in the weak sense:
	\begin{equation}
		\ps{\L_{-(1,1)}V_{\beta_1}(s_1)\prod_{k=1}^NV_{\alpha_k}(z_k)\prod_{l=2}^MV_{\beta_l}(s_l)}=\partial_{s_1}^2\ps{\prod_{k=1}^NV_{\alpha_k}(z_k)\prod_{l=1}^MV_{\beta_l}(s_l)}.
	\end{equation}
	\begin{proof}
		In this proof we will say that a quantity $\mc I_{\delta,\eps,\rho}$ is \lq\lq almost regular" if it can be written as a sum $I_{\delta,\eps,\rho}+\mathfrak R_{\delta,\eps,\rho}$ where $\lim\limits_{\delta,\eps,\rho\to0}I_{\delta,\eps,\rho}$ exists and is analytic in $\bm\alpha\in\mc A_{N,M}$ while $\mathfrak R_{\delta,\eps,\rho}$ is a remainder term that vanishes in the limit as soon as $\beta<0$. As can be seen from the proof of Lemma~\ref{lemma:desc_L} the following quantities are almost regular:
		\begin{align*}
			\int_{ \Heps}\left(\frac{\gamma \beta_1}{2(x-s_1)}+\frac{\gamma\beta_1}{2(\bar x-s_1)}\right)\ps{V_{\gamma}(x)\V}_{\delta,\eps}\norm{dx}^2\quad\text{and}\quad\int_{ \Reps}\frac{\gamma \beta_1}{2(x-s_1)}\ps{V_{\gamma}(x)\V}_{\delta,\eps}\mu_\partial(dx)
		\end{align*}
		since the corresponding remainder terms are linear combinations of correlation functions of the form $\ps{V_\gamma(s_1\pm\eps)\V}_{\delta,\eps,\rho}$.
		Our goal is to prove that $\partial_{s_1}^2\ps{\prod_{k=1}^NV_{\alpha_k}(z_k)\prod_{l=1}^MV_{\beta_l}(s_l)}_{\delta,\eps,\rho}$ can be written as a sum of almost regular terms plus another remainder term for which we will provide an explicit expression. 
		
		The procedure is rather heavy so we will start by detailing the method in the case where $\mu=0$. For the sake of simplicity we note $\R_u\coloneqq\Reps\setminus(s_1-u,s_1+u)$ and $\R_c\coloneqq\Reps\cap(s_1-u,s_1+u)$.
		Along the same lines as before we have 
		\begin{align*}
			&\lim\limits_{\rho\to0}\partial_{s_1}^2\ps{\prod_{k=1}^NV_{\alpha_k}(z_k)\prod_{l=1}^MV_{\beta_l}(s_l)}_{\delta,\eps,\rho}=\left(\sum_{k\neq 2N+1}\frac{\alpha_k\beta_1}{2(z_k-s_1)^2}+\sum_{k,l\neq 2N+1}\frac{\alpha_k\beta_1\alpha_l\beta_1}{4(z_k-s_1)(z_l-s_1)}\right)\ps{\V}_{\delta,\eps}\\
			&-\int_{\R_u}\left(\frac{\gamma \beta_1\left(1+\frac{\gamma \beta_1}2\right)}{2(x-s_1)^2}+\sum_{k\neq 2N+1}\frac{\alpha_k\beta_1\gamma\beta_1}{2(z_k-s_1) (x-s_1)}\right)\ps{V_{\gamma}(x)\V}_{\delta,\eps}\norm{dx}^2\\
			&-\int_{\R_c}\partial_x\left(\frac{\gamma\beta_1}{2(s_1-x)}\ps{V_{\gamma}(x)\V}_{\delta,\eps}\right)+\sum_{k\neq 2N+1}\left(\frac{\beta_1\alpha_k}{s_1-z_k}+\frac{\gamma\alpha_k}{2(x-z_k)}\right)\partial_x\ps{V_{\gamma}(x)\V}_{\delta,\eps}\mu_{\partial}(dx)\\
			&+\int_{\Reps^2}\frac{(\gamma\beta_1)^2}{4(s_1-x)(s_1-y)}\ps{V_{\gamma}(x)V_{\gamma}(y)\V}_{\delta,\eps}\mu_{\partial}(dx)\mu_{\partial}(dy)\\
			&+\int_{\R_c\times\Reps}\frac{\gamma^3\beta_1}{4(s_1-x)(x-y)}\ps{V_{\gamma}(x)V_{\gamma}(y)\V}_{\delta,\eps}\mu_{\partial}(dx)\mu_{\partial}(dy)\\
			&+\int_{\R_c\times\Reps}\sum_{k\neq 2N+1}\left(\frac{\beta_1\alpha_k}{s_1-z_k}+\frac{\gamma\alpha_k}{2(s_1-u-z_k)}\right)\frac{\gamma^2}{2(x-y)}\ps{V_{\gamma}(x)V_{\gamma}(y)\V}_{\delta,\eps}\mu_{\partial}(dx)\mu_{\partial}(dy).
		\end{align*}
		As $\eps,\delta\to0$, the first integral over $\R_u$ converges and is analytic in $\bm\alpha\in\mc A_{N,M}$. By integration by parts the second integral over $\R_c$ is given by a constant order term that is analytic in $\bm\alpha$ plus the quantity $ F(u)-F(\eps)$ where 
		\begin{align*}
			&F(u)\coloneqq\frac{\gamma\beta_1}{2u}\left(\ps{V_{\gamma}(s_1-u)\V}_{\delta,\eps}\mu_{1}+\ps{V_{\gamma}(s_1+u)\V}_{\delta,\eps}\mu_{2}\right)\\
			&+\sum_{k\neq 2N+1}\left(\frac{\beta_1\alpha_k}{s_1-z_k}+\frac{\gamma\alpha_k}{2(s_1-u-z_k)}\right)\ps{V_{\gamma}(s_1-u)\V}_{\delta,\eps}\mu_{1}\\
			&-\sum_{k\neq 2N+1}\left(\frac{\beta_1\alpha_k}{s_1-z_k}+\frac{\gamma\alpha_k}{2(s_1+u-z_k)}\right)\ps{V_{\gamma}(s_1+u)\V}_{\delta,\eps}\mu_{2}.
		\end{align*}
		The term $F(\eps)$ will be part of the remainder $\mathfrak{R}_{-(1,1)}(\delta,\eps,\bm\alpha)$ while the quantity $F(u)$ has a well-defined limit which is analytic in $\bm\alpha\in\mc A_{N,M}$.
		
		It now remains to treat the two-fold integrals and for this we split $\Reps$ between $\R_u$ and $\R_c$ like before. By doing so:
		\begin{align*}
			&\lim\limits_{\eps,\delta\to0}\int_{\R_u^2}\frac{(\gamma\beta_1)^2}{4(s_1-x)(s_1-y)}\ps{V_{\gamma}(x)V_{\gamma}(y)\V}_{\delta,\eps}\mu_{\partial}(dx)\mu_{\partial}(dy)
		\end{align*}
		exists and is analytic in $\bm\alpha\in\mc A_{N,M}$, while for the integral over $\R_c\times \R_u$ we write that $\R_c=\R_c'\sqcup\R_\eps\cap(s_1-u',s_1+u')$ with $u'<u$ to separate the singularities at $x=y=s_1\pm u$ and at $x=s$. We then treat the corresponding terms using integration by parts like before:
		\begin{align*}
			&\int_{\R_c'\times\R_u}\frac{\gamma^3\beta_1}{4(s_1-x)(x-y)}\ps{V_{\gamma}(x)V_{\gamma}(y)\V}_{\delta,\eps}\mu_{\partial}(dx)\mu_{\partial}(dy)\\
			&=\int_{\R_c'\times\R_u}\frac{\gamma\beta_1}{2(s_1-x)}\left(\partial_y-\frac{\gamma^2}{2(x-y)}-\sum_{k\neq2N+1}\frac{\gamma\alpha_k}{2(z_k-y)}\right)\ps{V_{\gamma}(x)V_{\gamma}(y)\V}_{\delta,\eps}\mu_{\partial}(dx)\mu_{\partial}(dy)\\
			&+\int_{\R_c'\times\R_u\times\R_c}\frac{\gamma^3\beta_1}{4(s_1-x)(w-y)}\ps{V_{\gamma}(x)V_{\gamma}(y)V_{\gamma}(w)\V}_{\delta,\eps}\mu_{\partial}(dx)\mu_{\partial}(dy)\mu_{\partial}(dw).
		\end{align*}
		To remove the singularity as $x\to s$ we proceed in the exact same fashion as in the proof of Lemma~\ref{lemma:desc_L}: in the end we see that we obtain a sum of integrals that converge as $\eps,\delta\to0$ and that are analytic in the weights, plus remainder terms coming from integration by parts and containing correlation functions of the form $\ps{V_{\gamma}(s_1\pm\eps)\V}_{\delta,\eps}$: these are almost regular quantities. 
		
		The most relevant information comes from the two-fold integral $\R_c\times\R_c$: indeed using symmetry in the $x,y$ variables we can write it as
		\begin{align*}
			&\int_{\R_c^2}\frac{\gamma^2\beta\left(\beta+\frac\gamma2\right)}{4(s_1-x)(s_1-y)}\ps{V_{\gamma}(x)V_{\gamma}(y)\V}_{\delta,\eps}\mu_{\partial}(dx)\mu_{\partial}(dy)\\
			&=\frac{\gamma\beta}2\int_{\R_c^2}\left(\frac1{s_1-x}\partial_y-\sum_{k}\frac{\gamma\alpha_k}{2(s_1-x)(z_k-y)}\right)\ps{V_{\gamma}(x)V_{\gamma}(y)\V}_{\delta,\eps}\mu_{\partial}(dx)\mu_{\partial}(dy)\\
			&+\frac{\gamma^2}2\int_{\R_c^2\times\R_u}\frac{\gamma\beta}{2(s_1-x)(w-y)}\ps{V_{\gamma}(x)V_{\gamma}(y)V_{\gamma}(w)\V}_{\delta,\eps}\mu_{\partial}(dx)\mu_{\partial}(dy)\mu_{\partial}(dw).
		\end{align*}
		We can treat this last term along the same lines as in the proof of Lemma~\ref{lemma:desc_L}: in order to remove the singularity at $x=s_1$ we use integration by parts to get that this term is the sum of quantities of the form correlation functions of the form 
		\begin{align*}
			&\int_{\R_c\times\R_u}\frac{1}{(w-y)}\ps{V_{\gamma}(s_1\pm\eps)V_{\gamma}(y)V_{\gamma}(w)\V}_{\delta,\eps}\mu_{\partial}(dy)\mu_{\partial}(dw),\\
			&\int_{\R_c^2\times\R_u}\frac{1}{(z_k-x)(w-y)}\ps{V_{\gamma}(x)V_{\gamma}(y)V_{\gamma}(w)\V}_{\delta,\eps}\mu_{\partial}(dx)\mu_{\partial}(dy)\mu_{\partial}(dw)\quad\text{and}\\
			&\int_{\R_c^2\times\R_u^2}\frac{1}{(z-x)(w-y)}\ps{V_{\gamma}(x)V_{\gamma}(y)V_{\gamma}(w)V_\gamma(z)\V}_{\delta,\eps}\mu_{\partial}(dx)\mu_{\partial}(dy)\mu_{\partial}(dw)\mu_{\partial}(dz).
		\end{align*}
		As such and thanks to Lemma~\ref{lemma:fusion_integrability} these terms are almost regular.
		Let us therefore concentrate on the derivative term, which by integration by parts is equal to
		\begin{align*}
			&\frac{\gamma\beta}2\int_{\R_c}\frac1{s-x}\left(\mu_{1}\ps{V_{\gamma}(x)V_{\gamma}(s_1-\eps)\V}_{\delta,\eps}-\mu_{2}\ps{V_{\gamma}(x)V_{\gamma}(s_1+\eps)\V}_{\delta,\eps}\right)\mu_{\partial}(dx)\\
			-&\frac{\gamma\beta}2\int_{\R_c}\frac1{s_1-x}\left(\mu_{1}\ps{V_{\gamma}(x)V_{\gamma}(s_1-u)\V}_{\delta,\eps}\mu_{\partial}(dx)-\mu_{2}\ps{V_{\gamma}(x)V_{\gamma}(s_1+u)\V}_{\delta,\eps}\right)\mu_{\partial}(dx).
		\end{align*}
		As is now standard integrating by parts the last line will yield an almost regular quantity. As for the first term it is equal to
		\begin{align*}
			&\frac{\gamma\beta}2\int_{\eps}^u\frac1{x}\left(\mu_{1}^2\ps{V_{\gamma}(s-x)V_{\gamma}(s-\eps)\V}_{\delta,\eps}+\mu_{2}^2\ps{V_{\gamma}(s+x)V_{\gamma}(s+\eps)\V}_{\delta,\eps}\right)dx\\
			-&\frac{\gamma\beta}2\int_{\eps}^u\frac{\mu_1\mu_2}{x}\left(\ps{V_{\gamma}(s+x)V_{\gamma}(s-\eps)\V}_{\delta,\eps}+\ps{V_{\gamma}(s-x)V_{\gamma}(s+\eps)\V}_{\delta,\eps}\right)dx.
		\end{align*}
		
		Finally, recollecting terms we see that we can write that, if $\mu=0$,
		\begin{align*}
			&\lim\limits_{\rho\to0}\partial_{s_1}^2\ps{\prod_{k=1}^NV_{\alpha_k}(z_k)\prod_{l=1}^MV_{\beta_l}(s_l)}_{\delta,\eps,\rho}=\quad\text{almost regular terms}\\
			&-\frac{\gamma\beta_1}{2\eps}\left(\ps{V_{\gamma}(s_1-\eps)\V}_{\delta,\eps}\mu_{1}+\ps{V_{\gamma}(s_1+\eps)\V}_{\delta,\eps}\mu_{2}\right)\\
			+&\frac{\gamma\beta}2\int_{\eps}^u\frac1{x}\left(\mu_{1}^2\ps{V_{\gamma}(s_1-x)V_{\gamma}(s_1-\eps)\V}_{\delta,\eps}+\mu_{2}^2\ps{V_{\gamma}(s_1+x)V_{\gamma}(s_1+\eps)\V}_{\delta,\eps}\right)dx\\
			-&\frac{\gamma\beta}2\int_{\eps}^u\frac{\mu_{1}\mu_{2}}{x}\left(\ps{V_{\gamma}(s_1+x)V_{\gamma}(s_1-\eps)\V}_{\delta,\eps}+\ps{V_{\gamma}(s_1-x)V_{\gamma}(s_1+\eps)\V}_{\delta,\eps}\right)dx.
		\end{align*}
		
		Let us now assume that $\mu>0$ and denote $\H_v\coloneqq\Heps\setminus \left(s_1+v\D\right)$ and $\H_c\coloneqq\Heps\cap \left(s_1+v\D\right)$ with $v>u$. Then
		\begin{align*}
			&\lim\limits_{\rho\to0}\partial_{s_1}^2\ps{\prod_{k=1}^NV_{\alpha_k}(z_k)\prod_{l=1}^MV_{\beta_l}(s_l)}_{\delta,\eps,\rho}^{\mu>0}-\partial_{s_1}^2\ps{\prod_{k=1}^NV_{\alpha_k}(z_k)\prod_{l=1}^MV_{\beta_l}(s_l)}^{\mu=0}_{\delta,\eps,\rho}\\
			&=-\mu\int_{\H_v}\Re\left(\frac{\gamma \beta_1\left(1+\frac{\gamma \beta_1}2\right)}{(x-s_1)^2}+\sum_{k\neq 2N+1}\frac{\alpha_k\beta_1\gamma\beta_1}{(z_k-s_1) (x-s_1)}+\frac{\gamma^2\beta_1^2}{4(s_1-x)(s_1-\bar x)}\right)\ps{V_{\gamma}(x)\V}_{\delta,\eps}\norm{dx}^2\\
			&-\mu\int_{\H_c}\left(\partial_x+\partial_{\bar x}\right)\left(\frac{\gamma\beta_1}{2(s_1-x)}+\frac{\gamma\beta_1}{2(s_1-\bar x)}\right)\ps{V_{\gamma}(x)\V}_{\delta,\eps}\norm{dx}^2\\
			&+\text{ two-fold integrals}.
		\end{align*}
		The integration by parts yield in the limit the boundary term 
		\[
		\mu\int_{\H\cap \partial B(s_1,v)}\left(\frac{\gamma\beta_1}{2(s_1-\xi)}+\frac{\gamma\beta_1}{2(s_1-\bar \xi)}\right)\ps{V_{\gamma}(\xi)\V}\frac{id\bar \xi-id\xi}{2}
		\]
		which is analytic in $\bm\alpha\in\mc A_{N,M}$. As for the two-fold integrals, they are given by
		\begin{align*}
			&-\mu^2\int_{\Heps^2}\Re\left(\frac{\gamma\beta}{s_1-x}\right)\Re\left(\frac{\gamma\beta}{s_1-y}\right)\ps{V_{\gamma}(x)V_\gamma(y)\V}\norm{dx}^2\norm{dy}^2\\
			&-\mu^2\int_{\H_c\times\Heps}\Re\left(\frac{\gamma\beta}{s_1-x}\right)\Re\left(\frac{\gamma^2}{y-x}+\frac{\gamma^2}{y-\bar x}\right)\ps{V_{\gamma}(x)V_\gamma(y)\V}\norm{dx}^2\norm{dy}^2\\
			&-2\mu\int_{\Heps\times\Reps}\left(\frac{\gamma\beta}{2(s_1-x)}+\frac{\gamma\beta}{2(s_1-\bar x)}\right)\frac{\gamma\beta}{2(y-s_1)}\norm{dx}^2\mu_{\partial}(dy)\\
			&-\mu\int_{\H_c\times\Reps}\left(\frac{\gamma\beta}{2(s_1-x)}+\frac{\gamma\beta}{2(s_1-\bar x)}\right)\left(\frac{\gamma^2}{2(y-x)}+\frac{\gamma^2}{2(y-\bar x)}\right)\ps{V_{\gamma}(x)V_\gamma(y)\V}\norm{dx}^2\mu_{\partial}(dy)\\
			&-\mu\int_{\H_c\times\Reps}\frac{\gamma\beta}{2(s_1-y)}\left(\frac{\gamma^2}{2(x-y)}+\frac{\gamma^2}{2(\bar x-y)}\right)\ps{V_{\gamma}(x)V_\gamma(y)\V}\norm{dx}^2\mu_{\partial}(dy).
		\end{align*}
		The method is the exactly the same as in the $\mu=0$ case: for the integral over $\H_c\times\H_v$ we write $\H_c=\H_c'\sqcup \H_c\cap\left(s_1+v'\D\right)$ and use integration by parts. Like for the case of one derivative this yields almost regular terms. The same applies for the integral over $\H_c\times\R_u$ so that we only need to consider the integrals over $\H_c^2$ and $\H_c\times\R_c$. For instance the integral over $\H_c\times\R_c$ is equal to
		\begin{align*}
			&\mu\int_{\H_c\times\R_c}\left(\frac{\gamma\beta}{2(s_1-x)}+\frac{\gamma\beta}{2(s_1-\bar x)}\right)\partial_{y}\ps{V_{\gamma}(x)V_\gamma(y)\V}\norm{dx}^2\mu_{\partial}(dy)\\
			&+\mu\int_{\H_c\times\R_c}\frac{\gamma\beta}{2(s_1-y)}\left(\partial_x+\partial_{\bar x}\right)\ps{V_{\gamma}(x)V_\gamma(y)\V}\norm{dx}^2\mu_{\partial}(dy)
		\end{align*}
		up to almost regular terms. The first line is given by
		\begin{align*}
			&\mu\int_{\H_c}\left(\frac{\gamma\beta}{2(s_1-x)}+\frac{\gamma\beta}{2(s_1-\bar x)}\right)\left(\mu_{1}\ps{V_{\gamma}(x)V_\gamma(s_1-\eps)\V}-\mu_2\ps{V_{\gamma}(x)V_\gamma(s_1+\eps)\V}\right)\norm{dx}^2\\
			&+\mu\int_{\H_c}\left(\frac{\gamma\beta}{2(s_1-x)}+\frac{\gamma\beta}{2(s_1-\bar x)}\right)\left(\mu_{2}\ps{V_{\gamma}(x)V_\gamma(s_1+u)\V}-\mu_{1}\ps{V_{\gamma}(x)V_\gamma(s_1-u)\V}\right)\norm{dx}^2
		\end{align*}
		where the last term is handled like a first order derivative by integration by parts. As for the first term in the above equation, by integration by parts in $x$ this is a remainder term scaling like $\eps^{-\frac{\gamma\beta}2}$.
		Likewise we have
		\begin{align*}
			&\mu\int_{\H_c\times\R_c}\frac{\gamma\beta}{2(s_1-y)}\left(\partial_x+\partial_{\bar x}\right)\ps{V_{\gamma}(x)V_\gamma(y)\V}\norm{dx}^2\mu_{\partial}(dy)\\
			&=\mu\int_{\H\cap\partial B(s_1,v)\times\R_c}\frac{\gamma\beta}{2(s_1-y)}\ps{V_{\gamma}(\xi)V_\gamma(y)\V}\frac{id\xi-id\bar x}2\mu_{\partial}(dy)
		\end{align*}
		that will yield almost regular terms.
		The integral over $\H_c\times \H_c$ is handled in a similar, though easier, way. It is seen to yield only almost regular terms.
	\end{proof}
	
	\subsection{On the remainder terms}\label{subsec:remainder}
	We now provide some additional details on the remainder terms that appear in the definition of the descendant fields. More precisely we want to describe their behavior in the $\delta,\eps\to0$ limit. 
	\subsubsection{The $\mathfrak R_{-2}$ remainder term}
	Our first result concerns the remainder term $\mathfrak{R}_{-2}$ that shows up in the definition of the $\L_{-2}$ descendant and defined in Equation~\eqref{eq:rem2}. We show that:
	\begin{lemma}\label{lemma:remainderL2}
		As $\delta,\eps\to0$, we can write that 
		\begin{equation}
			\mathfrak R_{-2}(\delta,\beta,\bm\alpha)=\frac1{\eps}\Big(\ps{V_{\gamma}(t-\eps)V_\beta(t)\V}_{\delta,\eps}\mu_{\partial}(t-\eps)+\ps{V_{\gamma}(t+\eps)V_\beta(t)\V}_{\delta,\eps}\mu_{\partial}(t+\eps)\Big) +\tilde{\mathfrak R}_\delta,
		\end{equation}
		where the remainder term $\tilde{\mathfrak R}_\delta$ is such that:
		\begin{enumerate}
			\item if $-\frac1\gamma<\beta<Q-2\gamma$ then
			\begin{equation}
				\tilde{\mathfrak R}_\delta=\delta^{-\gamma(\beta+\frac\gamma2)}\times \left(\mu 2^{-\gamma(\beta+\frac\gamma2)}\sin\left(\pi\frac{\gamma\beta}{2}\right)\frac{\Gamma\left(-\frac{\gamma\beta}{2}\right)\Gamma\left(1+\gamma\beta\right)}{\Gamma\left(1+\frac{\gamma\beta}{2}\right)}\ps{V_{2\gamma+\beta}(t)\V}+o(1)\right);
			\end{equation}
			\item if $-\frac1\gamma\vee(Q-2\gamma)<\beta$ then $\tilde{\mathfrak R}_\delta=o\left(\delta^{-\gamma(\beta+\frac\gamma2)}\right)$;
			\item if $\beta<-\frac1\gamma$ and $\gamma^2<\frac43$ then 
			\begin{equation*}
				\tilde{\mathfrak R}_\delta=\delta^{1-\frac{\gamma^2}2}\times \left(-2^{-\frac{\gamma^2}2}\mu\int_\R\frac{1}{x^2}\ps{V_{2\gamma}(t+x)V_\beta(t)\V}dx+o(1)\right);
			\end{equation*}
			\item if $\beta<-\frac1\gamma$ and $\gamma^2>\frac43$ then for some $\eta>0$:
			\begin{equation*}
				\tilde{\mathfrak R}_\delta=o\left(\delta^{\eta}\right).
			\end{equation*}
		\end{enumerate}
	\end{lemma}
	\begin{proof}
		To start with recall that from Lemma~\ref{lemma:desc_two_L} 
		\[
		\tilde{\mathfrak R}_\delta=-\mu\delta\int_{\R}\frac{\ps{V_{\gamma}(x+i\delta)V_\beta(t)\V}_{\delta,\eps,\rho}}{(t-x)^2+\delta^2}dx.
		\]
		
		In the case where $\beta>-\frac1\gamma$ we make the change of variable $x\leftrightarrow t+\delta x$ to end up with
		\[
		\tilde{\mathfrak R}_\delta=-\delta^{-\gamma(\beta+\frac\gamma2)}\times 2^{-\frac{\gamma^2}2}\mu\int_{\R}(1+x^2)^{-1-\frac\gamma2\beta}F_\delta(x)dx
		\]
		where $F_\delta(x)\coloneqq\norm{u-t}^{\gamma\beta}\norm{u-\overline u}^{\frac{\gamma^2}2}\ps{V_\gamma(u)V_\beta(t)\V}$, $u=t+\delta(x+i)$. It is readily seen thanks to the probabilistic representation that under the assumption that $2\gamma+\beta<Q$, the function $F_\delta(x)$ converges to $\ps{V_{2\gamma+\beta}(t)\V}$ as $\delta\to0$; on the contrary if $2\gamma+\beta>Q$ then $\lim\limits_{\delta\to0}F_\delta(x)=0$. By using this fact we claim that we end up with
		\[
		\tilde{\mathfrak R}_\delta=-\delta^{-\gamma(\beta+\frac\gamma2)}\times \left(2^{-\frac{\gamma^2}2}\mu\int_{\R}(1+x^2)^{-1-\frac\gamma2\beta}dx\ps{V_{2\gamma+\beta}(t)\V}+o(1)\right)
		\]
		if $\beta<2Q-2\gamma$ and $\tilde{\mathfrak R}_\delta=o\left(\delta^{-\gamma(\beta+\frac\gamma2)}\right)$ otherwise. The integral is evaluated in Lemma~\ref{lemma:eval_intR}, concluding for the proof for items $(1)$ and $(2)$. Let us now justify that in the above we can indeed take the $\delta\to0$ limit as claimed. To see why this is indeed the case, we can split for $r>0$ small enough the integral between $(-r\delta^{-1},r\delta^{-1})$ and $\R\setminus(-r\delta^{-1},r\delta^{-1})$. Then we have
		\begin{align*}
			\tilde{\mathfrak R}_\delta&=-\delta^{-\gamma(\beta+\frac\gamma2)}\times \left(2^{-\frac{\gamma^2}2}\mu\int_{(-r \delta^{-1},r \delta^{-1})}(1+x^2)^{-1-\frac\gamma2\beta}dx\ps{V_{2\gamma+\beta}(t)\V}\right)\\
			&-\delta^{-\gamma(\beta+\frac\gamma2)}\times \left(2^{-\frac{\gamma^2}2}\mu\int_{(-r \delta^{-1},r \delta^{-1})}(1+x^2)^{-1-\frac\gamma2\beta}\left(F_\delta(x)-\ps{V_{2\gamma+\beta}(t)\V}\right)dx\right)\\
			&-\delta^{-\gamma(\beta+\frac\gamma2)}\times \left(2^{-\frac{\gamma^2}2}\mu\int_{\R\setminus(-r \delta^{-1},r \delta^{-1})}(1+x^2)^{-1-\frac\gamma2\beta}F_\delta(x)dx\right).
		\end{align*}
		Since $\beta>-\frac1\gamma$ the first integral is convergent and yields the desired result in the $\delta\to0$ limit. For the second integral we can bound the integrand by the constant $$C_r\coloneqq \sup_{x\in(-r \delta^{-1},r \delta^{-1})}\norm{F_\delta(x)-\ps{V_{2\gamma+\beta}(t)\V}},$$ independent of $\delta$. 
		As for the one ranging over $\R\setminus(-r \delta^{-1},r \delta^{-1})$, by following the same reasoning as in the proof of Lemma~\ref{lemma:inf_integrability} we see that we likewise have the uniform bound 
		\[
		\norm{F_\delta(x)}\leq C_r'\left(1+\delta \norm{x}\right)^{\gamma(\beta+\frac\gamma2)-4}
		\]
		for some constant $C_r'$ independent of $\delta$ and as soon as $x$ stays at positive distance $h$ from the boundary insertions. Around these insertions we likewise have a bound independent of $\delta$ derived from Lemma~\ref{lemma:fusion} ensuring absolute convergence of the integral. From this we deduce
		\[
		\lim\limits_{\delta\to 0}\int_{\R\setminus(-r \delta^{-1},r \delta^{-1})}(1+x^2)^{-1-\frac\gamma2\beta}F_\delta(x)dx=0.
		\]
		As a consequence recollecting terms we end up with the bound
		\[
		\limsup\limits_{\delta\to0} \norm{\int_\R(1+x^2)^{-1-\frac\gamma2\beta}(F_\delta(x)-\ps{V_{2\gamma+\beta}(t)\V})dx}\leq C_r\int_\R(1+x^2)^{-1-\frac\gamma2\beta}dx
		\]
		for any positive $r$. Now using the probabilistic representation~\eqref{eq:ana_correl} of the correlation functions it is readily seen that $\lim\limits_{r\to0}C_r=0$. As a consequence we can take $r$ to $0$ to get the desired result. 
		
		Let us now assume that $\beta<-\frac1\gamma$. If $\gamma^2<\frac43$ we simply write that
		\[
		\tilde{\mathfrak R}_\delta\sim-\delta^{1-\frac{\gamma^2}2}\times 2^{\frac{\gamma^2}2}\mu\int_{\R}(t-x)^{-2}G_\delta(x)dx
		\]
		with $G_\delta(x)\coloneqq\norm{u-\overline u}^{\frac{\gamma^2}2}\ps{V_\gamma(u)V_\beta(t)\V}$, $u=x+\delta i$. Note that the singularity at $x=t$ is integrable via our assumption that $\beta<-\frac1\gamma$. In the same fashion as above $G_\delta(x)$ converges to $\ps{V_{2\gamma}(x)V_\beta(t)\V}$ so that we have
		\[
		\tilde{\mathfrak R}_\delta=\delta^{1-\frac{\gamma^2}2}\times \left(2^{-\frac{\gamma^2}2}\mu\int_{\R}(t-x)^{-2}\ps{V_{2\gamma}(x)V_\beta(t)\V}dx+o(1)\right)
		\]
		as expected from item $(3)$. To see why we can use the same reasoning as for $F_\delta$, that is we first split for $r>0$ small enough the integral between $(t-r,t+r)$ and $\R\setminus(t-r,t+r)$. Over $(t-r,t+r)$ we can factorize, using the definition of the correlation functions~\eqref{eq:ana_correl}, $G_\delta$ as the product of a prefactor $P_\delta$ (which is easily dealt with) and an integral $\mc I_\delta$ (for which we have uniform convergence since $\gamma^2<\frac43$). As for $\R\setminus(t-r,t+r)$ we have bounds at infinity, uniform in $\delta$, from the reasoning conducted in the proof of Lemma~\ref{lemma:inf_integrability}. Now let us assume that $\gamma^2>\frac43$. We can first bound for any positive $R$ and $\eta>0$ small enough, based on the fusion asymptotics,
		\[
		\delta\int_{\delta^{\eta}}^R\frac{\ps{V_{\gamma}(t+x+i\delta)V_\beta(t)\V}_{\delta,\eps,\rho}}{x^2+\delta^2}dx\leq \delta^{1-\frac{\gamma^2}2+(\gamma-\frac Q2)^2-3\eta},
		\]
		where the exponent is positive. The same reasoning applies for $x$ close to $0$ so we only need to take care of the case where $\norm{x}>R$. Then based on the probabilistic representation we can infer that $\delta^{(\gamma-\frac Q2)^2-\eta}G_\delta(x)$ remains uniformly bounded in $x$ as $\delta\to 0$, concluding for this case too.   
	\end{proof}
	
	\subsubsection{The $\mathfrak R_{-(1,1)}$ remainder term}
	We now turn to the $\mc L_{-(1,1)}$ descendant. In that case we will need to distinguish between the cases where $\gamma<\sqrt2$ and the one where $\gamma>\sqrt2$. We start with the following:
	\begin{lemma}\label{lemma:remainderL11}
		Assume that $\gamma<\sqrt2$. Then as $\delta,\eps\to0$, we can write that 
		\begin{equation}
			\mathfrak{R}_{-(1,1)}(\delta,\eps,\bm\alpha)=-\frac{\gamma\beta}{2\eps}\left(\ps{V_{\gamma}(t-\eps)\V}_{\delta,\eps}\mu_{L}+\ps{V_{\gamma}(t+\eps)\V}_{\delta,\eps}\mu_{R}\right)+\tilde{\mathfrak R}_\eps,
		\end{equation}
		where the remainder term $\tilde{\mathfrak R}_\eps$ is such that:
		\begin{enumerate}
			\item if $-\gamma<\beta<Q-2\gamma$ then 
			\begin{equation}
				\tilde{\mathfrak R}_\delta=\eps^{-\gamma(\beta+\frac\gamma2)}\times \left(\mu_L^2+\mu_R^2-2\mu_L\mu_R\cos(\pi\frac{\gamma}{2}\sigma(\beta))\right)\frac{\Gamma(\frac\gamma2(\beta+\gamma))\Gamma(1+\frac{\gamma^2}{2})}{\Gamma(1+\frac{\gamma\beta}{2})}\ps{V_{2\gamma+\beta}(t)\V}
			\end{equation}
			with $\sigma(\beta)\in\R$ and such that $\sigma(-\frac{\gamma}{2})=-\frac\gamma2$. 
			\item if $-\gamma\vee(Q-2\gamma)<\beta$ then $\tilde{\mathfrak R}_\eps=o\left(\eps^{-\gamma(\beta+\frac\gamma2)}\right)$;
			\item if $\beta<-\gamma$ then $\tilde{\mathfrak R}_\eps\sim C\eps^{-\frac{\gamma}2\beta}$ for some $C\in\R$.
		\end{enumerate}
	\end{lemma}
	\begin{proof}
		To start with from the proof of Lemma~\ref{lemma:desc2_L} we see that the remainder term $\mathfrak{R}_{-(1,1)}(\delta,\eps,\bm\alpha)$ has an expansion of the form
		\begin{equation}\label{eq:rem11}
			\begin{split}
				&\mathfrak{R}_{-(1,1)}(\delta,\eps,\bm\alpha)=-\frac{\gamma\beta}{2\eps}\left(\ps{V_{\gamma}(t-\eps)\V}_{\delta,\eps}\mu_{L}+\ps{V_{\gamma}(t+\eps)\V}_{\delta,\eps}\mu_{R}\right)\\
				+&\frac{\gamma\beta}2\int_{\eps}^u\frac1{x}\left(\mu_{L}^2\ps{V_{\gamma}(t-x)V_{\gamma}(t-\eps)\V}_{\delta,\eps}+\mu_{R}^2\ps{V_{\gamma}(t+x)V_{\gamma}(t+\eps)\V}_{\delta,\eps}\right)dx\\
				-&\frac{\gamma\beta}2\int_{\eps}^u\frac{\mu_{L}\mu_{R}}{x}\left(\ps{V_{\gamma}(t+x)V_{\gamma}(t-\eps)\V}_{\delta,\eps}+\ps{V_{\gamma}(t-x)V_{\gamma}(t+\eps)\V}_{\delta,\eps}\right)dx+\check{\mathfrak{R}}_\eps
			\end{split}
		\end{equation} 
		where $\check{\mathfrak{R}}_\eps$ scales at most like $\eps^{-\frac{\gamma\beta}{2}}$.
		
		Let us start by treating items $(1)$ and $(2)$ and make the change of variable $x\leftrightarrow \eps x$ to get
		\begin{align*}
			\tilde{\mathfrak{R}}_\eps=\eps^{-\gamma(\beta+\frac\gamma2)}&\times \frac{\gamma\beta}2\int_{1}^{\frac u\eps}x^{-1-\frac{\gamma\beta}{2}}\\
			&\left((x+1)^{-\frac{\gamma^2}{2}}\mu_L\mu_R\left(F_\eps^{L,R}(x)+F_\eps^{L,R}(x)\right)-(x-1)^{-\frac{\gamma^2}{2}}\left(\mu_L^2F_\eps^L(x)+\mu_R^2F_\eps^R(x)\right)\right)dx,
		\end{align*}
		where we have introduced $F_\eps^{L,R}(x)\coloneqq \norm{\eps x+\eps}^{\frac{\gamma^2}2}\norm{\eps x}^{\frac{\gamma\beta}2}\norm{\eps}^{\frac{\gamma\beta}2}\ps{V_\gamma(t+\eps x)V_\gamma(t-\eps)V_\beta(t)\V}$, $F_\eps^{L}(x)\coloneqq \norm{\eps x-\eps}^{\frac{\gamma^2}2}\norm{\eps x}^{\frac{\gamma\beta}2}\norm{\eps}^{\frac{\gamma\beta}2}\ps{V_\gamma(t-\eps x)V_\gamma(t-\eps)V_\beta(t)\V}$ and likewise for $F_\eps^{R,L}$ and $F^R_\eps$.
		Under the assumptions of item $(1)$ $F_\eps(x)$ converges to $\ps{V_{2\gamma+\beta}(t)\V}$ so that based on the same reasoning as the one used in Lemma~\ref{lemma:remainderL2} (with $F_\eps$ playing the counterpart of $F_\delta$ there) we arrive to
		\begin{align*}
			\tilde{\mathfrak{R}}_\eps\sim\eps^{-\gamma(\beta+\frac\gamma2)}&\times \frac{\gamma\beta}2\int_{1}^{+\infty}x^{-1-\frac{\gamma\beta}{2}}\left((x+1)^{-\frac{\gamma^2}{2}}2\mu_L\mu_R-(x-1)^{-\frac{\gamma^2}{2}}\left(\mu_L^2+\mu_R^2\right)\right)dx\ps{V_{2\gamma+\beta}(t)\V}.
		\end{align*}
		This integral is well-defined thanks to the assumptions that $\gamma^2<2$ and $\beta>-\gamma$. It is computed in Lemma~\ref{lemma:eval_int1} and yields the expected result. If we assume that $\beta>Q-2\gamma$ then $F_\eps$ converges to $0$ so that along the same lines as in Lemma~\ref{lemma:remainderL2} $\tilde{\mathfrak{R}}_\eps=o\left(\eps^{-\gamma(\beta+\frac\gamma2)}\right)$. This proves item $(2)$ too.
		
		If we now assume that $\beta<-\gamma$ then by similar arguments we have  
		\begin{align*}
			\tilde{\mathfrak{R}}_\eps\sim&\eps^{-\frac{\gamma\beta}2}\frac{\gamma\beta}2\int_{0}^u\frac1{x}\left(\mu_{L}^2\ps{V_{\gamma}(t-x)V_{\gamma+\beta}(t)\V}+\mu_{R}^2\ps{V_{\gamma}(t+x)V_{\gamma+\beta}(t)\V}\right)dx\\
			-&\frac{\gamma\beta}2\int_{0}^u\frac{\mu_{L}\mu_{R}}{x}\left(\ps{V_{\gamma}(t+x)V_{\gamma+\beta}(t)\V}+\ps{V_{\gamma}(t-x)V_{\gamma+\beta}(t)\V}\right)dx
		\end{align*}
		with the singularity at $x=0$ being integrable, showing item $(3)$.
	\end{proof}
	
	We now turn to the case where $\gamma>\sqrt 2$. We have the following analog of Lemma~\ref{lemma:remainderL11}:
	\begin{lemma}\label{lemma:remainderL11_bis}
		Assume that $\gamma>\sqrt2$. Then as $\delta,\eps\to0$, we can write that 
		\begin{equation}
			\mathfrak{R}_{-(1,1)}(\delta,\eps,\bm\alpha)=-\frac{\gamma\beta}{2\eps}\left(\ps{V_{\gamma}(t-\eps)\V}_{\delta,\eps}\mu_{L}+\ps{V_{\gamma}(t+\eps)\V}_{\delta,\eps}\mu_{R}\right)+\tilde{\mathfrak R}_\eps,
		\end{equation}
		where the remainder term $\tilde{\mathfrak R}_\eps$ is such that: 
		\begin{enumerate}
			\item if $-\gamma\vee(Q-2\gamma)<\beta$ then for some positive $\eta$, $\tilde{\mathfrak R}_\eps=o\left(\eps^{-\gamma(\beta+\frac\gamma2)+\eta}\right)$;
			\item if $\beta<-\gamma$ then $\tilde{\mathfrak R}_\eps\sim C\eps^{-\frac{\gamma}2\beta}$ for some $C\in\R$.
		\end{enumerate}
	\end{lemma}
	\begin{proof}
		Item $(3)$ is treated in the exact same way as item $(3)$ in the proof of Lemma~\ref{lemma:remainderL11} so that we can treat the cases with $\beta>-\gamma$. For the integral term proportional to $\mu_L\mu_R$ the reasoning remains valid, but for the terms containing $\mu_L^2$ and $\mu_R^2$ the method breaks down because of the singularity at $x=\eps$. However we can use the same reasoning as in the proof of item $(4)$ in Lemma~\ref{lemma:remainderL2} to understand the behavior of this integral. Namely we first show that for $x$ close to $0$
		\begin{align*}
			\int_{\eps}^{\eps^\eta}\frac1{x}\ps{V_{\gamma}(t-x)V_{\gamma}(t-\eps)\V}dx\leq C\eps^{-\frac{\gamma\beta}{2}}\int_{\eps}^{\eps^\eta}x^{-1-\frac{\gamma\beta}{2}}(x-\eps)^{-\frac{\gamma^2}{2}+\frac12(\beta+2\gamma-Q)^2-\eta}dx
		\end{align*}which is a $o\left(\eps^{-\gamma(\beta+\frac\gamma2)+\eta}\right)$. The same applies away from the origin.
	\end{proof}
	
	\subsection{Higher equations of motion}\label{subsec:bpzhem}
	Assume that $\beta=-\chi$ where $\chi\in\{\frac\gamma2,\frac2\gamma\}$. Then on the one hand we can write that 
	\[
	\L_{-2}V_\beta(t)[\Phi]= -\frac1{\chi^2}:\left(-\chi\partial^2\Phi(t)+:\left(\chi\partial\Phi(t)\right):^2\right)V_\beta(t):.
	\]
	On the other hand we see from Lemma~\ref{lemma:GaussianIPP} (see also~\cite[Subsection 5.2]{Toda_OPEWV}) that in order to define the $\L_{-(1,1)}$ descendant of $V_{-\chi}$ we used the functional $\left(-\chi\partial^2\Phi(t)+:\left(\chi\partial\Phi(t)\right):^2\right)$ since 
	\begin{align*}
		&\partial_t^2\ps{V_{-\chi}(t)\prod_{k=1}^NV_{\alpha_k}(z_k)\prod_{l=1}^MV_{\beta_l}(s_l)}_{\delta,\eps,\rho}\\
		&=\ps{:\left(-\chi\partial^2\Phi(t)+:\left(\chi\partial\Phi(t)\right):^2\right)V_{-\chi}(t):\prod_{k=1}^NV_{\alpha_k}(z_k)\prod_{l=1}^MV_{\beta_l}(s_l)}_{\delta,\eps,\rho}.
	\end{align*}
	Based on this simple identification we are in position to prove the BPZ differential equations. We will distinguish between the two cases $\chi=-\frac2\gamma$ and $\chi=-\frac\gamma2$.
	
	\subsubsection{The case $\chi=-\frac2\gamma$}
	We start with the simplest case, that is when $\chi=-\frac2\gamma$.
	\begin{theorem}\label{thm:bpz_set}
		Under the assumptions of Proposition~\ref{prop:analycity}, in the weak sense:
		\begin{equation}\label{eq:bpz_set}
			\left(\frac{\gamma^2}4\partial_t^2+\sum_{k=1}^{2N+M}\frac{\partial_{z_k}}{t-z_k}+\frac{\Delta_{\alpha_k}}{(t-z_k)^2}\right)\ps{V_{-\frac2\gamma}(t)\V}=\left(1-\frac{\gamma^2}{4}\right)\left(\mu_L+\mu_R\right)\ps{V_{\gamma-\frac2\gamma}(s_1)\V}.
		\end{equation}
	\end{theorem}
	\begin{proof}
		As explained above, we can write that for any positive $\delta,\eps,\rho$,
		\[
		-\frac1{\chi^2}\partial_t^2\ps{V_{-\chi}(t)\prod_{k=1}^NV_{\alpha_k}(z_k)\prod_{l=1}^MV_{\beta_l}(s_l)}_{\delta,\eps,\rho}=\ps{\L_{-2}V_{-\chi}(t)\prod_{k=1}^NV_{\alpha_k}(z_k)\prod_{l=1}^MV_{\beta_l}(s_l)}_{\delta,\eps,\rho}.
		\]
		Now let us consider the remainder terms that appear in the definition of the second order derivative in Lemma~\ref{lemma:desc2_L} and of the descendant $\mc L_{-2}$ in Lemma~\ref{lemma:desc_two_L}.
		By definition and using the previous equality we can infer that 
		\begin{align*}
			&\frac1{\chi^2}\ps{\L_{-(1,1)}V_{-\chi}(t)\prod_{k=1}^NV_{\alpha_k}(z_k)\prod_{l=1}^MV_{\beta_l}(s_l)}+\ps{\L_{-2}V_{-\chi}(t)\prod_{k=1}^NV_{\alpha_k}(z_k)\prod_{l=1}^MV_{\beta_l}(s_l)}\\
			&=\lim\limits_{\delta,\eps,\rho\to0}\frac1{\chi^2}\mathfrak{R}_{-(1,1)}(\delta,\eps,\bm\alpha)+\mathfrak{R}_{-2}(\delta,\eps,\bm\alpha).
		\end{align*}
		To conclude it remains to use the explicit expressions provided for these remainder terms. Indeed from Lemma~\ref{lemma:remainderL2} we know that
		\begin{align*}
			\mathfrak{R}_{-2}(\delta,\eps,\bm\alpha)=\frac1\eps\left(\ps{V_{\gamma}(t-\eps)V_{-\frac2\gamma}(t)\V}_{\delta,\eps,\rho}\mu_{L}+\ps{V_{\gamma}(t+\eps)V_{-\frac2\gamma}(t)\V}_{\delta,\eps,\rho}\mu_{R}\right)+o(1)
		\end{align*}
		while thanks to Lemma~\ref{lemma:remainderL11} we have
		\begin{align*}
			\mathfrak{R}_{-(1,1)}(\delta,\eps,\bm\alpha)=-\frac1\eps\left(\ps{V_{\gamma}(t-\eps)V_{-\frac2\gamma}(t)\V}_{\delta,\eps,\rho}\mu_{L}+\ps{V_{\gamma}(t+\eps)V_{-\frac2\gamma}(t)\V}_{\delta,\eps,\rho}\mu_{R}\right)+o(1).
		\end{align*}
		Now it is readily seen that $\lim\limits_{\delta,\eps,\rho\to0}\frac1\eps\ps{V_{\gamma}(t+\eps)V_{-\frac2\gamma}(t)\V}_{\delta,\eps,\rho}=\ps{V_{\gamma-\frac2\gamma}(t)\V}_{\delta,\eps,\rho}$; therefore
		\begin{align*}
			&\lim\limits_{\delta,\eps,\rho\to0}\frac1{\chi^2}\mathfrak{R}_{\delta,\eps,\rho}^{(1,1)}+\mathfrak{R}_{\delta,\eps,\rho}^{(2)}=\left(1-\frac{\gamma^2}{4}\right)\left(\mu_L+\mu_R\right)\ps{V_{\gamma-\frac2\gamma}(s_1)\V}.			
		\end{align*}
		This allows to conclude that, as desired,
		\begin{align*}
			&\left(\frac1{\chi^2}\Lc_{-(1,1)}+\Lc_ {-2}\right)\ps{V_{-\chi}(t)\prod_{k=1}^NV_{\alpha_k}(z_k)\prod_{l=1}^MV_{\beta_l}(s_l)}\\
			&=\left(1-\frac{\gamma^2}{4}\right)\left(\mu_L+\mu_R\right)\ps{V_{\gamma-\frac2\gamma}(s_1)\V}.
		\end{align*}
	\end{proof}

	\subsubsection{The case $\chi=-\frac\gamma2$ with $\gamma<\sqrt2$}
	We now turn to the case where $\chi=-\frac\gamma2$ which is more involved than the previous one. But to start with we consider the case $\gamma<\sqrt2$ which remains simpler than when $\gamma\geq\sqrt2$:
	\begin{theorem}\label{thm:bpz_gamma1}
		Assume that $\gamma<\sqrt2$ and that the requirements of Proposition~\ref{prop:analycity} are met. Then, in the weak sense:
		\begin{equation}\label{eq:bpz_gamma1}
			\begin{split}
				&\left(\frac4{\gamma^2}\partial_t^2+\sum_{k=1}^{2N+M}\frac{\partial_{z_k}}{t-z_k}+\frac{\Delta_{\alpha_k}}{(t-z_k)^2}\right)\ps{V_{-\frac\gamma2}(t)\V}=\\
				&\left(\mu_L^2+\mu_R^2-2\mu_L\mu_R\cos\left(\frac{\pi\gamma^2}{4}\right)-\mu\sin\left(\frac{\pi\gamma^2}{4}\right)\right)\frac{\Gamma\left(\frac{\gamma^2}{4}\right)\Gamma\left(1-\frac{\gamma^2}{2}\right)}{\Gamma\left(1-\frac{\gamma^2}{4}\right)}\ps{V_{\frac{3\gamma}2}(t)\V}.
			\end{split}
		\end{equation}
	\end{theorem}
	\begin{proof} 
		We proceed in the same fashion as in the case where $\chi=-\frac2\gamma$ by considering the remainder terms involved. To start with from Lemma~\ref{lemma:remainderL2} we have 
		\begin{align*}
			&\mathfrak{R}_{-2}(\delta,\eps,\bm\alpha)=\frac1\eps\left(\ps{V_{\gamma}(t-\eps)\V}_{\delta,\eps,\rho}\mu_{L}+\ps{V_{\gamma}(t+\eps)\V}_{\delta,\eps,\rho}\mu_{R}\right)\\
			&-\mu\sin\left(\pi\frac{\gamma^2}{4}\right)\frac{\Gamma\left(\frac{\gamma^2}{4}\right)\Gamma\left(1-\frac{\gamma^2}{2}\right)}{\Gamma\left(1-\frac{\gamma^2}{4}\right)}\ps{V_{\frac{3\gamma}2}(t)\V}+o(1).
		\end{align*}
		Indeed since $\gamma<\sqrt2$ we are in the case where $-\frac1\gamma<\beta<Q-2\gamma$.
		Likewise from Lemma~\ref{lemma:desc2_L}:
		\begin{align*}
			&\mathfrak{R}_{-(1,1)}(\delta,\eps,\bm\alpha)=\frac{\gamma^2}4\frac1\eps\left(\ps{V_{\gamma}(s_1-\eps)\V}_{\delta,\eps,\rho}\mu_{L}+\ps{V_{\gamma}(s_1+\eps)\V}_{\delta,\eps,\rho}\mu_{R}\right)\\
			+&\left(\mu_L^2+\mu_R^2-2\mu_L\mu_R\cos\left(\frac{\pi\gamma^2}{4}\right)\right)\frac{\Gamma\left(\frac{\gamma^2}{4}\right)\Gamma\left(1-\frac{\gamma^2}{2}\right)}{\Gamma\left(1-\frac{\gamma^2}{4}\right)}\ps{V_{\frac{3\gamma}2}(t)\V}+o(1).
		\end{align*}
		We see that the terms that appear in the first lines in the expression of the remainder terms will compensate each other, so that the only remaining terms will be those on the second lines. This shows that putting these two expansions together yields, as desired,
		\begin{align*}
			&\frac4{\gamma^2}\mathfrak{R}_{-(1,1)}(\delta,\eps,\bm\alpha)+\mathfrak{R}_{-(1,1)}(\delta,\eps,\bm\alpha)\\
			&=\left(\mu_L^2+\mu_R^2-2\mu_L\mu_R\cos\left(\frac{\pi\gamma^2}{4}\right)-\mu\sin\left(\frac{\pi\gamma^2}{4}\right)\right)\frac{\Gamma\left(\frac{\gamma^2}{4}\right)\Gamma\left(1-\frac{\gamma^2}{2}\right)}{\Gamma\left(1-\frac{\gamma^2}{4}\right)}\ps{V_{\frac{3\gamma}2}(t)\V}+o(1).
		\end{align*}
	\end{proof}
	
	\subsubsection{The case $\chi=-\frac\gamma2$ with $\gamma>\sqrt2$}
	Let us now assume that $\gamma>\sqrt2$. In that case we prove the following:
	\begin{theorem}\label{thm:bpz_gamma2}
		Assume that $\gamma<\sqrt2$ and take weights $\bm\alpha$ in $\mc A_{N,M}$. Then, in the weak sense:
		\begin{equation}\label{eq:bpz_gamma2}
			\begin{split}
				&\left(\frac4{\gamma^2}\partial_t^2+\sum_{k=1}^{2N+M}\frac{\partial_{z_k}}{t-z_k}+\frac{\Delta_{\alpha_k}}{(t-z_k)^2}\right)\ps{V_{-\frac\gamma2}(t)\V}=0.
			\end{split}
		\end{equation}
	\end{theorem}
	\begin{proof} 
		Like for the $\gamma<\sqrt2$ case the leading order terms in the asymptotics of the remainder terms will compensate each other in the limit so that we only have to look at the remainder $\tilde{\mathfrak R}$. For this we simply apply item $(4)$ of Lemma~\ref{lemma:remainderL2} and item $(1)$ of Lemma~\ref{lemma:remainderL11_bis} to see that under our assumptions this remainder term is a $o(1)$.
	\end{proof}	
	
	\subsection{BPZ differential equations}
	Basic properties related to the conformal covariance of the model then allow to prove that specific correlation functions, containing one such degenerate field together with one boundary insertion and either one bulk or two boundary Vertex Operators, are solutions of a differential equation in one variable. Such differential equations are key in order to compute the structure constants of the theory.
	
	\subsubsection{Three-point boundary}
	Let us consider the four-point correlation function\\ $\ps{V_{-\chi}(t)\prod_{l=1}^3V_{\beta_l}(s_l)}$, with $s_1=0$, $s_2=1$, $s_3=+\infty$. Then using the global Ward identities from Theorem~\ref{thm:ward_global_set} we can express the $\L_{-1}$ descendants of the Vertex Operators with insertions at $s_1$, $s_2$ and $s_3$ in terms of the conformal weights and the descendant $\L_{-1}V_{-\chi}(t)$. Namely we have (when inserted within the correlation function considered here):
	\begin{align*}
		&\L_{-1}V_{\beta_3}(+\infty)\sim \frac{2\Delta_{\beta_3}}{+\infty};\quad 	\L_{-1}V_{\beta_2}(1)=\Delta_{\beta_3}-\Delta_{\beta_2}-\Delta_{\beta_1}-\Delta_{-\chi}-t\L_{-1}V_{-\chi}(t);\\
		& 	\L_{-1}V_{\beta_1}(0)=(t-1)\L_{-1}V_{-\chi}(t)+\Delta_{-\chi}+\Delta_{\beta_1}+\Delta_{\beta_2}-\Delta_{\beta_3}.
	\end{align*}
	We can then replace the descendants that appear in the corresponding higher equation of motion from Theorem~\ref{thm:bpz_set} for $\chi=\frac2\gamma$ and Theorems~\ref{thm:bpz_gamma1} or~\ref{thm:bpz_gamma2} when $\chi=\gamma$. If we further assume that the cosmological constants are chosen in such a way that the right-hand sides in the different higher equations of motion vanish (which is the case if we assume the assumptions in Corollary~\ref{cor:bpz} to be satisfied), this reasoning allows us to obtain a differential equation in the variable $t$. Elementary algebraic manipulations then show that this differential equation is an hypergeometric one:
	\begin{theorem}\label{thm:last}
		Assume that $\bm\alpha\in\mc A_{N,M}$, and choose the cosmological constants via
		\begin{equation}
			\mu_L=g(\sigma_l)\text{ and }\mu_R=g(\sigma_R),\quad g(\sigma)\coloneqq \frac{\cos\left(\pi\gamma(\sigma-\frac Q2)\right)}{\sqrt{\sin\left(\pi\frac{\gamma^2}4\right)}}
		\end{equation}
		where $\sigma_L-\sigma_R=\pm\frac\beta2$. Then for $\chi\in\{\frac\gamma2,\frac2\gamma\}$:
		\begin{equation}
			\left(t(t\partial_t+A)(t\partial_t+B)-(t\partial_t+C-1)t\partial_t\right)\ps{V_{-\chi}(t)\prod_{l=1}^3V_{\beta_l}(s_l)}=0
		\end{equation}
		where 
		\begin{equation}
			A=\chi(\beta_1+\beta_2+\beta_3-\chi-2Q),\quad B\coloneqq A+\chi(Q-\beta_1),\quad
			C=1+\chi(\beta_2-Q).
		\end{equation}
	\end{theorem}
	
	\subsubsection{Bulk-boundary correlator}
	We now turn to the three-point correlation function\\ $\ps{V_{-\chi}(t)V_\alpha(i)V_{\beta}(+\infty)}$. We can proceed like before based on the global Ward identities from Theorem~\ref{thm:ward_global_set} together with Theorem~\ref{thm:desc_two_set} to obtain a differential equation in the variable $t$:
	\begin{theorem}
		Under the assumptions of Theorem~\ref{thm:last}, for $\chi\in\{\frac\gamma2,\frac2\gamma\}$:
		\begin{equation}
			\begin{split}
				&\left(s(s\partial_s+A)(s\partial_s+B)-(s\partial_s+C-1)s\partial_s\right)\mc H(s)=0,\\
				&\text{where}\quad\mc H(s)\coloneqq\norm{t-i}^{-\chi\alpha}\ps{V_{-\chi}\left(t\right)V_\alpha(i)V_{\beta}(+\infty)},\quad t\coloneqq \sqrt{\frac{s}{1-s}}\cdot
			\end{split}
		\end{equation}
		Here 
		\begin{equation}
			A=\frac\chi2(2\alpha+\beta-\chi-2Q),\quad B\coloneqq A+\chi(Q-\alpha),\quad
			C=1+\frac\chi2(\beta-2Q).
		\end{equation}
	\end{theorem}

	%%%%%%%%%%%%%%%%%%%%%%%%%%%%%%%%%%%
	%%%%%%%%%%%%%%%%%%%%%%%%%%%%%%%%%%%
	
	%%%%%%%%%%%%%%%%%%%%%%%%%%%%%%%%%%%%%%%%%%%%%%%%
	%%%%%%%%%%%%%%%%%%%%%%%%%%%%%%%%%%%%%%%%%%%%%%%%
	
	%%%%%%%%%%%%%%%%%%%%%%%%%%%%%%%%%%%%%%%%%%%%%%%%
	%%%%%%%%%%%%%%%%%%%%%%%%%%%%%%%%%%%%%%%%%%%%%%%%
	
	%%%%%%%%%%%%%%%%%%%%%%%%%%%%%%%%%%%%%%%%%%%%%%%%
	%%%%%%%%%%%%%%%%%%%%%%%%%%%%%%%%%%%%%%%%%%%%%%%%
	
	%%%%%%%%%%%%%%%%%%%%%%%%%%%%%%%%%%%
	%%%%%%%%%%%%%%%%%%%%%%%%%%%%%%%%%%%
	%%%%%%%%%%%%%%%%%%%%%%%%%%%%%%%%%%%
	
	\appendix
	
	\section{Evaluation of some integrals}
	In this short section we evaluate some integrals involved in the computations of the remainder terms that show up in the definition of the descendant fields, and that show up within the proofs of Lemmas~\ref{lemma:remainderL2} and~\ref{lemma:remainderL11}. These exact computations rely on hypergeometric functions, that is series of the form
	\[
	\hyper{a,b}{c}{z}\coloneqq\sum_{n\in\N}\frac{(a)_n(b)_n}{(c)_n}\frac{z^n}{n!}\quad\text{for }\norm{z}<1
	\]
	and suitable $a,b$ and $c$, were $(a)_n=a(a+1)\cdots(a+n-1)$ is the Pochhammer symbol. 
	\begin{lemma}~\label{lemma:eval_intR}
		Assume that $\beta>-\frac1\gamma$. Then
		\begin{equation}
			\int_0^{+\infty} \left(1+x^2\right)^{-1-\frac{\gamma\beta}2}=2^{-1-\gamma\beta}\sin\left(-\pi\frac{\gamma\beta}{2}\right)\frac{\Gamma\left(-\frac{\gamma\beta}{2}\right)\Gamma\left(1+\gamma\beta\right)}{\Gamma\left(1+\frac{\gamma\beta}{2}\right)}\cdot
		\end{equation}
	\end{lemma}
	\begin{proof}
		If we expand in powers of $x^2$ over $(0,1)$ and $(1,+\infty)$ we get for $u>-\frac12:$
		\begin{align*}
			&\int_0^{+\infty} \left(1+x^2\right)^{-1-u}=\sum_{n\in\N}\frac{(1+u)_n(-1)^n}{n!}\left(\frac1{2n+1}-\frac1{2n+1+2u}\right)\\
			&=\left[\hyper{1+u,\frac12}{\frac32}{-1}+\frac1{1+2u}\hyper{1+u,\frac12+u}{\frac32+u}{-1}\right]\\
			&=\frac2{1+2u}\left[\left(\frac12+u\right)\hyper{1+u,\frac12}{\frac32}{-1}+\frac12\hyper{1+u,\frac12+u}{\frac32+u}{-1}\right]\\
			&=\frac{2}{1+2u}\frac{\Gamma\left(\frac32+u\right)\Gamma\left(\frac32\right)}{\Gamma\left(1+u\right)}
		\end{align*}
		where we have used that for suitable $a,b$,
		\[
		b\hspace{0.1cm}\hyper{a+b,a}{a+1}{-1}+a\hspace{0.1cm}\hyper{a+b,b}{b+1}{-1}=\frac{\Gamma(a+1)\Gamma(b+1)}{\Gamma(a+b)}\cdot
		\]
		Now we can combine the following identities for the Gamma function:
		\begin{align*}
			\Gamma(1+z)=z\Gamma(z);\quad \Gamma(z)\Gamma(1-z)=\frac{\pi}{\sin\pi z};\quad \Gamma(z)\Gamma(z+\frac12)=2^{1-2z}\sqrt\pi\Gamma(2z)
		\end{align*}
		to rewrite the latter as desired:
		\[
		\frac{2}{1+2u}\frac{\Gamma\left(\frac32+u\right)\Gamma\left(\frac32\right)}{\Gamma\left(1+u\right)}=2^{-1-2u}\sin\left(-\pi u\right)\frac{\Gamma\left(-u\right)\Gamma\left(1+2u\right)}{\Gamma\left(1+u\right)}\cdot
		\]
	\end{proof}
	
	\begin{lemma}~\label{lemma:eval_int1}
		Assume that $\gamma<\sqrt2$ and take $\beta>-\gamma$. Then
		\begin{equation}
			\begin{split}
				&\int_{1}^{+\infty}x^{-1-\frac{\gamma\beta}{2}}(x-1)^{-\frac{\gamma^2}{2}}dx=\frac{\Gamma\left(\frac\gamma2(\beta+\gamma)\right)\Gamma\left(1-\frac{\gamma^2}{2}\right)}{\Gamma\left(1+\frac{\gamma\beta}2\right)}\quad\text{and}\\
				&\int_{1}^{+\infty}x^{-1-\frac{\gamma\beta}{2}}(x+1)^{-\frac{\gamma^2}{2}}dx=\cos\left(\pi\frac{\gamma}{2}\sigma(\beta)\right)\frac{\Gamma\left(\frac{\gamma^2}{4}\right)\Gamma\left(1-\frac{\gamma^2}{2}\right)}{\Gamma\left(1-\frac{\gamma^2}{4}\right)}
			\end{split}
		\end{equation}
		for some $\sigma(\beta)$ with $\sigma(-\frac\gamma2)=-\frac\gamma2$.
	\end{lemma}
	\begin{proof}
		Let us make the change of variable $x\leftrightarrow\frac1x$ in the integrals. We thus obtain
		\[
		\int_{0}^{1}x^{-1+\frac{\gamma}{2}(\beta+\gamma)}(1\pm x)^{-\frac{\gamma^2}{2}}dx.
		\]
		If we denote by $\mc F$ the hypergeometric function $\hyper{\frac{\gamma^2}2,\frac\gamma2(\beta+\gamma)}{1+\frac\gamma2(\beta+\gamma)}{\cdot}$ then the above are found to be given by $\frac{2}{\gamma(\beta+\gamma)}\mc F(\mp 1)$. We can evaluate $\mc F(1)$ via Gauss'summation theorem, which states that $\mc F(1)=\frac{\Gamma(1+\frac\gamma2(\beta+\gamma))\Gamma(1-\frac{\gamma^2}2)}{\Gamma(1+\frac\gamma2\beta)}$. This proves the first equality.
		As for the second one first note that since $(1+x)^{-\frac{\gamma^2}{2}}\leq (1-x)^{-\frac{\gamma^2}{2}}$ we have $0<\mc F(-1)<\mc F(1)$ showing the existence of $\sigma.$ In the special case where $\beta=-\frac\gamma2$ we can use Kummer's theorem which shows that $\mc F(-1)=\frac{\Gamma(1+\frac{\gamma^2}4)\Gamma(1+\frac{\gamma^2}4)}{\Gamma(1+\frac{\gamma^2}2)}$.  Therefore the second integral is found to be given by $\frac{\Gamma(\frac{\gamma^2}4)\Gamma(1-\frac{\gamma^2}2)}{\Gamma(1-\frac{\gamma^2}4)}\times\frac{\Gamma(1-\frac{\gamma^2}4)\Gamma(1+\frac{\gamma^2}4)}{\Gamma(1-\frac{\gamma^2}2)\Gamma(1+\frac{\gamma^2}2)}$.
		We conclude thanks to the identity
		\begin{align*}
			\cos\pi z=\frac{\pi}{\Gamma(\frac12+z)\Gamma(\frac12-z)}=\frac{\Gamma(1+z)\Gamma(1-z)}{\Gamma(1+2z)\Gamma(1-2z)}\cdot
		\end{align*}
	\end{proof}
	
	%\newpage
	\bibliography{biblio}
	\bibliographystyle{plain}
	
\end{document}